\documentclass[oneside,leqno,a4paper]{amsart}

\usepackage[a4paper,margin=3cm]{geometry}
\usepackage[utf8]{inputenc}
\usepackage[T1]{fontenc}
\usepackage{lmodern}
\usepackage[english]{babel}
\usepackage{microtype}

\usepackage{amsmath}
\usepackage{amssymb}
\usepackage{amsthm}
\usepackage{mathrsfs}

\usepackage{enumitem}

\usepackage[
	colorlinks=true,
	linkcolor=blue,
	citecolor=blue,
	urlcolor=blue,
	backref=page
]{hyperref}

\usepackage[textsize=tiny,color=green!40]{todonotes}

\newcommand{\R}{\mathbb{R}}
\newcommand{\N}{\mathbb{N}}
\newcommand{\Prob}{\mathbb{P}}
\newcommand{\1}{\mathbf{1}}
\newcommand{\cF}{\mathcal{F}}
\newcommand{\cB}{\mathcal{B}}
\newcommand{\cD}{\mathcal{D}}
\newcommand{\cM}{\mathcal{M}}

\newcommand{\abs}[1]{\left\lvert#1\right\rvert}
\newcommand{\norm}[1]{\left\lVert#1\right\rVert}

\newcommand{\loc}{\mathrm{loc}}
\DeclareMathOperator{\E}{\mathbb{E}}
\DeclareMathOperator{\supp}{supp}
\DeclareMathOperator*{\esslim}{ess\,lim}
\DeclareMathOperator{\diver}{div}

\numberwithin{equation}{section}

\theoremstyle{plain}
\newtheorem{theorem}{Theorem}[section]
\newtheorem{proposition}[theorem]{Proposition}
\newtheorem{lemma}[theorem]{Lemma}
\newtheorem{corollary}[theorem]{Corollary}

\theoremstyle{definition}
\newtheorem{definition}[theorem]{Definition}
\newtheorem{assumption}[theorem]{Assumption}

\theoremstyle{remark}
\newtheorem{remark}[theorem]{Remark}

\begin{document}

\title[Strong initial traces]
{Existence of strong initial traces for
\\ stochastic conservation laws}

\author[M. Erceg]{M. Erceg}
\address[M. Erceg]
{\newline University of Zagreb
\newline Zagreb, Croatia}
\email[]{marko.erceg@math.hr}

\author[K. H. Karlsen]{K. H. Karlsen}
\address[Kenneth H. Karlsen]
{\newline Department of Mathematics
\newline University of Oslo
\newline Oslo, Norway}
\email[]{kennethk@math.uio.no}

\author[N. Konatar]{N. Konatar}
\address[N. Konatar]
{\newline Faculty of Mathematics and Natural Sciences
\newline University of Montenegro
\newline Cetinjski put bb, 81000 Podgorica, Montenegro}
\email[]{konatarn@yahoo.com}

\author[D. Mitrovi{\'c}]{D. Mitrovi{\'c}}
\address[Darko Mitrovi{\'c}]
{\newline Faculty of Mathematics and Natural Sciences
\newline University of Montenegro
\newline Cetinjski put bb, 81000 Podgorica, Montenegro}
\email[]{darkom@ucg.ac.me}



\raggedbottom
\allowdisplaybreaks

\date{\today}

\begin{abstract}
We prove existence and uniqueness of a strong initial trace for every bounded
kinetic solution of a stochastic scalar conservation law, although no initial
value is prescribed and no nondegeneracy condition is imposed on the flux.
The core of the proof is pathwise.  After fixing a realization,
the rescaled stochastic terms and kinetic measure vanish in the blow-up
limit, so Panov's compactness argument applies; degenerate flux intervals are
handled by recursive dimension reduction.  The stochastic setting creates 
several additional difficulties.  The martingale identities must remain valid on one
common full-probability set throughout the reductions, which requires a
parameterized stochastic-Fubini construction.  Moreover, the pathwise trace
is not automatically measurable because its exceptional sets may depend on
the realization.  Deterministic time averages and right-continuity of the
filtration yield a jointly measurable initial trace, with local strong
convergence along essential times, both almost surely and in mean.
\end{abstract}

\maketitle

\setcounter{tocdepth}{1}

% Table of contents without page numbers.
% This keeps the section titles and their hyperlinks.
{\small
\begingroup
\makeatletter
\let\oldcontentsline\contentsline
\renewcommand{\contentsline}[4]{%
	\oldcontentsline{#1}{#2}{}{#4}%
}
\tableofcontents
\makeatother
\endgroup}

%%%%%%%%%%%%%%%%%%%%%%%%%
%%%%%%%%%%%%%%%%%%%%%%%%%
\section{Introduction}

We consider bounded kinetic solutions, defined only for positive times, of
\begin{equation}\label{eq:spde}
	du+\diver_x f(u)\,dt
	=
	\sum_{\ell\geq1} b_\ell(t,x,u)\,dW_\ell(t),
	\qquad
	(t,x)\in(0,T)\times\R^d.
\end{equation}
No initial value is included in the definition of solution. We show that the
equation nevertheless determines a unique strong trace at the initial
hyperplane, without any flux nondegeneracy assumption.

More precisely, we construct an
$\cF_0\otimes\cB(\R^d)$-measurable function
$u_0:\Omega\times\R^d\to[-M,M]$ such that, for every compact
$K\Subset\R^d$,
\begin{equation}\label{eq:intro-pathwise-trace}
	\esslim_{t\downarrow0}
	\int_K
	\abs{u(\omega,t,x)-u_0(\omega,x)}\,dx
	=
	0
\end{equation}
for almost every $\omega\in\Omega$. As a consequence,
\begin{equation}\label{eq:intro-expected-trace}
	\esslim_{t\downarrow0}
	\E\int_K
	\abs{u(t,x)-u_0(x)}\,dx
	=
	0.
\end{equation}
Thus $u(t)$ converges strongly to $u_0$ in
$L^1(\Omega\times K)$, in the essential-time sense. Throughout the paper,
$\esslim_{t\downarrow0}F(t)=0$ means that
$$
\lim_{\delta\downarrow0}
\operatorname*{ess\,sup}_{0<t<\delta}\abs{F(t)}
=0.
$$

Vasseur and Panov developed the deterministic theory of strong traces for
multidimensional scalar conservation laws
\cite{Vasseur:2001lr,Panov:2005lr,Panov:2007aa}. Our argument uses Panov's
two-parameter H-measure localization method within Tartar's microlocal
compactness framework \cite{Panov:1994gf,Tartar:1990mq}.
{The same approach was previously used to show existence of strong traces 
for degenerate parabolic second-order partial differential problems
\cite{Erceg:2022aa}.} 

The deterministic {($b_\ell=0$, $\ell\in\N$)} 
kinetic formulation of {\eqref{eq:spde}} was introduced by
Lions--Perthame--Tadmor \cite{Lions:1994qy}. Early stochastic
theories were developed by Kim \cite{Kim:2003mz}, Feng--Nualart
\cite{Feng:2008ul}, and Vallet--Wittbold \cite{Vallet:2009uq}. On the torus,
Debussche--Vovelle \cite{Debussche:2010fk}
introduced the stochastic kinetic-solution
framework for multiplicative noise, and Dotti--Vovelle
\cite{Dotti:2018aa} later established convergence of approximations.
The work \cite{Frid:2021us} developed a
corresponding kinetic well-posedness theory on the unbounded domain $\R^d$.
These stochastic theories treat Cauchy or initial-boundary value problems
with a prescribed initial state.  In contrast, the bounded kinetic solution
of \eqref{eq:spde} in Definition~\ref{def:kinetic-solution} is defined only
on $(0,T)\times\R^d$, and no datum at $t=0$ is assumed.
The works of Feng--Nualart and Debussche--Vovelle were followed by an
extensive literature that cannot be surveyed here; the references above are
limited to developments most directly relevant to the present problem.

Strong spatial boundary traces for stochastic conservation laws on bounded
domains were established in \cite{Frid:2022trace} and used there to prove
well-posedness of a Neumann initial-boundary value problem. Their trace
theorem assumes a $C^3$ flux satisfying a nondegeneracy condition and
identifies the trace under deformations of the spatial boundary. In
contrast, we recover the temporal trace at $t=0$ of a bounded kinetic
solution posed only on $(0,T)\times\R^d$, with no prescribed initial value
and no flux nondegeneracy assumption.

Fixing a realization does not immediately reduce the stochastic equation to
a deterministic quasi-solution: its kinetic formulation still contains the
distributional time derivative of a martingale.  We show that, under the
space--time rescaling near the initial hyperplane, this martingale
contribution and the It\^o correction vanish in the limit.  The limiting
kinetic equation is therefore deterministic, and Panov's compactness argument
can be applied pathwise.

Panov's H-measure localization method
\cite{Panov:1994gf,Panov:2005lr,Panov:2007aa} reduces a possible loss of
compactness to an interval on which a directional component of the
space--time flux $\Phi(r):=(r,f(r))$ is constant.  On such an interval, a
linear change of variables isolates a spatial coordinate that drops out of
the transformed transport operator.  Fixing this transverse coordinate
produces a kinetic equation in one fewer spatial dimension, which can be
treated inductively.  The stochastic ingredients are the pathwise
construction of the weak trace, simultaneous blow-up estimates for the
kinetic measure and martingale terms, and the treatment of the stochastic
integral under dimension reduction.  This integral also depends on the
transverse coordinate.  An $L^2$-valued stochastic Fubini argument shows
that, for almost every fixed value of that coordinate, the corresponding
section is the It\^o integral in the lower-dimensional kinetic equation.
The kinetic measure itself is sliced pathwise by deterministic measure
theory; its sections need no stochastic measurability.

Compactness and trace identification are pathwise.  The first step produces
a weak kinetic trace $h_0^\omega$; the blow-up and dimension-reduction
argument identifies it as the equilibrium associated with a strong trace
$u_0^\omega$.  Subsequences, exceptional spatial points, admissible
transverse coordinates, and section measures may depend on $\omega$.
Consequently, this construction does not imply measurability of
$\omega\mapsto u_0^\omega$.

We avoid measurable selection by using the deterministic time averages
$u_n:=n\int_0^{1/n}u(t)\,dt$.  They are
$\cF_{1/n}$-measurable as $L^1_{\loc}(\R^d)$-valued maps and converge almost
surely to $u_0^\omega$.  Tail measurability (for every $m$, the limit of
$(u_n)_{n\geq m}$ is $\cF_{1/m}$-measurable) and right-continuity of the
filtration place the limit in $\cF_0$, while the Bochner--Fubini
identification supplies a jointly measurable representative.  The uniform
bound then gives \eqref{eq:intro-expected-trace} by dominated convergence.

The remaining part of this paper is organized as follows: 
Section~\ref{sec:stochastic-kinetic} states the stochastic and kinetic
framework and the main theorem.  Section~\ref{sec:weak-trace} constructs the
pathwise weak trace, and Section~\ref{sec:panov-compactness} records the
required part of Panov's compactness theory. Section~\ref{sec:pathwise-blowup}
establishes the blow-up estimates and compactness condition.
Sections~\ref{sec:pathwise-dimension-reduction}
and~\ref{sec:pathwise-identification} treat degenerate intervals by slicing and
induction. Section~\ref{sec:final-measurability} resolves the measurability
problem by deterministic time averages and proves convergence in expectation
and uniqueness.

%%%%%%%%%%%%%%%%%%%%%%%%%
%%%%%%%%%%%%%%%%%%%%%%%%%
\section{Stochastic and kinetic framework}
\label{sec:stochastic-kinetic}

\subsection{The stochastic basis and coefficients}

Let
$\mathcal S=\bigl(\Omega,\cF,(\cF_t)_{t\geq0},\Prob\bigr)$
be a complete stochastic basis whose filtration is complete and
right-continuous. Let $(W_\ell)_{\ell\geq1}$ be a sequence of independent
real-valued standard $(\cF_t)$-Wiener processes.
Thus each $W_\ell$ is adapted and, for $0\leq s<t$, the increment
$W_\ell(t)-W_\ell(s)$ is independent of $\cF_s$ and has law
$\mathcal N(0,t-s)$. We denote by $\mathcal P$ the predictable sigma-field
on $\Omega\times(0,T)$; see \cite{DaPrato:2014aa} for the standard
stochastic framework.

Set
$$
a(\lambda):=f'(\lambda),
\qquad
G^2(t,x,\lambda)
:=
\sum_{\ell\geq1}
\abs{b_\ell(t,x,\lambda)}^2.
$$

We impose the following assumptions on the flux and the noise coefficients.

\begin{assumption}\label{ass:coeff}
The flux satisfies \(f\in C^1(\R;\R^d)\). For every \(\ell\geq1\), the
coefficient
\[
b_\ell:[0,T]\times\R^d\times\R\longrightarrow\R
\]
is jointly Borel measurable. For every \(t\in[0,T]\), the map
\[
(x,\lambda)\longmapsto b_\ell(t,x,\lambda)
\]
is of class \(C^1\), and the maps
\[
(t,x,\lambda)\longmapsto
\partial_\lambda b_\ell(t,x,\lambda),
\qquad
(t,x,\lambda)\longmapsto
\nabla_x b_\ell(t,x,\lambda)
\]
are jointly Borel measurable. Moreover,
for every compact set $K\Subset\R^d$ and every $R>0$,
\begin{equation}\label{eq:noise-local-bound}
	\sup_{(t,x,\lambda)\in[0,T]\times K\times[-R,R]} \sum_{\ell\geq1}
		\left(
	\abs{b_\ell(t,x,\lambda)}^2
	+
	\abs{\partial_\lambda b_\ell(t,x,\lambda)}^2
	+
	\abs{\nabla_x b_\ell(t,x,\lambda)}^2
	\right)
	<\infty.
\end{equation}
\end{assumption}

The zeroth-order summand in \eqref{eq:noise-local-bound} ensures that the
stochastic integral in \eqref{eq:kinetic-weak} and the processes in
\eqref{eq:scalar-martingale-pairing} and
\eqref{eq:spatial-martingale-process} are locally well defined in $L^2$.
{The derivative summands ensure
that the coefficients transformed in
Lemma~\ref{lem:pathwise-linear-change} again satisfy
\eqref{eq:noise-local-bound}.}  All later uses of these derivatives require
only the almost-everywhere chain rule and the essential-supremum bounds in
\eqref{eq:noise-local-bound}; hence the arguments remain valid under the
corresponding local Lipschitz hypotheses.

\subsection{Kinetic measures up to the initial boundary}

Fix $M>0$ and choose $L>M+1$.  Throughout, every admissible kinetic test
function has $\lambda$-support compactly contained in $(-L,L)$.

\begin{definition}[Local kinetic measure up to $t=0$]
\label{def:kinetic-measure}
A random nonnegative Radon measure $m$ on
$(0,T)\times\R^d\times\R$ is called a local kinetic measure up to $t=0$
if the following properties hold:
\begin{enumerate}[label=\textup{(\roman*)}]
	\item for every
	$\psi\in C_c\bigl((0,T)\times\R^d\times\R\bigr)$, the map
	$$
	\omega
	\longmapsto
	\int_{(0,T)\times\R^d\times\R}
	\psi(t,x,\lambda)\,dm(\omega;t,x,\lambda)
	$$
	is measurable;

	\item for every compact set $K\Subset\R^d$ and every $R>0$,
	\begin{equation}\label{eq:expected-measure-bound}
		\E
		m\bigl(
		(0,T)\times K\times[-R,R]
		\bigr)
		<\infty;
	\end{equation}

	\item for every
	$\psi\in C_c(\R^d\times\R)$, the process
	$$
	t\longmapsto
	\int_{(0,t]\times\R^d\times\R}
	\psi(x,\lambda)\,dm(s,x,\lambda)
	$$
	admits a predictable version.
\end{enumerate}
As part of the bounded-solution framework, we assume that
$\supp m\subset(0,T)\times\R^d\times[-M,M]$ almost surely.
\end{definition}

To make the realization in the pathwise arguments explicit, let
$K_j:=\{x\in\R^d:\abs{x}\leq j\}$ be the closed ball of radius $j$ centered
at the origin.  Let $\Omega_{\mathrm{supp}}$ be a full-probability set on
which the support condition in
Definition~\ref{def:kinetic-measure} holds, and define
\begin{equation}\label{eq:local-mass-event}
	\Omega_m
	:=
	\Omega_{\mathrm{supp}}
	\cap
	\bigcap_{j=1}^{\infty}
	\left\{
	\omega\in\Omega:
	m\bigl(
	\omega;
	(0,T)\times K_j\times[-j,j]
	\bigr)<\infty
	\right\}.
\end{equation}
By \eqref{eq:expected-measure-bound} and countable intersection,
$\Prob(\Omega_m)=1$.  If $\omega\in\Omega_m$, $K\Subset\R^d$, and $R>0$,
choose $j$ such that $K\subset K_j$ and $R\leq j$.  Then
$$
m\bigl(
\omega;
(0,T)\times K\times[-R,R]
\bigr)
\leq
	m\bigl(
	\omega;
	(0,T)\times K_j\times[-j,j]
	\bigr)
	<\infty.
$$
Thus $\Omega_m$ is the single full-probability set on which all local
finiteness assertions up to $t=0$ are used.

We now specify the class of solutions to which the trace theorem applies.

\begin{definition}[Kinetic solution without prescribed initial data]
\label{def:kinetic-solution}
A function $u:\Omega\times(0,T)\times\R^d\to\R$ is called a bounded
kinetic solution of \eqref{eq:spde} on $(0,T)\times\R^d$ if:

\begin{enumerate}[label=\textup{(\roman*)}]
	\item $u$ is $\mathcal P\otimes\cB(\R^d)$-measurable;

	\item $\abs{u(\omega,t,x)}\leq M$ 
	for almost every $(\omega,t,x)$;

	\item there exists a local kinetic measure $m$
	in the sense of Definition~\ref{def:kinetic-measure}
	such that, with
	$h(\omega,t,x,\lambda):=\1_{\{\lambda<u(\omega,t,x)\}}$,
	the identity
	\begin{equation}\label{eq:kinetic-weak}
	\begin{aligned}
		&\int_0^T
		\int_{\R^d}
		\int_{\R}
		h(\omega,t,x,\lambda)
		\bigl(
		\partial_t\varphi(t,x,\lambda)
		+
		a(\lambda)\cdot\nabla_x\varphi(t,x,\lambda)
		\bigr)
		\,d\lambda\,dx\,dt\\
		&\quad=
		-\sum_{\ell\geq1}
		\int_0^T
		\left(
		\int_{\R^d}
		b_\ell(t,x,u(\omega,t,x))
		\varphi(t,x,u(\omega,t,x))
		\,dx
		\right)
		dW_\ell(t)\\
		&\qquad
		-\frac12
		\int_0^T
		\int_{\R^d}
		G^2(t,x,u(\omega,t,x))
		\partial_\lambda\varphi(t,x,u(\omega,t,x))
		\,dx\,dt\\
		&\qquad
		+
		\int_{(0,T)\times\R^d\times\R}
		\partial_\lambda\varphi(t,x,\lambda)
		\,dm(\omega;t,x,\lambda)
	\end{aligned}
	\end{equation}	holds almost surely for every
	$\varphi
	\in
	C_c^1\bigl(
	(0,T)\times\R^d\times(-L,L)
	\bigr)$.
\end{enumerate}
\end{definition}

For $a<b$, let
$T_{a,b}(r):=\max\{a,\min\{b,r\}\}$ denote truncation to $[a,b]$.
Since $u=T_{-M,M}(u)$ almost everywhere, we henceforth replace $u$ by
$T_{-M,M}(u)$ and retain the notation $u$.  This predictable representative
satisfies $\abs{u(\omega,t,x)}\leq M$ for every $(\omega,t,x)$ and changes
neither the equivalence class of $u$ {(with respect to the space $L^\infty(\Omega\times(0,T)\times\R^d)$)} nor the kinetic identity
\eqref{eq:kinetic-weak}.

Definition~\ref{def:kinetic-solution} follows the periodic stochastic
kinetic framework of Debussche--Vovelle \cite{Debussche:2010fk} and its
$\R^d$ counterpart in \cite{Frid:2021us}.
It records entropy dissipation through the kinetic measure and the effect
of the stochastic forcing through the martingale term and the It\^o
correction, and is stable under approximation \cite{Dotti:2018aa}.
Unlike the standard Cauchy formulation, our test functions are supported
away from $t=0$ and no initial kinetic datum is prescribed.

Assumption~\ref{ass:coeff} ensures that the series of stochastic integrals in
\eqref{eq:kinetic-weak} converges in $L^2(\Omega)$ for every admissible test
function.

With the convention $h=\1_{\{\lambda<u\}}$, the corresponding formal
kinetic equation is
\begin{equation}\label{eq:kinetic-formal}
	\bigl(
	\partial_t+a(\lambda)\cdot\nabla_x
	\bigr)h
	=
	\sum_{\ell\geq1}
	b_\ell(t,x,\lambda)
	\delta_{\lambda=u}\,\dot W_\ell
	+
	\partial_\lambda
	\left(
	m
	-
	\frac12
	G^2(t,x,\lambda)\delta_{\lambda=u}
	\right).
\end{equation}

The main result recovers the missing value at $t=0$ from
\eqref{eq:kinetic-weak} alone.

\begin{theorem}[Strong initial trace]\label{thm:main}
Let Assumption~\ref{ass:coeff} hold, and let $u$ be a bounded kinetic
solution in the sense of Definition~\ref{def:kinetic-solution}.  
Then there exists a unique
function
$$
u_0
\in
L^\infty\bigl(
\Omega\times\R^d,
\cF_0\otimes\cB(\R^d),
\Prob\otimes\mathcal L^d
\bigr),
\qquad
\abs{u_0}\leq M,
$$
and a set
$$
\Omega_0\subset\Omega,
\qquad
\Prob(\Omega_0)=1,
$$
such that, for every $\omega\in\Omega_0$ and every compact set
$K\Subset\R^d$,
\begin{equation*}
	\esslim_{t\downarrow0}
	\int_K
	\abs{u(\omega,t,x)-u_0(\omega,x)}
	\,dx=0.
\end{equation*}
Here $\cF_0=\bigcap_{t>0}\cF_t$ is the initial sigma-algebra of the
right-continuous filtration $(\cF_t)_{t\geq0}${, and 
$\mathcal L^d$ represents the Lebesgue measure on $\R^d$.}

Consequently, for every compact $K\Subset\R^d$,
\begin{equation*}
	\esslim_{t\downarrow0}
	\E\int_K
	\abs{u(t,x)-u_0(x)}
	\,dx
	=
	0.
\end{equation*}
Thus $u(t)\to u_0$ strongly in $L^1(\Omega\times K)$, in the
essential-time sense. 
\end{theorem}

%%%%%%%%%%%%%%%%%%%%%%%%%
%%%%%%%%%%%%%%%%%%%%%%%%%
\section{The weak kinetic trace}
\label{sec:weak-trace}

We first construct the weak trace pathwise, without asserting measurability
in $\omega$. Joint measurability and convergence in expectation are recovered
in Section~\ref{sec:final-measurability}, after the weak trace has been
identified as an equilibrium.

For $\phi\in C_c^1(\R^d\times(-L,L))$, define
$$
\mathscr H_\phi(t,\omega)
:=
\int_{\R^d}\int_{\R}
h(\omega,t,x,\lambda)\phi(x,\lambda)
\,d\lambda\,dx
$$
and
\begin{equation}\label{eq:scalar-martingale-pairing}
M_\phi(t,\omega)
:=
\sum_{\ell\geq1}
\int_0^t
\left(
\int_{\R^d}
b_\ell(s,x,u(\omega,s,x))
\phi(x,u(\omega,s,x))
\,dx
\right)
dW_\ell(s).
\end{equation}
By the local square-summability assumption
\eqref{eq:noise-local-bound}, $M_\phi$ admits a continuous version
starting from zero. Throughout, symbols of the form $M_\phi$ denote
scalar martingales, whereas $\mathscr M$ denotes spatially dependent
$L^2_{\loc}$-valued martingales.

The following time-regularization of \eqref{eq:kinetic-weak} is analogous
to the time-integrated kinetic identities used by
Debussche--Vovelle \cite{Debussche:2010fk} and
Dotti--Vovelle \cite{Dotti:2018aa}.

\begin{lemma}[Time-integrated kinetic identity]
\label{lem:time-integrated}
For every
\[
\phi\in C_c^1(\R^d\times(-L,L)),
\]
there exists a set
\[
\Omega_\phi\subset\Omega_m,
\qquad
\Prob(\Omega_\phi)=1,
\]
such that, for every $\omega\in\Omega_\phi$,
$\mathscr H_\phi(\cdot,\omega)$ has a c\`adl\`ag representative on
$(0,T)$. For every $0<s<t<T$, this representative satisfies
\begin{equation}\label{eq:time-integrated}
\begin{aligned}
	\mathscr H_\phi(t,\omega)-\mathscr H_\phi(s,\omega)
	&=
	\int_s^t
	\int_{\R^d}\int_{\R}
	h(\omega,r,x,\lambda)
	a(\lambda)\cdot\nabla_x\phi(x,\lambda)
	\,d\lambda\,dx\,dr
	\\
	&\quad+
	M_\phi(t,\omega)-M_\phi(s,\omega)
	\\
	&\quad+
	\frac12
	\int_s^t
	\int_{\R^d}
	G^2(r,x,u(\omega,r,x))
	\partial_\lambda\phi(x,u(\omega,r,x))
	\,dx\,dr
	\\
	&\quad-
	\int_{(s,t]\times\R^d\times\R}
	\partial_\lambda\phi(x,\lambda)
	\,dm(\omega;r,x,\lambda).
\end{aligned}
\end{equation}
Moreover, this representative has a finite right limit at $t=0$.
Consequently,
\[
\esslim_{t\downarrow0}
\mathscr H_\phi(t,\omega)
\]
exists and is finite.
\end{lemma}

\begin{proof}
Fix $\phi\in C_c^1(\R^d\times(-L,L))$. Choose
$\Omega_\phi\subset\Omega_m$ of probability one such that the chosen
continuous version of $M_\phi$ is defined and
\eqref{eq:kinetic-weak} holds for
\[
\varphi(t,x,\lambda)=\theta(t)\phi(x,\lambda)
\]
simultaneously for $\theta$ in a fixed countable determining family in
$C_c^1((0,T))$. Since
\[
\int_0^T\theta(t)\,dM_\phi(t)
=
-\int_0^T\theta'(t)M_\phi(t)\,dt,
\]
and all the resulting terms depend continuously on $\theta$, the identity
extends to every $\theta\in C_c^1((0,T))$.

Fix $\omega\in\Omega_\phi$. Define
\[
\begin{aligned}
A_\phi(r,\omega)
&:=
\int_{\R^d}\int_{\R}
h(\omega,r,x,\lambda)
a(\lambda)\cdot\nabla_x\phi(x,\lambda)
\,d\lambda\,dx
\\
&\quad+
\frac12
\int_{\R^d}
G^2(r,x,u(\omega,r,x))
\partial_\lambda\phi(x,u(\omega,r,x))
\,dx
\end{aligned}
\]
and define the finite signed measure $\nu_\phi^\omega$ on $(0,T)$ by
\[
\nu_\phi^\omega(B)
:=
\int_{B\times\R^d\times\R}
\partial_\lambda\phi(x,\lambda)
\,dm(\omega;r,x,\lambda)
\]
for every Borel set $B\subset(0,T)$. Its cumulative function
\[
V_\phi(t,\omega)
:=
\nu_\phi^\omega((0,t])
=
\int_{(0,t]\times\R^d\times\R}
\partial_\lambda\phi(x,\lambda)
\,dm(\omega;r,x,\lambda)
\]
is c\`adl\`ag and has bounded variation.

The identity obtained above can be written in the sense of distributions
on $(0,T)$ as
\[
D_t\bigl(\mathscr H_\phi-M_\phi\bigr)
=
A_\phi(t,\omega)\,dt-\nu_\phi^\omega.
\]
Consequently, there exists a constant $c_\phi(\omega)$ such that
\[
\widetilde{\mathscr H}_\phi(t,\omega)
:=
c_\phi(\omega)
+
M_\phi(t,\omega)
+
\int_0^t A_\phi(r,\omega)\,dr
-
V_\phi(t,\omega)
\]
is a representative of $\mathscr H_\phi(\cdot,\omega)$. Since
$M_\phi$ is continuous, $A_\phi(\cdot,\omega)\in L^1(0,T)$, and
$V_\phi$ is c\`adl\`ag, this representative is c\`adl\`ag.

Moreover, for every $0<s<t<T$,
\[
V_\phi(t,\omega)-V_\phi(s,\omega)
=
\int_{(s,t]\times\R^d\times\R}
\partial_\lambda\phi(x,\lambda)
\,dm(\omega;r,x,\lambda).
\]
Therefore, the preceding representation gives exactly
\eqref{eq:time-integrated}. We henceforth denote this representative again
by $\mathscr H_\phi$.

Finally, since $\omega\in\Omega_m$,
\[
m\bigl(
\omega;
(0,T)\times\supp\phi
\bigr)
<\infty.
\]
It follows that
\[
V_\phi(t,\omega)\longrightarrow0
\qquad\text{as }t\downarrow0.
\]
Furthermore,
\[
M_\phi(t,\omega)\longrightarrow0
\]
by continuity at zero, while
\[
\int_0^t A_\phi(r,\omega)\,dr\longrightarrow0
\]
because $A_\phi(\cdot,\omega)\in L^1(0,T)$. Hence
\[
\lim_{t\downarrow0}
\mathscr H_\phi(t,\omega)
=
c_\phi(\omega),
\]
which proves the asserted finite right limit and therefore the essential
limit.
\end{proof}

We next construct the initial weak trace from the scalar limits supplied by
Lemma~\ref{lem:time-integrated}.
{The construction uses no measurable selection of $h_0^\omega$ and no
	weak-$*$ compactness on the product set 
	$\Omega\times\R^d\times(-L,L)$.}
    
\begin{lemma}[Pathwise weak kinetic trace]
\label{lem:weak-trace}
Recall that $M$ is the uniform bound for $u$ and that $L>M+1$ is the fixed
kinetic cutoff chosen before Definition~\ref{def:kinetic-measure}.
There exists a set
$$
\Omega_{\mathrm w}\subset\Omega,
\qquad
\Prob(\Omega_{\mathrm w})=1,
$$
such that, for every $\omega\in\Omega_{\mathrm w}$, there is a unique
function
$$
h_0^\omega\in
L^\infty(\R^d\times\R),
\qquad
0\leq h_0^\omega\leq1,
$$
with
$$
h_0^\omega(x,\lambda)=1
\quad\text{for }\lambda\leq-L,
\qquad
h_0^\omega(x,\lambda)=0
\quad\text{for }\lambda\geq L,
$$
and such that, for every
$$
\varphi\in C_c^1(\R^d\times\R),
$$
\begin{equation}\label{eq:pathwise-weak-trace}
	\esslim_{t\downarrow0}
	\int_{\R^d}\int_{\R}
	\bigl(
	h(\omega,t,x,\lambda)-h_0^\omega(x,\lambda)
	\bigr)
	\varphi(x,\lambda)
	\,d\lambda\,dx
	=
	0.
\end{equation}
No measurability of the map
$\omega\mapsto h_0^\omega$ is asserted or used.
\end{lemma}

\begin{proof}
Choose an exhaustion
$$
U_1\Subset U_2\Subset\cdots\Subset\R^d,
\qquad
\bigcup_{j\geq1}U_j=\R^d,
$$
and, for every $j$, choose a countable vector space
$$
\cD_j
\subset
C_c^1(U_j\times(-L,L))
$$
over $\mathbb Q$, dense in
$
L^1(U_j\times(-L,L)).
$
With the events $\Omega_\varphi$ from
Lemma~\ref{lem:time-integrated}, define
\begin{equation}\label{eq:weak-trace-event}
	\Omega_{\mathrm w}
	:=\Omega_m
	\cap
	\bigcap_{\varphi\in\bigcup_{j\geq1}\cD_j}
	\Omega_\varphi .
\end{equation}
The determining family is countable, so
$\Prob(\Omega_{\mathrm w})=1$.  

For every
$\omega\in\Omega_{\mathrm w}$, all scalar pairings
$$
\mathscr H_\varphi(t,\omega)
=
\int_{\R^d}\int_{\R}
h(\omega,t,x,\lambda)\varphi(x,\lambda)
\,d\lambda\,dx,
\qquad
\varphi\in\bigcup_{j\geq1}\cD_j,
$$
have finite essential limits at $t=0$.  The time-integrated identity also
holds on this set for the same determining family; that is,
\eqref{eq:time-integrated} holds on $\Omega_{\mathrm w}$ 
{simultaneously for
every $\varphi\in\bigcup_{j\geq1}\cD_j$.}

Fix $\omega\in\Omega_{\mathrm w}$ and $j\geq1$, and write
$Q_j:=U_j\times(-L,L)$. Choose a full-measure set
$E_\omega\subset(0,T)$ such that $h(\omega,t,\cdot,\cdot)$ is defined
for $t\in E_\omega$ and, for
every $\varphi\in\bigcup_{j\geq1}\cD_j$,
$\mathscr H_\varphi(t,\omega)$ equals the c\`adl\`ag representative supplied by
Lemma~\ref{lem:time-integrated}.

Let $t_n\in E_\omega$ with $t_n\downarrow0$. Since
$0\leq h(\omega,t_n,x,\lambda)\leq1$,
weak-$*$ sequential compactness of the unit ball of $L^\infty(Q_j)$ gives a
subsequence, not relabeled, and a function
$h_{0,j}^\omega\in L^\infty(Q_j)$ such that
$$
h(\omega,t_n,\cdot,\cdot)
\overset{*}{\rightharpoonup}
h_{0,j}^\omega
\quad\text{in }L^\infty(Q_j).
$$
Moreover, $0\leq h_{0,j}^\omega\leq1$ almost everywhere on $Q_j$.
For every $\varphi\in\cD_j$,
$$
\int_{Q_j}
h_{0,j}^\omega(x,\lambda)\varphi(x,\lambda)
\,d\lambda\,dx
=
\esslim_{t\downarrow0}
\mathscr H_\varphi(t,\omega).
$$
The right-hand side {exists by Lemma~\ref{lem:time-integrated} and it} 
is independent of both the sequence $(t_n)$ and the
extracted subsequence.  Hence, every weak-$*$ accumulation point has the same
pairings against $\cD_j$.  Since $\cD_j$ is dense in
$L^1(Q_j)$, the weak-$*$ accumulation point is unique.

This uniqueness implies the essential weak-$*$ convergence of the whole family
on $Q_j$.  Indeed, if the convergence failed, then, using the uniform
$L^\infty$-bound and the density of $\cD_j$, one could find
$\varphi\in\cD_j$, $\varepsilon>0$, and a sequence
$t_n\in E_\omega$, $t_n\downarrow0$, such that
$$
\abs{
\int_{Q_j}
\bigl(
h(\omega,t_n,x,\lambda)-h_{0,j}^\omega(x,\lambda)
\bigr)
\varphi(x,\lambda)
\,d\lambda\,dx
}
\geq\varepsilon.
$$
Weak-$*$ compactness would then produce a different accumulation point, which
is impossible.

The uniform bound
$0\leq h(\omega,t,\cdot,\cdot),h_{0,j}^\omega\leq1$ and the
$L^1(Q_j)$-density of $\cD_j$ extend the convergence from $\cD_j$ to
every test function in $L^1(Q_j)$.

If $j<k$, then the restriction of $h_{0,k}^\omega$ to $Q_j$ and
$h_{0,j}^\omega$ have the same pairings against $\cD_j$.  Hence
they agree almost everywhere on $Q_j$.  After choosing representatives with
this restriction property, define
$$
h_0^\omega(x,\lambda)
:=
h_{0,j}^\omega(x,\lambda)
\quad
\text{whenever }(x,\lambda)\in Q_j.
$$
The restriction identity makes this definition independent of $j$ outside
one null set.  It gives
$
h_0^\omega\in L^\infty(\R^d\times(-L,L))
$
with $0\leq h_0^\omega\leq1$.
Extend it by setting $h_0^\omega(x,\lambda)=1$ for $\lambda\leq-L$
and $h_0^\omega(x,\lambda)=0$ for $\lambda\geq L$.
Because $\abs{u}\leq M<L$, the same identities hold for
$h(\omega,t,x,\lambda)$ outside $(-L,L)$.  The convergence already proved
on each $Q_j$, applied to the restriction of an arbitrary compactly
supported test function to $Q_j$, therefore proves
\eqref{eq:pathwise-weak-trace} on the full kinetic line.  If
$\widehat h_0^\omega$ also satisfies \eqref{eq:pathwise-weak-trace}, then
$h_0^\omega-\widehat h_0^\omega$ has zero pairing with every
$\varphi\in\cD_j$.  Density gives equality almost everywhere on every
$Q_j$; testing below $-L$ and above $L$ gives the same constant values
there.  Hence the trace is unique on $\R^d\times\R$.
\end{proof}

The next lemma incorporates this weak trace as the boundary term at $t=0$ in
the kinetic identity \eqref{eq:kinetic-weak}.

\begin{lemma}[Pathwise kinetic identity including the weak trace]
\label{lem:kinetic-with-trace}
For every test function
\[
\varphi
\in
C_c^1\bigl([0,T)\times\R^d\times(-L,L)\bigr),
\]
there exists a set
\[
\Omega_{\mathrm{tr},\varphi}
\subset
\Omega_{\mathrm w},
\qquad
\Prob(\Omega_{\mathrm{tr},\varphi})=1,
\]
such that, for every
$\omega\in\Omega_{\mathrm{tr},\varphi}$, the identity
\begin{equation}\label{eq:kinetic-with-trace}
\begin{aligned}
	&\int_0^T
	\int_{\R^d}
	\int_{\R}
	h(\omega,t,x,\lambda)
	\bigl(
	\partial_t\varphi(t,x,\lambda)
	+
	a(\lambda)\cdot\nabla_x\varphi(t,x,\lambda)
	\bigr)
	\,d\lambda\,dx\,dt
	\\
	&\quad+
	\int_{\R^d}
	\int_{\R}
	h_0^\omega(x,\lambda)
	\varphi(0,x,\lambda)
	\,d\lambda\,dx
	\\
	&=
	-\sum_{\ell\geq1}
	\int_0^T
	\left(
	\int_{\R^d}
	b_\ell(t,x,u(\omega,t,x))
	\varphi(t,x,u(\omega,t,x))
	\,dx
	\right)
	dW_\ell(t)
	\\
	&\quad-
	\frac12
	\int_0^T
	\int_{\R^d}
	G^2(t,x,u(\omega,t,x))
	\partial_\lambda\varphi(t,x,u(\omega,t,x))
	\,dx\,dt
	\\
	&\quad+
	\int_{(0,T)\times\R^d\times\R}
	\partial_\lambda\varphi(t,x,\lambda)
	\,dm(\omega;t,x,\lambda)
\end{aligned}
\end{equation}
holds. Moreover, for every prescribed countable family
\[
\mathscr T
\subset
C_c^1\bigl([0,T)\times\R^d\times(-L,L)\bigr),
\]
the event
\[
\Omega_{\mathrm{tr},\mathscr T}
:=
\bigcap_{\varphi\in\mathscr T}
\Omega_{\mathrm{tr},\varphi}
\]
is a full-probability subset of $\Omega_{\mathrm w}$ on which
\eqref{eq:kinetic-with-trace} holds simultaneously for every
$\varphi\in\mathscr T$.
\end{lemma}

\begin{proof}
Fix
\[
\varphi
\in
C_c^1\bigl([0,T)\times\R^d\times(-L,L)\bigr)
\]
and define the associated scalar martingale by
\begin{equation}\label{eq:time-dependent-scalar-martingale}
M_\varphi(t,\omega)
:=
\sum_{\ell\geq1}
\int_0^t
\left(
\int_{\R^d}
b_\ell(r,x,u(\omega,r,x))
\varphi(r,x,u(\omega,r,x))
\,dx
\right)
dW_\ell(r).
\end{equation}
By \eqref{eq:noise-local-bound}, $M_\varphi$ admits a continuous
version on $[0,T]$ starting from zero. If
\[
\varphi(t,x,\lambda)=\phi(x,\lambda)
\]
is independent of time, then $M_\varphi=M_\phi$, where $M_\phi$ is
defined in \eqref{eq:scalar-martingale-pairing}.

Fix a deterministic sequence $\delta_n\downarrow0$ with
$0<\delta_n<T$. For every $n$, choose a nondecreasing function
\[
\theta_n\in C^1([0,T])
\]
such that
\[
0\leq\theta_n\leq1,
\qquad
\theta_n(t)=0
\quad\text{for }0\leq t\leq\frac{\delta_n}{2},
\qquad
\theta_n(t)=1
\quad\text{for }t\geq\delta_n.
\]
In particular,
\[
\theta_n'\geq0,
\qquad
\supp\theta_n'
\subset
\left[\frac{\delta_n}{2},\delta_n\right],
\qquad
\int_0^T\theta_n'(t)\,dt=1.
\]

Choose
\[
\Omega_{\mathrm{tr},\varphi}
\subset
\Omega_{\mathrm w}
\]
of probability one such that the chosen continuous version of
$M_\varphi$ is fixed and \eqref{eq:kinetic-weak} holds for every member
of the countable family
\(
(\theta_n\varphi)_{n\geq1}.
\)
Fix $\omega\in\Omega_{\mathrm{tr},\varphi}$.
Use
\[
\theta_n(t)\varphi(t,x,\lambda)
\]
as a test function in \eqref{eq:kinetic-weak}. Since
$\theta_n(t)\to1$ for every $t>0$, dominated convergence applies to
the terms containing
\[
\theta_n\partial_t\varphi,
\qquad
\theta_n a(\lambda)\cdot\nabla_x\varphi,
\qquad
\theta_n G^2\partial_\lambda\varphi.
\]
Furthermore, since
\[
\Omega_{\mathrm{tr},\varphi}
\subset
\Omega_{\mathrm w}
\subset
\Omega_m,
\]
the local finiteness property \eqref{eq:local-mass-event} and dominated
convergence give convergence of the kinetic-measure term.
The stochastic term is
\[
\int_0^T
\theta_n(t)\,dM_\varphi(t,\omega).
\]
Integration by parts gives
\[
\int_0^T
\theta_n(t)\,dM_\varphi(t,\omega)
=
M_\varphi(T,\omega)
-
\int_0^T
M_\varphi(t,\omega)\theta_n'(t)\,dt.
\]
Since $\theta_n'\geq0$,
$\supp\theta_n'\subset[\delta_n/2,\delta_n]$, and
$\int_0^T\theta_n'(t)\,dt=1$, we obtain
\[
\left|
\int_0^T
M_\varphi(t,\omega)\theta_n'(t)\,dt
\right|
\leq
\sup_{0\leq t\leq\delta_n}
\left|
M_\varphi(t,\omega)
\right|
\longrightarrow0
\]
by continuity of $M_\varphi(\cdot,\omega)$ at zero. Hence
\[
\int_0^T
\theta_n(t)\,dM_\varphi(t,\omega)
\longrightarrow
M_\varphi(T,\omega),
\]
which is the stochastic integral appearing in
\eqref{eq:kinetic-with-trace}.
It remains to identify the term generated by $\theta_n'$. Set
\[
B_n
:=
\int_0^T
\theta_n'(t)
\int_{\R^d}\int_{\R}
h(\omega,t,x,\lambda)
\varphi(t,x,\lambda)
\,d\lambda\,dx\,dt.
\]
To apply the weak-trace property, freeze the test function at $t=0$
and define
\[
F(t)
:=
\int_{\R^d}\int_{\R}
h(\omega,t,x,\lambda)
\varphi(0,x,\lambda)
\,d\lambda\,dx
\]
and
\[
F_0
:=
\int_{\R^d}\int_{\R}
h_0^\omega(x,\lambda)
\varphi(0,x,\lambda)
\,d\lambda\,dx.
\]
Lemma~\ref{lem:weak-trace}, applied to the fixed test function
$\varphi(0,\cdot,\cdot)$, gives
\[
\esslim_{t\downarrow0}F(t)=F_0.
\]

Using $0\leq h\leq1$ and $\int_0^T\theta_n'(t)\,dt=1$, we obtain
\[
\begin{aligned}
\left|B_n-F_0\right|
&\leq
\int_0^T
\theta_n'(t)
\left|F(t)-F_0\right|
\,dt
\\
&\quad+
\int_0^T
\theta_n'(t)
\int_{\R^d}\int_{\R}
h(\omega,t,x,\lambda)
\left|
\varphi(t,x,\lambda)-\varphi(0,x,\lambda)
\right|
\,d\lambda\,dx\,dt
\\
&\leq
\operatorname*{ess\,sup}_{0<t<\delta_n}
\left|F(t)-F_0\right|
\\
&\quad+
\sup_{0<t<\delta_n}
\norm{
\varphi(t,\cdot,\cdot)
-
\varphi(0,\cdot,\cdot)
}_{L^1(\R^d\times\R)}.
\end{aligned}
\]
The first term tends to zero by the weak-trace property, while the
second tends to zero by the continuity of $\varphi$ at $t=0$.
Consequently,
\[
B_n\longrightarrow F_0.
\]

Passing to the limit in \eqref{eq:kinetic-weak}, tested with
$\theta_n\varphi$, now gives \eqref{eq:kinetic-with-trace}.

Finally, if $\mathscr T$ is a prescribed countable family of test
functions, then
\[
\Omega_{\mathrm{tr},\mathscr T}
=
\bigcap_{\varphi\in\mathscr T}
\Omega_{\mathrm{tr},\varphi}
\]
has probability one, and \eqref{eq:kinetic-with-trace} holds on this
event simultaneously for every $\varphi\in\mathscr T$.
\end{proof}

The next well-known lemma (see, e.g., \cite{Debussche:2010fk})
upgrades weak convergence to strong convergence once the weak limit is
known to be a subgraph function.
\begin{lemma}[Rigidity of subgraph functions]
\label{lem:rigidity-subgraph}
Let $(X,\mu)$ be a finite measure space, let
$$
v_n:X\longrightarrow[-M,M],
$$
and define
$$
g_n(x,\lambda)
:=
\1_{\{\lambda<v_n(x)\}}
\quad\text{on }X\times(-L,L),
$$ {where $L>M+1$.}
Assume that
$$
g_n\overset{*}{\rightharpoonup}g
\quad\text{in }
L^\infty(X\times(-L,L)).
$$
Then the following statements are equivalent:
\begin{enumerate}[label=\textup{(\roman*)}]
	\item there exists a measurable function
	$$
	v:X\longrightarrow[-M,M]
	$$
	such that
	$$
	g(x,\lambda)
	=
	\1_{\{\lambda<v(x)\}}
	$$
	almost everywhere;

	\item
	$$
	g_n\longrightarrow g
	\quad\text{strongly in }
	L^1(X\times(-L,L)).
	$$
\end{enumerate}
Whenever these conditions hold,
$$
\norm{g_n-g}_{L^1(X\times(-L,L))}
=
\norm{v_n-v}_{L^1(X)}.
$$
In particular,
$$
v_n\longrightarrow v
\quad\text{strongly in }L^1(X).
$$
\end{lemma}

\begin{proof}
Assume first that
$$
g(x,\lambda)
=
\1_{\{\lambda<v(x)\}}.
$$
{Using the assumed weak-$\ast$ convergence in 
	$L^\infty(X\times (-L,L))$ together with the fact that 
	constant functions are in $L^1(X\times (-L,L))$, we have}
$$
\norm{g_n}_{L^2(X\times(-L,L))}^2
=
\int_{X\times(-L,L)}g_n
\longrightarrow
\int_{X\times(-L,L)}g
=
\norm{g}_{L^2(X\times(-L,L))}^2.
$$
Hence
$$
g_n\longrightarrow g
\quad\text{strongly in }L^2(X\times(-L,L)),
$$
and therefore strongly in $L^1(X\times(-L,L))$.

Conversely, suppose that
$$
g_n\longrightarrow g
\quad\text{strongly in }L^1(X\times(-L,L)).
$$
For $n,m\geq1$, the subgraph identity gives
$$
\norm{v_n-v_m}_{L^1(X)}
=
\norm{g_n-g_m}_{L^1(X\times(-L,L))},
$$
{where we have used that for all $r,q\in[-M,M]$,
$$
\abs{r-q}
=
\int_{-L}^L
\abs{
	\1_{\{\lambda<r\}}
	-
	\1_{\{\lambda<q\}}
}
\,d\lambda.
$$
Thus $(v_n)$ is a Cauchy sequence} in $L^1(X)$.  Let
$v_n\to v$ in $L^1(X)$ for some measurable
$v:X\to[-M,M]$. 

{
The previous identity yields
$$
\lim_{n\to\infty} \norm{g_n-\1_{\{\lambda<v(x)\}}}_{L^1(X\times(-L,L))}
	= \lim_{n\to\infty} \norm{v_n-v}_{L^1(X)} = 0 \,.
$$
Uniqueness of the strong limit gives
$g(x,\lambda)=\1_{\{\lambda<v(x)\}}$ almost everywhere,
completing the proof. }

See also \cite[Lemma~2.3]{Frid:2022trace}.
\end{proof}

Applying Lemma~\ref{lem:rigidity-subgraph} to time sections yields the
following essential-time version.

\begin{corollary}[Essential-time rigidity]
\label{cor:essential-chi-rigidity}
Let $(X,\mu)$ be a finite measure space and let
$v:(0,T)\times X\to[-M,M]$ be measurable. For each $t\in(0,T)$, let
$g_t\in L^\infty(X\times(-L,L))$ be represented by
$g_t(x,\lambda):=\1_{\{\lambda<v(t,x)\}}$.
Suppose that
$$
\esslim_{t\downarrow0}
g_t
=
g_0
$$
weak-$*$ in $L^\infty(X\times(-L,L))$, where
$g_0(x,\lambda)=\1_{\{\lambda<v_0(x)\}}$ for almost every
$(x,\lambda)\in X\times(-L,L)$ and some measurable function
$v_0:X\to[-M,M]$.
Then
$$
\esslim_{t\downarrow0}
\norm{g_t-g_0}_{L^1(X\times(-L,L))}
=
0.
$$
Equivalently,
$$
\esslim_{t\downarrow0}
\int_X
\abs{v(t,x)-v_0(x)}
\,d\mu(x)
=
0.
$$
\end{corollary}

\begin{proof}
Since $0\leq g_t,g_0\leq1$ and $X\times(-L,L)$ has finite measure, the
functions $1$ and $g_0$
belong to $L^1(X\times(-L,L))$.  The essential weak-$*$ convergence therefore
gives
$$
\esslim_{t\downarrow0}
\int_{X\times(-L,L)}
g_t(x,\lambda)
\,d\mu(x)\,d\lambda
=
\int_{X\times(-L,L)}
g_0(x,\lambda)
\,d\mu(x)\,d\lambda
$$
and
$$
\esslim_{t\downarrow0}
\int_{X\times(-L,L)}
g_t(x,\lambda)g_0(x,\lambda)
\,d\mu(x)\,d\lambda
=
\int_{X\times(-L,L)}
g_0(x,\lambda)^2
\,d\mu(x)\,d\lambda.
$$
Because $g_t$ and $g_0$ take only the values zero and one,
$g_t^2=g_t$ and $g_0^2=g_0$,
and hence
\begin{align*}
	\norm{g_t-g_0}_{L^1(X\times(-L,L))}
	&=
	\int_{X\times(-L,L)}
	\abs{g_t-g_0}^2
	\,d\mu\,d\lambda
	\\
	&=
	\int_{X\times(-L,L)}
	g_t
	\,d\mu\,d\lambda
	+
	\int_{X\times(-L,L)}
	g_0
	\,d\mu\,d\lambda
	\\
	&\quad
	-
	2
	\int_{X\times(-L,L)}
	g_t g_0
	\,d\mu\,d\lambda.
\end{align*}
{Using the previously established identities, 
	we get that} the right-hand side converges to zero in 
	the essential-time sense. Therefore,
$$
\esslim_{t\downarrow0}
\norm{g_t-g_0}_{L^1(X\times(-L,L))}
=
0.
$$
Finally, the pointwise identity
$$
\abs{r-q}
=
\int_{-L}^L
\abs{
\1_{\{\lambda<r\}}
-
\1_{\{\lambda<q\}}
}
\,d\lambda
$$
gives the equivalent convergence of $v(t,\cdot)$ to $v_0$ in $L^1(X)$.
\end{proof}

%%%%%%%%%%%%%%%%%%%%%%%%%
%%%%%%%%%%%%%%%%%%%%%%%%%
\section{Panov's compactness principle}
\label{sec:panov-compactness}

We record the deterministic compactness results used for the pathwise
blow-up sequences.

Let $\mathcal O\subset\R^N$ be open, and let $(v_n)$ be bounded in
$L^\infty(\mathcal O)$. After
passing to a subsequence, let $\nu_z$ be the Young measure 
{\cite{Young:1969} (or the measure-valued function 
\cite{Panov:1994gf})} generated by
$(v_n)$.  For $p\in\R$, set
$$
U_n^p(z)
:=
\mathbf 1_{\{v_n(z)>p\}}
-
\nu_z((p,\infty)).
$$

Panov's two-parameter construction associates H-measures \cite{Tartar:1990mq}
with these centered
level functions; see \cite[{Section~5}]{Panov:2007aa} and
\cite[{Section~6}]{Panov:2005lr}.

\begin{theorem}[Panov's two-parameter H-measures]
\label{thm:panov-hmeasure}
There exist a subsequence, not relabeled, and a set
\[
P\subset\R,
\]
whose complement is at most countable, such that, for every $p\in P$,
\[
\mathbf 1_{\{v_n>p\}}
\overset{*}{\rightharpoonup}
\nu_z((p,\infty))
\qquad\text{in }L^\infty_{\loc}(\mathcal O).
\]
Consequently,
\[
U_n^p\overset{*}{\rightharpoonup}0
\qquad\text{in }L^\infty_{\loc}(\mathcal O).
\]

For every $p,q\in P$, there exists a locally finite complex Borel measure
\[
\mu^{pq}
\in
\cM_{\loc}
\bigl(
\mathcal O\times\mathbb S^{N-1}
\bigr)
\]
representing the H-measure associated with the pair
$(U_n^p,U_n^q)$. The family
\[
\{\mu^{pq}\}_{p,q\in P}
\]
is Hermitian and positive semidefinite. In particular,
\[
\mu^{pp}\geq0
\qquad\text{for every }p\in P.
\]

Moreover, the map
\[
P\times P\ni(p,q)
\longmapsto
\mu^{pq}
\in
\cM_{\loc}
\bigl(
\mathcal O\times\mathbb S^{N-1}
\bigr)
\]
is continuous with respect to the local total-variation topology.
More precisely, for every compact set
\[
K\subset\mathcal O\times\mathbb S^{N-1}
\]
and every $(p_0,q_0)\in P\times P$,
\[
\left|
\mu^{pq}-\mu^{p_0q_0}
\right|(K)
\longrightarrow0
\qquad\text{as }(p,q)\longrightarrow(p_0,q_0),
\]
where $|\mu|$ denotes the total variation of $\mu$. In particular, for
every
\[
\psi\in
C_c
\bigl(
\mathcal O\times\mathbb S^{N-1}
\bigr),
\]
the map
\[
(p,q)
\longmapsto
\int_{\mathcal O\times\mathbb S^{N-1}}
\psi(z,\xi)\,d\mu^{pq}(z,\xi)
\]
is continuous on $P\times P$.

Finally,
\[
\mu^{pp}=0
\quad\text{for every }p\in P
\]
if and only if
\[
\mu^{pq}=0
\quad\text{for every }p,q\in P.
\]
These equivalent conditions hold if and only if
\[
v_n\longrightarrow
\overline v,
\qquad
\overline v(z)
:=
\int_{\R}\lambda\,d\nu_z(\lambda),
\]
strongly in $L^1_{\loc}(\mathcal O)$.
\end{theorem}

The localization theorem below identifies the obstruction represented by
a nonzero H-measure: it forces a directional component of the flux to be
constant on a nondegenerate interval; see
\cite[Theorem~5.7]{Panov:2007aa} and
\cite[Theorem~6.7]{Panov:2005lr}.

\begin{theorem}[Panov's localization theorem]
	\label{thm:panov-localization}
	Let $\Phi\in C(\R;\R^N)$, and let $(v_n)$ be a bounded sequence in $L^\infty(\mathcal O)$. Suppose that, for every $k\in\R$, the one-sided entropy productions
	\begin{equation}\label{eq:panov-C}
	\diver_z
	\left[
	\mathbf 1_{\{v_n>k\}}
	\bigl(
	\Phi(v_n)-\Phi(k)
	\bigr)
	\right]
	\end{equation}
	are precompact in $H^{-1}_{\loc}(\mathcal O)$.
	
	If $\mu^{p_0p_0}\neq0$ for some $p_0\in P$, then one can find an open interval $D\ni p_0$ and a vector $\eta\in \mathbb S^{N-1}$ such that
	$$
	\eta\cdot\Phi(r) = \mathrm{const} \qquad \text{for every }r\in D.
	$$
	Consequently, if $r\mapsto\eta\cdot\Phi(r)$ is nonconstant on every nondegenerate interval for every $\eta\in \mathbb S^{N-1} $, then $(v_n)$ is strongly precompact in $L^1_{\loc}(\mathcal O)$.
\end{theorem}

\begin{remark}
{
The assumption that \eqref{eq:panov-C} is precompact is equivalent
to the customary signed formulation. To see this, let us denote
the distribution in \eqref{eq:panov-C} by $A_n(k)$. For every 
$k\in\R$, the signed entropy production satisfies the algebraic identity
$$
\diver_z
	\left[
	\operatorname{sgn}(v_n-k)
	\bigl(
	\Phi(v_n)-\Phi(k)
	\bigr)
	\right]
	=
	2A_n(k)
	-
	\diver_z\Phi(v_n).
$$
Since $(v_n)$ is bounded in $L^\infty(\mathcal O)$, 
there exists a constant $C>0$ such that $\abs{v_n}\leq C$ 
uniformly. Assuming that $A_n(k)$ is precompact in 
$H^{-1}_{\loc}(\mathcal O)$ for every $k \in \R$, 
we can evaluate it at some $k_0<-C$. For such $k_0$, 
we have $\mathbf 1_{\{v_n>k_0\}} = 1$ everywhere, 
which yields $A_n(k_0)=\diver_z\Phi(v_n)$. Therefore, 
the divergence of the flux is also precompact. 
It follows that the right-hand side of the identity 
above is a sum of precompact sequences, and hence 
the signed formulation is precompact. 
The reverse implication can be obtained analogously.}
\end{remark}

Theorems~\ref{thm:panov-hmeasure}
and~\ref{thm:panov-localization}, including the continuity of the
two-parameter family and the localization conclusion, originate in
Panov's work \cite{Panov:1994gf,Panov:2005lr,Panov:2007aa}.
We use the following consequence to treat possible degenerate intervals
separately by truncation.

\begin{lemma}[Compactness from compact truncations]
\label{lem:panov-closure}
Let $J\subset\R$ be open, let $K_0=[\alpha,\beta]\Subset J$, and let
$F\subset J$ be countable and dense. Denote by
$\mathscr I_\Phi(K_0)$ the family of intervals
$$
[c,d]\Subset(\alpha,\beta),
\qquad
c,d\in F,
\qquad
c<d,
$$
for which one can find an open interval
$$
D\Subset J,
\qquad
[c,d]\Subset D,
$$
and a vector $\eta\neq0$ such that
$$
\eta\cdot\Phi(r)
=
\mathrm{const}
\qquad
\text{for every }r\in D.
$$

Suppose that $(v_n)$ is bounded in $L^\infty(\mathcal O)$, takes values
in $K_0$, and satisfies Panov's condition
\eqref{eq:panov-C}.  Assume also that
$T_{c,d}(v_n)$ is strongly precompact in
$L^1_{\loc}(\mathcal O)$ for every
$[c,d]\in\mathscr I_\Phi(K_0)$.
Then $(v_n)$ is strongly precompact in
$L^1_{\loc}(\mathcal O)$.
\end{lemma}

\begin{proof}
The assertion is immediate if $\alpha=\beta$.  Assume therefore that
$\alpha<\beta$, and suppose, to obtain a contradiction, that there is a
subsequence, still denoted by $(v_n)$, having no strongly convergent
subsequence in $L^1_{\loc}(\mathcal O)$.

After passing to a further subsequence, let $\nu_z$ be the corresponding
Young measure, and let $\{\mu^{pq}\}_{p,q\in P}$ be the family supplied
by Theorem~\ref{thm:panov-hmeasure}. Since the
sequence is not strongly precompact, this family does not vanish identically.
Its positive semidefiniteness therefore gives a level $p\in P$ such that
$\mu^{pp}\neq0$.

The centered level functions vanish identically for $p<\alpha$ and
$p\geq\beta$. Hence $p\in[\alpha,\beta)$.
By Theorem~\ref{thm:panov-localization}, one can find an open interval
$D\ni p$ and a vector $\eta\neq0$ such that
$$
\eta\cdot\Phi(r)
=
\mathrm{const}
\qquad
\text{for every }r\in D.
$$
Since $K_0\Subset J$, we may shrink $D$ so that $D\Subset J$. Because
$\mu^{pp}$ is a nonzero nonnegative measure, there exists
$\psi\in C_c(\mathcal O\times \mathbb S^{N-1})$, with $\psi\geq0$, such that
$\int\psi\,d\mu^{pp}>0$.
By continuity of the diagonal family, there exists an open interval
$D_1\ni p$, with $D_1\Subset D$, such that
$$
\int\psi\,d\mu^{qq}>0
\qquad
\text{for every }q\in P\cap D_1.
$$
In particular,
 $\mu^{qq}\neq0$ for every $q\in P\cap D_1$. Since
$p\in[\alpha,\beta)$, the set $D_1\cap(\alpha,\beta)$ contains a
nonempty open interval. We may therefore choose
$c,d\in F$ such that
$$
[c,d]
\Subset
D_1\cap(\alpha,\beta).
$$
Then $[c,d]\in\mathscr I_\Phi(K_0)$. Set $w_n:=T_{c,d}(v_n)$.
By assumption, $(w_n)$ is strongly precompact locally.  Passing to a
further subsequence, we may assume that
$$
w_n\longrightarrow w
\quad\text{strongly in }L^1_{\loc}(\mathcal O)
$$
and almost everywhere in $\mathcal O$.
Furthermore, there exists a countable set $E_w\subset\R$ such that, for
every $q\notin E_w$,
$$
\mathbf 1_{\{w_n>q\}}
\longrightarrow
\mathbf 1_{\{w>q\}}
\quad\text{strongly in }L^1_{\loc}(\mathcal O).
$$
Indeed, let $Q_1\Subset Q_2\Subset\cdots\Subset\mathcal O$ be a
countable exhaustion. For each $j$, there are at most countably many
levels $q$ for which
$$
\mathcal L^N
\bigl(
\{z\in Q_j:w(z)=q\}
\bigr)>0.
$$
Outside the union of these exceptional sets, the asserted convergence follows
from the almost-everywhere convergence and dominated convergence on every
$Q_j$.

Since $\R\setminus P$ is at most countable, we can choose
$q\in P\cap(c,d)\setminus E_w$. By the choice of $[c,d]$,
$\mu^{qq}\neq0$.
Since $q\in(c,d)$, truncation does not change the corresponding level set:
$\mathbf 1_{\{w_n>q\}}=\mathbf 1_{\{v_n>q\}}$.
Because $q\in P$ (see Theorem \ref{thm:panov-hmeasure}),
$$
\mathbf 1_{\{v_n>q\}}
\overset{*}{\rightharpoonup}
\nu_z((q,\infty))
\quad\text{in }L^\infty_{\loc}(\mathcal O).
$$
On the other hand,
$$
\mathbf 1_{\{w_n>q\}}
\longrightarrow
\mathbf 1_{\{w>q\}}
\quad\text{strongly in }L^1_{\loc}(\mathcal O).
$$
The weak-$*$ limit is unique, and hence
$$
\nu_z((q,\infty))
=
\mathbf 1_{\{w>q\}}
$$
almost everywhere.  It follows that
$$
U_n^q
=
\mathbf 1_{\{v_n>q\}}
-
\nu_z((q,\infty))
\longrightarrow0
\quad\text{strongly in }L^1_{\loc}(\mathcal O).
$$
Since
$\abs{U_n^q}\leq1$,
the convergence also holds strongly in
$L^2_{\loc}(\mathcal O)$.  Therefore, the H-measure associated with
$U_n^q$ vanishes: $\mu^{qq}=0$.
This contradicts the choice of $q$.  The contradiction proves the strong
precompactness of $(v_n)$ in $L^1_{\loc}(\mathcal O)$.
\end{proof}

We also use the following standard consequence of Murat's interpolation
lemma \cite{MuratII}.

\begin{lemma}[Murat's interpolation lemma]
\label{lem:murat-interpolation}
Let $Q\Subset\R^N$ be a bounded domain with Lipschitz boundary.
Suppose that $(F_n)$ is bounded in
$
W^{-1,\infty}(Q)
$
and that $F_n=A_n+B_n$, where $(A_n)$ is relatively compact in
$H^{-1}(Q)$ and $(B_n)$ is bounded in $\cM(Q)$.
Then $(F_n)$ is relatively compact in $H^{-1}(Q)$.  If, in addition,
\begin{equation}\label{eq:murat-zero-assumptions}
A_n\longrightarrow0
\quad\text{in }H^{-1}(Q),
\qquad
B_n\longrightarrow0
\quad\text{in }\mathcal D'(Q),
\end{equation}
then
$$
F_n\longrightarrow0
\quad\text{in }H^{-1}(Q).
$$
\end{lemma}

%%%%%%%%%%%%%%%%%%%%%%%%%
%%%%%%%%%%%%%%%%%%%%%%%%%
\section{Blow-up at the initial hyperplane}
\label{sec:pathwise-blowup}

\subsection{Countable preparation for dimension reduction}
\label{subsec:countable-reduction}

The blow-up argument gives compactness unless
Theorem~\ref{thm:panov-localization} yields a kinetic interval on which a
directional component of the space--time flux is constant.  On such an interval,
Lemma~\ref{lem:pathwise-linear-change}
isolates one spatial direction, and fixing the corresponding transverse
coordinate produces a problem in one lower dimension.  This reduction may
have to be repeated, but at most $d$ times.

Once $\omega$ is fixed, the kinetic measure can be sliced deterministically.
The martingale identities require additional preparation: for a fixed test
function and coordinate map, each identity used after a reduction holds
outside an exceptional probability-null set that may depend on those choices.
Panov's reduction and the transverse coordinate are selected only during the
pathwise argument, so separate almost-sure statements for fixed choices are
insufficient.  To obtain one realization for which every possible reduction
step is valid, we select in \textit{advance} all required test functions,
coordinate maps, and continuous versions of the associated martingale
processes, and then remove the union of their exceptional sets.
The deterministic flux and countable kinetic levels make this collection
countable, so the union remains a null set.  The following reduction scheme
presents this preparation.

\medskip

A reduction datum $\alpha$ consists of a fixed finite chain of one-step
reductions from the original datum and carries the spatial dimension
$d_\alpha$, the kinetic interval $I_\alpha$, the attached truncation
$[a_\alpha,b_\alpha]$, and the reduced flux $f_\alpha$.  The original datum
$\alpha_0$ has components $\bigl(d,(-L,L),[-M,M],f\bigr)$ and an empty
reduction chain.
Here $d$ is the spatial dimension in \eqref{eq:spde}, $M$ is the uniform
bound $\abs{u}\leq M$, $L>M+1$ determines the ambient kinetic interval, and
$f:\R\to\R^d$ is the flux function in \eqref{eq:spde}.
For a reduction datum $\alpha$ of spatial dimension
$q=d_\alpha\geq1$, use the countable dense set of
kinetic levels fixed in Lemma~\ref{lem:panov-closure}.  For every degenerate
interval $[c,e]$ in the resulting countable family, fix an open interval
$D$ with $[c,e]\Subset D\Subset I_\alpha$, a vector
$\eta\in\R^{q+1}\setminus\{0\}$ such that
$r\mapsto\eta\cdot(r,f_\alpha(r))$ is constant on $D$, and one linear change
of variables from
Lemma~\ref{lem:pathwise-linear-change}.  The associated one-step reduction
$\beta$ has spatial dimension $d_\beta=q-1$, kinetic interval
$
I_\beta:=D,
$
attached truncation
$
[a_\beta,b_\beta]:=[c,e],
$
and flux
$$
f_\beta(r)
:=
\pi_\parallel
\bigl(
A_{\alpha\beta}f_\alpha(r)
+
c_{\alpha\beta}r
\bigr),
$$
where $\pi_\parallel$ projects onto the first $q-1$ coordinates.  Every
one-step reduction is deterministic and lowers the dimension.  The
countable family of all reduction data constructed in this way is denoted
by $\mathfrak R$ (finite reduction chains carrying $d_\alpha$, $I_\alpha$,
$[a_\alpha,b_\alpha]$, and $f_\alpha$).

For every reduction datum $\alpha\in\mathfrak R$, fix a countable
$L^1_{\loc}$-determining family
$$
\mathscr D_\alpha
\subset
C_c^1(\R^{d_\alpha}\times I_\alpha)
$$
of spatial--kinetic test functions.  Precisely, for every relatively compact open
cylinder
$
Q\Subset\R^{d_\alpha}\times I_\alpha,
$
the rational span of
$
\mathscr D_\alpha\cap C_c^1(Q)
$
is dense in $L^1(Q)$.  We choose these members as finite rational sums of
$
\chi(x)\rho(\lambda)
$
with $\rho$ in the countable family $\mathcal R$ fixed below and
$\supp\rho\Subset I_\alpha$.  Fix also a
countable family
$
\Theta\subset C_c^1(0,T)
$
whose rational span is dense in the $C^1$ norm among functions supported in
each fixed compact subinterval of $(0,T)$, and define
$$
\mathscr T_\alpha
:=
\operatorname{span}_{\mathbb Q}
\left\{
\theta(t)\psi(x,\lambda):
\theta\in\Theta,\quad
\psi\in\mathscr D_\alpha
\right\}.
$$
The spatial factors $\chi$ and the time factors in $\Theta$ are chosen so
that, for every fixed $\rho\in\mathcal R$ with
$\supp\rho\Subset I_\alpha$, the rational span of the tensors in
$\mathscr T_\alpha$ with kinetic factor $\rho$ is dense, in the $C^1$ norm
on each fixed compact cylinder, among the functions
$\eta(t,x)\rho(\lambda)$ with
$\eta\in C_c^1((0,T)\times\R^{d_\alpha})$.  This additional property is
available by separability and will be used to extend the reduced kinetic
identities by continuity.
For every datum and one-step reduction, include the cutoff sequences from
\eqref{eq:cutoff-approximation}, balls with rational centers and radii,
compact transverse intervals with rational endpoints, the fixed mollifier
sequence in Lemma~\ref{lem:pathwise-measure-slicing}, and the fixed linear
pullbacks associated with the reduction.  These additions remain countable.
All stochastic integrals and martingale processes associated with this fixed
collection are constructed before a realization is chosen.

We also record the transverse spatial coordinates fixed in the preceding
reduction steps.  For a datum $\alpha$, the integer
$r_\alpha:=d-d_\alpha$ is the number of such coordinates, and
$\sigma=(\zeta_1,\ldots,\zeta_{r_\alpha})\in\R^{r_\alpha}$ collects their
fixed values, while $x\in\R^{d_\alpha}$ denotes the remaining active spatial
variables.  Thus $\sigma$ is a spatial parameter; for the original datum
$\alpha_0$, $r_{\alpha_0}=0$ and $\sigma=\varnothing$.

Together, $x$ and $\sigma$ determine the physical point
$
y=\mathcal X_\alpha(t,x,\sigma)\in\R^d.
$
Composing the inverse changes of variables
\eqref{eq:pathwise-linear-change} along the reduction chain from
$\alpha_0$ to
$\alpha$ gives the coordinate reconstruction map
$$
\mathcal X_\alpha:
(0,T)\times\R^{d_\alpha}\times\R^{r_\alpha}
\longrightarrow
\R^d.
$$
For the original datum,
$
\mathcal X_{\alpha_0}(t,x,\varnothing)=x,
$
because its active coordinate is the original physical coordinate.
Suppose that $\beta$ is a one-step reduction of $\alpha$ obtained using
$
w=A_{\alpha\beta}x+c_{\alpha\beta}t,
$
where $q=d_\alpha$ and $w'=(w_1,\ldots,w_{q-1})$.  Fixing
$w_q=\zeta$ makes $w'$ the active coordinate for $\beta$ and appends
$\zeta$ to the accumulated parameter.  The reconstruction maps then satisfy
\begin{equation}\label{eq:reduction-coordinate-recursion}
	\mathcal X_\beta
	\bigl(
	t,w',(\sigma,\zeta)
	\bigr)
	=
	\mathcal X_\alpha
	\left(
	t,
	A_{\alpha\beta}^{-1}
	\bigl(
	(w',\zeta)-c_{\alpha\beta}t
	\bigr),
	\sigma
	\right).
\end{equation}
Accordingly, before fixing either $\omega$ or $\sigma$, define
\begin{equation}\label{eq:reduction-random-sections}
\begin{aligned}
	u_\alpha(\omega,t,x,\sigma)
	&:=
	u\bigl(
	\omega,t,\mathcal X_\alpha(t,x,\sigma)
	\bigr),\\
	b_{\ell,\alpha}(t,x,\sigma,\lambda)
	&:=
	b_\ell
	\bigl(
	t,\mathcal X_\alpha(t,x,\sigma),\lambda
	\bigr).
\end{aligned}
\end{equation}
Thus $u_\alpha$ is the original solution written in the reduced coordinates,
while $b_{\ell,\alpha}$ is the original noise coefficient evaluated at the
same physical point $\mathcal X_\alpha(t,x,\sigma)$ and the same state value
$\lambda$. Because every reconstruction map is 
deterministic and leaves time unchanged,
Definition~\ref{def:kinetic-solution}\textup{(i)} implies that $u_\alpha$ is
$\mathcal P\otimes\cB(\R^{d_\alpha}\times\R^{r_\alpha})$-measurable.
Consequently, every subsequently reduced solution is jointly predictable in
$(\omega,t)$ and Borel measurable in $(x,\sigma)$, rather than a new
stochastic object introduced after fixing $\omega$.

\subsection{Pathwise martingale processes and rescaling}

Although their notation differs, all the martingales below are constructed
from the same Wiener processes $W_\ell$ and from $b_\ell$ or its reduced
pullbacks $b_{\ell,\alpha}$.  The notation records which variables have been
integrated out and which parameters remain.
The purpose of this subsection is to choose, before fixing $\omega$, continuous
versions of every martingale needed in the pathwise blow-up and in every
prescribed dimension reduction.  The scalar martingales
$M_\phi$ in \eqref{eq:scalar-martingale-pairing} and
$M_\varphi$ in \eqref{eq:time-dependent-scalar-martingale} integrate out
$x$ and use, respectively, a kinetic test function of $(x,\lambda)$ alone
and one that may also depend on $t$.  Their parameterized scalar counterparts
after reduction are
$M_{\alpha,\psi,J}$ and $Z_{\alpha,\phi,J}$, respectively, as defined in
\eqref{eq:reduction-scalar-martingales}.  In contrast, the script notation
$\mathscr M$ denotes an
$L^2_{\loc}$-valued martingale that retains the active spatial variable.
The reduction scheme uses three parameterized forms of these stochastic
integrals:
\begin{enumerate}[label=\textup{(\roman*)}]
	\item the spatial process $\mathscr M^\rho$, defined in
	\eqref{eq:global-spatial-martingale} from the local processes
	\eqref{eq:spatial-martingale-process}, leaves $x$ unintegrated and appears
	in the averaged kinetic equations
	\eqref{eq:pathwise-averaged-identity} and
	\eqref{eq:pathwise-averaged-identity-with-trace};

	\item the reduced spatial process
	$\mathscr M_{\alpha,J,K}^{\rho}$ in
	\eqref{eq:reduction-spatial-martingale} leaves $(\sigma,x)$ unintegrated,
	where $x\in\R^{d_\alpha}$ is the remaining active spatial variable and
	$\sigma\in\R^{r_\alpha}$ collects the transverse spatial parameters fixed
	in earlier reductions.  Its joint construction in $\sigma$ permits
	evaluation at a fixed transverse parameter after a reduction; the
	parameters retained in the pathwise argument are precisely
	$\sigma\in\mathcal A_\alpha(\omega)$, with
	$\mathcal A_\alpha(\omega)$ defined in
	\eqref{eq:recursive-admissible-parameters};

	\item the scalar processes $Z_{\alpha,\phi,J}$ and
	$M_{\alpha,\psi,J}$ in \eqref{eq:reduction-scalar-martingales} integrate
	in $x$.  The former uses a time-dependent test function
	$\phi\in\mathscr T_\alpha$ in the reduced kinetic identity, whereas the
	latter uses the constant-in-time extension
	$\widehat\psi(t,x,\lambda):=\psi(x,\lambda)$ of
	$\psi\in\mathscr D_\alpha$ in the time-integrated identity.
\end{enumerate}
The index $j$ restricts the original spatial martingale
\eqref{eq:spatial-martingale-process} to $B_j$; in the reduced processes
\eqref{eq:reduction-spatial-martingale}--\eqref{eq:reduction-scalar-martingales},
$J\Subset\R^{r_\alpha}$ and $K\Subset\R^{d_\alpha}$ restrict $\sigma$ and
the active variable $x$, respectively.
Equation~\eqref{eq:martingale-restriction} and the compatibility statement
following \eqref{eq:reduction-scalar-martingales} ensure consistency under
restriction.  The $\mathscr M$-processes retain $x$ and therefore enter the
spatial distributional equations and the blow-up argument; $Z$ and $M$
integrate out $x$ and provide the scalar martingale terms for fixed test
functions.

Let
$
\rho\in C_c^1((-L,L)).
$
For $j\geq1$, let
$
B_j:=\{x\in\R^d:\abs{x}<j\}
$
be the open ball of radius $j$ centered at the origin, and define the
$L^2(B_j)$-valued It\^o integral
\begin{equation}\label{eq:spatial-martingale-process}
	\mathscr M^\rho_j(t)
	:=
	\sum_{\ell\geq1}
	\int_0^t
	\rho\bigl(u(s,\cdot)\bigr)
	b_\ell\bigl(s,\cdot,u(s,\cdot)\bigr)
	\,dW_\ell(s).
\end{equation}
The stochastic integral is well defined because
\begin{align*}
	&\sum_{\ell\geq1}
	\E
	\int_0^T
	\norm{
	\rho(u(s,\cdot))
	b_\ell(s,\cdot,u(s,\cdot))
	}_{L^2(B_j)}^2
	\,ds
	\\
	&\qquad\leq
	T
	\abs{B_j}\norm{\rho}_\infty^2
	\sup_{(t,x,\lambda)\in[0,T]\times\overline B_j\times[-M,M]} \sum_{\ell\geq1}
		\abs{b_\ell(t,x,\lambda)}^2
	<\infty,
\end{align*}
{where we have used that $|u|\leq M$ a.e.~and Assumption \ref{ass:coeff}.}
Thus $\mathscr M^\rho_j$ admits a continuous $L^2(B_j)$-valued version
starting from zero.

If $k\geq j$, uniqueness of Hilbert-space-valued It\^o integrals gives
\begin{equation}\label{eq:martingale-restriction}
	\mathscr M^\rho_k(t)\big|_{B_j}
	=
	\mathscr M^\rho_j(t)
	\quad\text{in }L^2(B_j)
\end{equation}
for every $t\in[0,T]$, up to indistinguishability.  Since the pairs $(j,k)$
are countable, the continuous modifications can be fixed so that
\eqref{eq:martingale-restriction} holds simultaneously for every
$j\leq k$ and every $t$.  Define the unique process
\begin{equation}\label{eq:global-spatial-martingale}
\mathscr M^\rho
\in
C\bigl([0,T];L^2_{\loc}(\R^d)\bigr),
\qquad
\mathscr M^\rho(t)\big|_{B_j}
=
\mathscr M^\rho_j(t).
\end{equation}

Fix a countable family
$$
\mathcal R\subset C_c^1((-L,L))
$$
with the following properties:

\begin{enumerate}[label=\textup{(\roman*)}]
	\item on every compact interval $J\Subset(-L,L)$, the rational linear span
		of the functions in $\mathcal R$ supported in $J$ is dense in
		$C_c^1(J)$ in the $C^1$ norm;

	\item for every open interval $I\subset(-L,L)$ and every
	$[c,{e}]\Subset I$, there are a compact interval $I_0\Subset I$ and a
	sequence
	$$
	\rho_j^{c,{e},I}\in\mathcal R
	$$
	such that
	\begin{equation}\label{eq:cutoff-approximation}
		0\leq\rho_j^{c,{e},I}\leq1,
		\qquad
		\supp\rho_j^{c,{e},I}\subset I_0,
		\qquad
		\rho_j^{c,{e},I}
		\longrightarrow
		\1_{(c,{e})}
		\quad\text{in }L^1(\R).
	\end{equation}
\end{enumerate}

Such a family is obtained from countable dense families of smooth functions
on rational compact subintervals of $(-L,L)$, together with standard smooth
approximations of indicator functions.

For each $\rho\in\mathcal R$ and $j\geq1$, choose a jointly measurable
representative of $\mathscr M^\rho_j$ on
$\Omega\times[0,T]\times B_j$.  The sets
$B_j\setminus B_{j-1}$, with $B_0=\varnothing$, are disjoint and cover
$\R^d$.  On $B_j\setminus B_{j-1}$ use the representative chosen on $B_j$.
This countable pasting is jointly measurable.  By
\eqref{eq:martingale-restriction}, its restriction to every $B_k$ agrees
almost everywhere with $\mathscr M_k^\rho$; hence the particular annular
choice is immaterial at the level of $L^2_{\loc}(\R^d)$.  The purpose of the
pasting is to replace the compatible local equivalence classes by one
jointly measurable function on
$\Omega\times[0,T]\times\R^d$ for the later pathwise restrictions.  We have
therefore obtained a jointly measurable representative
\begin{equation}\label{eq:joint-martingale-representative}
	(\omega,t,x)
	\longmapsto
	\mathscr M^\rho(\omega,t,x)
\end{equation}
of the $L^2_{\loc}(\R^d)$-valued process for every
$\rho\in\mathcal R$.

For each reduction datum $\alpha\in\mathfrak R$, we use the martingale
integrand in
\eqref{eq:spatial-martingale-process}, expressed in the reduced coordinates.
At integration time $r$, the original spatial point is
$\mathcal X_\alpha(r,x,\sigma)$; hence
\eqref{eq:reduction-random-sections} replaces
$u(r,\cdot)$ and $b_\ell(r,\cdot,u(r,\cdot))$ by
$u_\alpha(r,x,\sigma)$ and
$b_{\ell,\alpha}(r,x,\sigma,u_\alpha(r,x,\sigma))$, respectively.  The
spatial martingale in \eqref{eq:reduction-spatial-martingale} and the scalar
martingales in \eqref{eq:reduction-scalar-martingales} are constructed
jointly in the full parameter $\sigma\in\R^{r_\alpha}$.  Let
$
\rho\in\mathcal R
$
have support compactly contained in $I_\alpha$.  For compact boxes
$
J\Subset\R^{r_\alpha}
$
and
$
K\Subset\R^{d_\alpha},
$
set, for $(\sigma,x)\in J\times K$,
\begin{align}
	\mathscr M_{\alpha,J,K}^{\rho}(t)(\sigma,x)
	:=
	\sum_{\ell\geq1}
	\int_0^t
	&\rho
	\bigl(
	u_\alpha(r,x,\sigma)
	\bigr)
	%\notag\\ &\times
	b_{\ell,\alpha}
	\bigl(
	r,x,\sigma,u_\alpha(r,x,\sigma)
	\bigr)
	\,dW_\ell(r).
	\label{eq:reduction-spatial-martingale}
\end{align}
Here and below the $\sigma$ variable and its integration are absent when
$r_\alpha=0$.  By \eqref{eq:noise-local-bound} and
\eqref{eq:reduction-random-sections}, the right-hand side of
\eqref{eq:reduction-spatial-martingale} defines a process in
\begin{equation}\label{eq:reduction-spatial-martingale-space}
	C
	\bigl(
	[0,T];L^2(J\times K)
	\bigr).
\end{equation}
Set
$C_{\alpha,J,K}:=\{(t,\mathcal X_\alpha(t,x,\sigma)):
(t,x,\sigma)\in[0,T]\times K\times J\}$.
This set is compact, and the sum of the expected squared Hilbert--Schmidt
norms of the integrands is bounded by
$$
T\abs{J}\abs{K}\norm{\rho}_\infty^2
\sup_{(t,y,\lambda)\in C_{\alpha,J,K}\times[-M,M]} \sum_{\ell\geq1}
\abs{b_\ell(t,y,\lambda)}^2
<\infty
$$
by Assumption~\ref{ass:coeff}, specifically the zeroth-order term in
\eqref{eq:noise-local-bound}, applied to the compact spatial projection of
$C_{\alpha,J,K}$ with $R=M$.

For $\psi\in\mathscr D_\alpha$, define its constant-in-time extension to
space--time by
$\widehat\psi(t,x,\lambda):=\psi(x,\lambda)$ for $t\in[0,T]$, and set
$\widehat{\mathscr D}_\alpha
:=\{\widehat\psi:\psi\in\mathscr D_\alpha\}$.
For
$
\varphi
\in
\mathscr T_\alpha
\cup
\widehat{\mathscr D}_\alpha,
$
define
\begin{equation}\label{eq:reduction-scalar-integrand}
	\mathcal I_{\ell,\alpha,\varphi}(\omega,t,\sigma)
	:=
	\int_{\R^{d_\alpha}}
	b_{\ell,\alpha}
	\bigl(
	t,x,\sigma,u_\alpha(\omega,t,x,\sigma)
	\bigr)
	\varphi
	\bigl(
	t,x,u_\alpha(\omega,t,x,\sigma)
	\bigr)
	\,dx.
\end{equation}
For
$\phi\in\mathscr T_\alpha$
and
$\psi\in\mathscr D_\alpha$,
the same estimate, with the compact spatial supports of $\phi$ and $\psi$,
gives continuous $L^2(J)$-valued martingales
\begin{equation}\label{eq:reduction-scalar-martingales}
	\begin{split}
		Z_{\alpha,\phi,J}(t)(\sigma)
		&:=
		\sum_{\ell\geq1}
		\int_0^t
		\mathcal I_{\ell,\alpha,\phi}(r,\sigma)
		\,dW_\ell(r),
		\\
		M_{\alpha,\psi,J}(t)(\sigma)
		&:=
		\sum_{\ell\geq1}
		\int_0^t
		\mathcal I_{\ell,\alpha,\widehat\psi}(r,\sigma)
		\,dW_\ell(r).
	\end{split}
\end{equation}
The parameterized It\^o construction in the proof of
Lemma~\ref{lem:pathwise-slice-equation} yields jointly measurable
representatives of
\eqref{eq:reduction-spatial-martingale} and
\eqref{eq:reduction-scalar-martingales}.  For 
$(\Prob\otimes\mathcal L^{r_\alpha})$-almost every 
$(\omega,\sigma)$, their $\sigma$-sections have, respectively,
continuous paths in $L^2(K)$ and in $\R$.  For prescribed boxes
$\widetilde J\subset J$ and $\widetilde K\subset K$, uniqueness of the
It\^o integrals allows these versions to be chosen so that, up to
indistinguishability,
\begin{align}
	\left.
	\mathscr M_{\alpha,J,K}^{\rho}
	\right|_{\widetilde J\times\widetilde K}
	&=
	\mathscr M_{\alpha,\widetilde J,\widetilde K}^{\rho},
	\notag\\
	\left.
	Z_{\alpha,\phi,J}
	\right|_{\widetilde J}
	&=
	Z_{\alpha,\phi,\widetilde J},
	\qquad
	\left.
	M_{\alpha,\psi,J}
	\right|_{\widetilde J}
	=
	M_{\alpha,\psi,\widetilde J}.
	\label{eq:reduction-martingale-restrictions}
\end{align}

We make the recursive parameter selection explicit.  For fixed $\omega$, let
$\mathcal C_\alpha(\omega)\subset\R^{r_\alpha}$ be the set of $\sigma$ for
which all of the following prescribed properties hold:
\begin{enumerate}[label=\textup{(\alph*)}]
	\item the representatives in
	\eqref{eq:reduction-spatial-martingale} and
	\eqref{eq:reduction-scalar-martingales} have continuous fixed-$\sigma$
	paths in $L^2(K)$ and $\R$, respectively, and satisfy the restriction
	identities for all localization boxes containing $\sigma$;

	\item for every one-step reduction $\alpha\to\beta$, the fixed-$\sigma$
	form of \eqref{eq:scalar-stochastic-fubini} holds for almost every
	$\zeta$, the resulting outgoing scalar sections, including their
	versions defined using constant-in-time test-function extensions, belong
	to $C([0,T];L^2(J))$, agree under
	restriction of $J_\alpha$ and $J$, and have the fixed-$\sigma$ form of
	\eqref{eq:mollified-stochastic-fubini} in $L^2(J_0)$ whenever
	$J_0\Subset J$;

	\item for every one-step reduction $\alpha\to\beta$, every prescribed
	compact interval $J\Subset\R$, and every prescribed compact box
	$K'\Subset\R^{d_\beta}$, the fixed-$\sigma$ section of
	\eqref{eq:sliced-spatial-martingale} belongs to
	$C([0,T];L^2(J\times K'))$, its fixed-$(\sigma,\zeta)$ section belongs
	to $C([0,T];L^2(K'))$ for almost every $\zeta\in J$, and all prescribed
	restriction identities hold.
\end{enumerate}
Here $J$ localizes the next transverse coordinate $\zeta$, whereas $K'$
localizes the remaining active coordinate $w'\in\R^{d_\beta}$.
Only prescribed accumulated-parameter boxes $J_\alpha$ containing $\sigma$
are relevant in these clauses; \textup{(b)}--\textup{(c)} are vacuous when
$d_\alpha=0$.
Here the outgoing processes in \textup{(b)}--\textup{(c)} are the jointly
measurable processes constructed in the proof of
Lemma~\ref{lem:pathwise-slice-equation}; the references are forward because
those constructions must be completed before $\omega$ is fixed.  The graph
$$
\left\{
(\omega,\sigma):
\sigma\in\mathcal C_\alpha(\omega)
\right\}
$$
is $\cF\otimes\cB(\R^{r_\alpha})$-measurable.  Indeed, continuity and
restriction can be expressed by countably many measurable conditions.
More precisely, let $X(\omega,t,\sigma)$ denote any prescribed jointly
measurable representative with values in a separable Hilbert space $H$.
Its rational-time values admit a continuous $H$-valued extension precisely
when
\begin{equation*}
	\lim_{n\to\infty}
	\sup_{\substack{s,t\in\mathbb Q\cap[0,T]\\ \abs{s-t}<1/n}}
	\norm{X(\omega,t,\sigma)-X(\omega,s,\sigma)}_H
	=
	0.
\end{equation*}
For two such continuous extensions, a prescribed restriction identity is
equivalent to
\begin{equation*}
	\sup_{t\in\mathbb Q\cap[0,T]}
	\norm{X_1(\omega,t,\sigma)-X_2(\omega,t,\sigma)}_H
	=
	0.
\end{equation*}
Finally, if $E(\omega,\sigma,\zeta)$ is a jointly measurable nonnegative
error for one of the assertions in \textup{(b)}--\textup{(c)}, extended by
zero outside its prescribed localization, then that assertion holds for
almost every $\zeta$ precisely when
\begin{equation*}
	\int_{-n}^n
	\bigl(1\wedge E(\omega,\sigma,\zeta)\bigr)
	\,d\zeta
	=
	0
	\qquad\text{for every }n\geq1.
\end{equation*}
The rational times, reduction data, test functions, localizations, and error
conditions form countable families.  The graph of $\mathcal C_\alpha$ is
therefore a countable intersection of measurable sets.

Recall that $\alpha\in\mathfrak R$ is a reduction datum,
$d_\alpha$ is its remaining active spatial dimension, and
$r_\alpha=d-d_\alpha$ is the number of accumulated transverse parameters.
The set $\mathcal C_\alpha(\omega)\subset\R^{r_\alpha}$ defined above
consists of the parameters $\sigma$ for which
\textup{(a)}--\textup{(c)} hold.  Define the \textit{recursively admissible}
sets $\mathcal A_\alpha(\omega)$ by induction in increasing $d_\alpha$:
\begin{equation}\label{eq:recursive-admissible-parameters}
	\mathcal A_\alpha(\omega)
	:=
	\mathcal C_\alpha(\omega)
	\cap
	\bigcap_{\beta:\,\alpha\to\beta}
	\left\{
		\sigma:
		\mathcal L^1
		\bigl(
		\R\setminus
		\{\zeta:(\sigma,\zeta)\in\mathcal A_\beta(\omega)\}
		\bigr)
		=
		0
	\right\},
\end{equation}
where the intersection ranges over all one-step reductions
$\alpha\to\beta$ (the prescribed steps from
Subsection~\ref{subsec:countable-reduction}, each with
$d_\beta=d_\alpha-1$ and one additional transverse coordinate).
If $d_\alpha=0$, there is no such reduction, so the
intersection is omitted and
$\mathcal A_\alpha(\omega)=\mathcal C_\alpha(\omega)\subset\R^d$.
This base case is distinct from the original datum $\alpha_0$, for which
$d_{\alpha_0}=d$ and $r_{\alpha_0}=0$.  If the graph of
$\mathcal A_\beta$ is measurable, then, for every $n\geq1$, the map
$$
(\omega,\sigma)
\longmapsto
\int_{-n}^n
\1_{\{(\sigma,\zeta)\notin\mathcal A_\beta(\omega)\}}
\,d\zeta
$$
is measurable.  Thus the zero-section condition in
\eqref{eq:recursive-admissible-parameters}, and hence the graph of
$\mathcal A_\alpha$, is measurable by induction.  This measurability permits
the Fubini induction below, which proves that every such graph has full
product measure and yields the full-probability event
$\Omega_{\mathrm{rec}}$ in
\eqref{eq:recursive-admissibility-event}.  Moreover, if $\omega$ is fixed
and $\sigma\in\mathcal A_\alpha(\omega)$, then
\eqref{eq:recursive-admissible-parameters} gives
$(\sigma,\zeta)\in\mathcal A_\beta(\omega)$ for almost every $\zeta\in\R$
and every one-step reduction $\alpha\to\beta$.
Lemma~\ref{lem:pathwise-slice-equation} uses this section property to
continue the reduction after $(\omega,\sigma)$ has been fixed, without
removing an additional probability-null set.

The graph of $\mathcal C_\alpha$ has full
$(\Prob\otimes\mathcal L^{r_\alpha})$-measure.  For \textup{(a)}, this follows
from the parameterized constructions
\eqref{eq:reduction-spatial-martingale} and
\eqref{eq:reduction-scalar-martingales}, the path-space property
\eqref{eq:reduction-spatial-martingale-space}, and the restriction identities
\eqref{eq:reduction-martingale-restrictions}.  For \textup{(b)} and
\textup{(c)}, the estimates
\eqref{eq:stochastic-fubini-BDG} and
\eqref{eq:fixed-parameter-spatial-martingale-bound}, followed by ordinary
Fubini, give the fixed-$\sigma$ forms of
\eqref{eq:scalar-stochastic-fubini},
\eqref{eq:mollified-stochastic-fubini}, and
\eqref{eq:sliced-spatial-martingale}.  The uniqueness arguments following
\eqref{eq:fixed-parameter-stochastic-fubini-bound} and
\eqref{eq:fixed-parameter-spatial-martingale-bound} show that these versions
coincide on overlaps when $J_\alpha$, $J$, or $K'$ is replaced by a smaller
prescribed localization; these are the restriction identities required in
\textup{(b)}--\textup{(c)}.  All data are countable.  Starting at
$d_\alpha=0$, Fubini's
theorem and
\eqref{eq:recursive-admissible-parameters} therefore show inductively that
the graph of every $\mathcal A_\alpha$ has full product measure.  Applying
Fubini only to these jointly measurable martingale representatives, and
using the countability of all prescribed data, gives
\begin{equation}\label{eq:recursive-admissibility-event}
	\Omega_{\mathrm{rec}}
	:=
	\left\{
		\omega:
		\mathcal L^{r_\alpha}
		\bigl(
		\R^{r_\alpha}\setminus\mathcal A_\alpha(\omega)
		\bigr)
		=
		0
		\text{ for every }\alpha\in\mathfrak R
	\right\},
	\qquad
	\Prob(\Omega_{\mathrm{rec}})=1.
\end{equation}
Here $\R^0=\{\varnothing\}$, so
\eqref{eq:recursive-admissibility-event} implies, for every
$\omega\in\Omega_{\mathrm{rec}}$,
$\varnothing\in\mathcal A_{\alpha_0}(\omega)$.  No kinetic measure or weak
trace indexed by $\sigma$ is selected in
\eqref{eq:recursive-admissible-parameters}.

For each $\rho\in\mathcal R$, choose a countable family
$
\mathscr E_\rho\subset C_c^1([0,T)\times\R^d)
$
whose rational span is dense in the $C^1$ norm on every fixed compact
space--time cylinder.  Let $\mathscr T_{\mathrm{kin}}$ be the rational span
of the tensor products
$
\eta(t,x)\rho(\lambda)
$
with
$
\rho\in\mathcal R
$
and
$
\eta\in\mathscr E_\rho.
$
This is the countable test family used to define
\begin{equation}\label{eq:countable-kinetic-event}
	\Omega_{\mathrm{kin}}
	:=
	\bigcap_{\varphi\in\mathscr T_{\mathrm{kin}}}
	\Omega_{\mathrm{tr},\varphi}.
\end{equation}
By Lemma~\ref{lem:kinetic-with-trace} and the countability of
$\mathscr T_{\mathrm{kin}}$,
$\Prob(\Omega_{\mathrm{kin}})=1$.  For each fixed
$\rho\in\mathcal R$, continuity of the deterministic terms and of
\eqref{eq:spatial-martingale-process} extends the identity in the
space--time factor.  Consequently, after $\omega$ is fixed, the identity
holds for every
$
\eta(t,x)\rho(\lambda),
\quad
\eta\in C_c^1([0,T)\times\R^d),
\quad
\rho\in\mathcal R,
$
and therefore for all translations, dilations, and fixed linear pullbacks of
these tensors used in $\mathfrak R$.

We shall refer to \(\mathscr T_{\mathrm{kin}}\) and, at a reduced
level \(\alpha\), to \(\mathscr T_\alpha\), together with their extensions
in the space--time factor and the prescribed translations, dilations,
mollifications, and fixed linear pullbacks introduced above, as the
prescribed determining family. Its localization to an open interval
\(I\subset(-L,L)\) consists of those members whose
\(\lambda\)-support is compactly contained in \(I\). Whenever a pathwise
kinetic identity is said below to hold on \(I\), this means that it holds
for this localized determining family on the already fixed event; no new
exceptional set is removed.

We now separate the probabilistic choice of martingale representatives from
the subsequent pathwise argument.  Each continuous It\^o representative and
each restriction or stochastic Fubini identity is initially available only
outside a null set that may depend on the reduction datum, localization, or
test function.  The blow-up argument later fixes one $\omega$ and uses all
prescribed martingales simultaneously, so neither new representatives nor
new probability-null sets may be selected at that stage.  We therefore make
all choices now and collect their properties on one full-probability event.

The final two items below refer forward to constructions in
Lemma~\ref{lem:pathwise-slice-equation}, because those later constructions
must already be included in this event.  This is not circular: the
probabilistic construction of those processes uses only the predictable
solution, the deterministic coordinate maps, and
\eqref{eq:noise-local-bound}; membership in $\Omega_*$ is used only after
the representatives have been constructed.  The parameterized It\^o
identities provide jointly measurable stochastic
integrals and justify mollification in the next transverse coordinate
$\zeta$; the sliced processes are the resulting scalar and
$L^2(K')$-valued martingales after $\zeta$ is fixed.
Introduce three events:
\begin{enumerate}[label=\textup{(\alph*)}]
	\item $\Omega_{\mathrm{rep}}$ is the event on which the continuous
	representatives of \eqref{eq:spatial-martingale-process} and
	\eqref{eq:reduction-spatial-martingale}--\eqref{eq:reduction-scalar-martingales},
	together with their stated restriction identities, are fixed for all
	prescribed data;

	\item $\Omega_{\mathrm{Fub}}$ is the event on which the parameterized
	It\^o identities \eqref{eq:scalar-stochastic-fubini} and
	\eqref{eq:mollified-stochastic-fubini} hold in their stated product
	spaces;

	\item $\Omega_{\mathrm{sl}}$ is the event on which the sliced processes
	in \eqref{eq:sliced-time-martingale} and
	\eqref{eq:sliced-spatial-martingale} have their stated continuity and
	restriction properties.
\end{enumerate}
The product-space events $\Omega_{\mathrm{Fub}}$ and
$\Omega_{\mathrm{sl}}$ alone give fixed-$\sigma$ assertions only for almost
every $\sigma$.  The definition of $\mathcal C_\alpha$ and the recursion
\eqref{eq:recursive-admissible-parameters} are what retain those assertions
at every parameter subsequently fixed by the pathwise induction.
For every fixed reduction datum, localization box, test function, and
mollifier index, the Hilbert-space It\^o construction and the parameterized
construction in Lemma~\ref{lem:pathwise-slice-equation} fail only on a null
set.  These choices form a countable family, so the union of their
exceptional sets is null and each event above has probability one.  Hence
\begin{equation}\label{eq:martingale-common-event}
	\Omega_{\mathrm{mar}}
	:=
	\Omega_{\mathrm{rec}}
	\cap
	\Omega_{\mathrm{rep}}
	\cap
	\Omega_{\mathrm{Fub}}
	\cap
	\Omega_{\mathrm{sl}},
	\qquad
	\Prob(\Omega_{\mathrm{mar}})=1.
\end{equation}

Combining this event with the local-mass, weak-trace, and kinetic-identity
events gives
\begin{equation}\label{eq:common-pathwise-event}
	\Omega_*
	:=
	\Omega_m
	\cap
	\Omega_{\mathrm w}
	\cap
	\Omega_{\mathrm{kin}}
	\cap
	\Omega_{\mathrm{mar}}.
\end{equation}
Here $\Omega_m$, defined in \eqref{eq:local-mass-event}, gives simultaneous
local finiteness and the support property from
Definition~\ref{def:kinetic-measure}; $\Omega_{\mathrm w}$, defined in
\eqref{eq:weak-trace-event}, gives the weak trace
\eqref{eq:pathwise-weak-trace}; and $\Omega_{\mathrm{kin}}$, defined in
\eqref{eq:countable-kinetic-event}, fixes
\eqref{eq:kinetic-with-trace} on the determining test function family.  The
event $\Omega_{\mathrm{mar}}$ in \eqref{eq:martingale-common-event} combines
recursive admissibility \eqref{eq:recursive-admissibility-event}, including
the fixed-$\sigma$ properties defining $\mathcal C_\alpha$, the
martingale restriction identity \eqref{eq:martingale-restriction}, the
stochastic Fubini identities \eqref{eq:scalar-stochastic-fubini} and
\eqref{eq:mollified-stochastic-fubini}, and the sliced processes
\eqref{eq:sliced-time-martingale} and
\eqref{eq:sliced-spatial-martingale}.  Their probability-one assertions
follow, respectively, from \eqref{eq:expected-measure-bound} and a countable
exhaustion, Lemma~\ref{lem:weak-trace}, the countability in
\eqref{eq:countable-kinetic-event}, and the countable martingale construction
preceding \eqref{eq:martingale-common-event}.  Their finite intersection
therefore satisfies 
$$
\Prob(\Omega_*)=1.
$$

For $\omega\in\Omega_*$ and $\rho\in\mathcal R$, define
\begin{equation*}
	F^\rho(t,x)
	:=
	\int_{-L}^L
	h(\omega,t,x,\lambda)\rho(\lambda)
	(1,a(\lambda))
	\,d\lambda
\end{equation*}
and, for Borel sets $E\subset(0,T)\times\R^d$,
\begin{equation*}
	\left\langle m^\omega,\rho'\right\rangle_\lambda(E)
	:=
	\int_{E\times(-L,L)}
	\rho'(\lambda)
	\,dm(\omega;t,x,\lambda).
\end{equation*}
Here the subscript $\lambda$ denotes integration only in the kinetic variable;
thus
$
\left\langle m^\omega,\rho'\right\rangle_\lambda
$
is a signed Radon measure on $(0,T)\times\R^d$.  For every
$\omega\in\Omega_*$ and $\rho\in\mathcal R$,
the kinetic identity with the weak trace
\eqref{eq:kinetic-with-trace}, restricted to
test functions supported in $(0,T)\times\R^d\times\R$, first on the fixed
countable determining family and then by continuity of every distributional
term, yields
\begin{equation}\label{eq:pathwise-averaged-identity}
	\diver_{(t,x)}F^\rho
	=
	\partial_t\mathscr M^\rho
	+
	\frac12
	\rho'(u)G^2(t,x,u)
	-
	\left\langle m^\omega,\rho'\right\rangle_\lambda
\end{equation}
in $\mathcal D'((0,T)\times\R^d)$.  For the martingale term, this extension
uses
\begin{equation}\label{eq:pathwise-martingale-pairing}
	\int_0^T
	\left\langle\eta(t),d\mathscr M^\rho(t)\right\rangle
	=
	-
	\int_0^T
	\left\langle
	\partial_t\eta(t),\mathscr M^\rho(t)
	\right\rangle
	\,dt
\end{equation}
for $\eta\in C_c^1((0,T)\times\R^d)$.  Define
$$
H_0^\rho(x)
:=
\int_{-L}^L
h_0^\omega(x,\lambda)\rho(\lambda)
\,d\lambda.
$$
For every $\omega\in\Omega_*$ and $\rho\in\mathcal R$, repeating the
derivation with \eqref{eq:kinetic-with-trace} gives, for every
$\eta\in C_c^1([0,T)\times\R^d)$,
\begin{align}
	&\int_0^T\int_{\R^d}
	F^\rho(t,x)\cdot\nabla_{(t,x)}\eta(t,x)
	\,dx\,dt
	+
	\int_{\R^d}
	H_0^\rho(x)\eta(0,x)
	\,dx
	\notag\\
	&\quad=
	\int_0^T\int_{\R^d}
	\mathscr M^\rho(t,x)\partial_t\eta(t,x)
	\,dx\,dt
	-
	\frac12
	\int_0^T\int_{\R^d}
	\eta(t,x)\rho'(u)G^2(t,x,u)
	\,dx\,dt
	\notag\\
	&\qquad+
	\int_{(0,T)\times\R^d\times(-L,L)}
	\eta(t,x)\rho'(\lambda)
	\,dm(\omega;t,x,\lambda).
	\label{eq:pathwise-averaged-identity-with-trace}
\end{align}
Thus \eqref{eq:pathwise-averaged-identity} and
\eqref{eq:pathwise-averaged-identity-with-trace} may be pulled back by every
translation and dilation associated with every reduction datum
$\alpha\in\mathfrak R$, without removing another probability-null set.

From this point through Section~\ref{sec:pathwise-identification}, fix a
realization $\omega$ in the full-probability event $\Omega_*$ defined by
\eqref{eq:common-pathwise-event}, and suppress it from the notation.  No
further probability-null set is removed.  Subsequences, Lebesgue exceptional
sets, admissible transverse coordinates, weak traces, and the Radon measures
$\widetilde m_\zeta^\omega$ may
depend on this fixed realization.

For $y\in\R^d$ and $\varepsilon>0$, define
\begin{equation*}
	\begin{aligned}
		u^{\varepsilon,y}(\tau,z)
		&:=
		u(\omega,\varepsilon\tau,y+\varepsilon z),
		\\
		h^{\varepsilon,y}(\tau,z,\lambda)
		&:=
		h(\omega,\varepsilon\tau,y+\varepsilon z,\lambda).
	\end{aligned}
\end{equation*}
Then
$h^{\varepsilon,y}(\tau,z,\lambda)
=\1_{\{\lambda<u^{\varepsilon,y}(\tau,z)\}}$.
These functions are defined for $0<\tau<T/\varepsilon$.
All arguments are local in $\tau$, so on every fixed bounded cylinder they
are defined for all sufficiently small $\varepsilon$.

For the fixed realization, write
$
m^\omega:=m(\omega;\cdot)
$
and define the scaling map
$
S_{\varepsilon,y}(\tau,z,\lambda)
:=
\bigl(
\varepsilon\tau,
y+\varepsilon z,
\lambda
\bigr).
$
The rescaled kinetic measure is
$$
\nu^{\varepsilon,y}
=
\nu^{\varepsilon,y}(\omega)
:=
\frac{1}{\varepsilon^d}
\bigl(
S_{\varepsilon,y}^{-1}
\bigr)_\#
m^\omega.
$$
Equivalently, it is characterized by
\begin{equation}\label{eq:pathwise-scaled-measure}
	\int
	\psi(\tau,z,\lambda)
	\,d\nu^{\varepsilon,y}(\tau,z,\lambda)
	:=
	\frac{1}{\varepsilon^d}
	\int
	\psi\left(
	\frac{t}{\varepsilon},
	\frac{x-y}{\varepsilon},
	\lambda
	\right)
	\,dm(\omega;t,x,\lambda)
\end{equation}
for
$\psi\in C_c((0,\infty)\times\R^d\times\R)$.
For $\rho\in\mathcal R$, set
\begin{equation*}
	\cM_{\varepsilon,y}^\rho(\tau,z)
	:=
	\mathscr M^\rho(\varepsilon\tau,y+\varepsilon z),
\end{equation*}
where $\mathscr M^\rho$ is the jointly measurable representative in
\eqref{eq:joint-martingale-representative}.

The next lemma extracts one scale sequence on which the rescaled measure,
martingale processes, and weak trace have simultaneous limits.

\begin{lemma}[Pathwise blow-up subsequences]
\label{lem:pathwise-blowup-subsequence}
Fix a realization $\omega\in\Omega_*$, where $\Omega_*$ is defined in
\eqref{eq:common-pathwise-event}, and let
$
\varepsilon_n\downarrow0.
$
There exist a subsequence, not relabeled, and a set
$$
Y_\omega\subset\R^d,
\qquad
\mathcal L^d(\R^d\setminus Y_\omega)=0,
$$
such that, for every $y\in Y_\omega$, every $R>0$, and every
$\rho\in\mathcal R$, the following convergences hold as $n\to\infty$:
\begin{align}
	\nu^{\varepsilon_n,y}
	\bigl(
	(0,R)\times B_R\times(-L,L)
	\bigr)
	&\longrightarrow0,
	\label{eq:pathwise-defect-vanish}
	\\
	\cM_{\varepsilon_n,y}^\rho
	&\longrightarrow0
	\quad\text{in }
	L^2\bigl((0,R)\times B_R\bigr),
	\label{eq:pathwise-martingale-vanish}
	\\
	h_0^\omega(y+\varepsilon_nz,\lambda)
	&\longrightarrow
	h_0^\omega(y,\lambda)
	\quad\text{in }
	L^1\bigl(B_R\times(-L,L)\bigr).
	\label{eq:pathwise-trace-lebesgue}
\end{align}
The subsequence and the set $Y_\omega$ may depend on $\omega$ and on the
original sequence $(\varepsilon_n)$.
\end{lemma}

\begin{proof}
Fix compact sets $K\Subset K'\Subset\R^d$ and $R>0$. For all sufficiently
small $\varepsilon$,
Fubini's theorem and \eqref{eq:pathwise-scaled-measure} give
\begin{equation}\label{eq:pathwise-defect-averaged-estimate}
\begin{aligned}
	&\int_K
	\nu^{\varepsilon,y}
	\bigl(
	(0,R)\times B_R\times(-L,L)
	\bigr)
	\,dy\\
	&\quad=
	\frac{1}{\varepsilon^d}
	\int_{(0,\varepsilon R)\times\R^d\times[-L,L]}
	\mathcal L^d
	\bigl(
	K\cap B_{\varepsilon R}(x)
	\bigr)
	\,dm(\omega;t,x,\lambda)\\
	&\quad\leq
	\abs{B_R}\,
	m\bigl(
	\omega;
	(0,\varepsilon R)\times K'\times[-L,L]
	\bigr).
\end{aligned}
\end{equation}
The measure on the right-hand side of
\eqref{eq:pathwise-defect-averaged-estimate} tends to zero by continuity
from above and the bound
$
m(\omega;(0,T)\times K'\times[-L,L])<\infty,
$
which follows from \eqref{eq:local-mass-event}.
Extract a subsequence for which the left-hand side of
\eqref{eq:pathwise-defect-averaged-estimate} is summable.  Fubini's theorem
then gives
\eqref{eq:pathwise-defect-vanish} for almost every $y\in K$.

Applying Fubini's theorem to
$\mathscr M^\rho(\varepsilon\tau,y+\varepsilon z)$ gives
\begin{align*}
	&\int_K
	\norm{
	\cM_{\varepsilon,y}^\rho
	}_{L^2((0,R)\times B_R)}^2
	\,dy
	\\
	&\quad=
	\int_0^R\int_{B_R}\int_K
	\abs{
	\mathscr M^\rho
	\bigl(
	\varepsilon\tau,
	y+\varepsilon z
	\bigr)
	}^2
	\,dy\,dz\,d\tau
	\\
	&\quad\leq
	\abs{B_R}
	\int_0^R
	\norm{
	\mathscr M^\rho(\varepsilon\tau)
	}_{L^2(K')}^2
	\,d\tau.
\end{align*}
For the fixed realization, the right-hand side tends to zero because
$\mathscr M^\rho\in C([0,T];L^2(K'))$ and
$\mathscr M^\rho(0)=0$.
A further summable extraction gives
\eqref{eq:pathwise-martingale-vanish} for almost every $y\in K$.

Finally, regard $x\mapsto h_0^\omega(x,\cdot)$ as a locally integrable
map with values in the separable Banach space
$L^1(-L,L)$.  At every Bochner--Lebesgue point $y$ of this map,
\begin{align*}
	&\int_{B_R}\int_{-L}^L
	\abs{
	h_0^\omega(y+\varepsilon z,\lambda)
	-
	h_0^\omega(y,\lambda)
	}
	\,d\lambda\,dz
	\\
	&\quad=
	\frac{1}{\varepsilon^d}
	\int_{B_{\varepsilon R}(y)}
	\norm{
	h_0^\omega(x,\cdot)
	-
	h_0^\omega(y,\cdot)
	}_{L^1(-L,L)}
	\,dx
	\longrightarrow0.
\end{align*}
This proves \eqref{eq:pathwise-trace-lebesgue} for almost every $y$ without
extraction.

A diagonal summable extraction over a countable exhaustion of space, positive
rational values of $R$, and $\rho\in\mathcal R$ gives one subsequence for
which \eqref{eq:pathwise-defect-vanish},
\eqref{eq:pathwise-martingale-vanish}, and
\eqref{eq:pathwise-trace-lebesgue} hold simultaneously.  Each assertion for
an arbitrary $R>0$ follows by comparison with a larger rational radius.
\end{proof}

\subsection{Panov's condition for the rescaled truncations}

Set $\Phi(r):=(r,f(r))\in\R^{d+1}$, so that
$\Phi'(r)=(1,a(r))$, where $a=f'$.

We verify Panov's condition \eqref{eq:panov-C} for every truncated blow-up by
using the averaged distributional identity
\eqref{eq:pathwise-averaged-identity}.

\begin{proposition}[Panov's condition for a truncated blow-up]
\label{prop:pathwise-condition-C}
Fix a realization $\omega\in\Omega_*$, where $\Omega_*$ is defined in
\eqref{eq:common-pathwise-event}.  Let $I\subset(-L,L)$ be open and let
$[a,b]\Subset I$.
Assume that \eqref{eq:kinetic-with-trace} holds on \(I\) for the
prescribed determining family.  Let
$(\varepsilon_n)$ and $Y_\omega$ be as in
Lemma~\ref{lem:pathwise-blowup-subsequence}, fix $y\in Y_\omega$, and set
$v_n(\tau,z):=T_{a,b}(u^{\varepsilon_n,y}(\tau,z))$.
Then, for every $k\in\R$,
\begin{equation}\label{eq:pathwise-panov-production}
	\diver_{(\tau,z)}
	\left[
	\1_{\{v_n>k\}}
	\bigl(
	\Phi(v_n)-\Phi(k)
	\bigr)
	\right]
\end{equation}
is precompact in $H^{-1}_{\loc}((0,\infty)\times\R^d)$.
Indeed, the distributions in
\eqref{eq:pathwise-panov-production} converge to zero locally in
$H^{-1}$.
\end{proposition}

\begin{proof}
Write $h_n:=h^{\varepsilon_n,y}$.
For $k\in[a,b]$, one has
\begin{equation}\label{eq:pathwise-truncated-flux-identity}
	\1_{\{v_n>k\}}
	\bigl(
	\Phi(v_n)-\Phi(k)
	\bigr)
	=
	\int_k^b
	h_n(\tau,z,\lambda)
	\Phi'(\lambda)
	\,d\lambda.
\end{equation}
For $k<a$, the vector-valued expression in
\eqref{eq:pathwise-panov-production} differs by a constant vector from
\begin{equation}\label{eq:pathwise-Phi-vn}
\Phi(v_n)
=
\Phi(a)
+
\int_a^b
h_n(\tau,z,\lambda)
\Phi'(\lambda)
\,d\lambda,
\end{equation}
whereas it vanishes for $k\geq b$.  It is therefore enough to treat
velocity averages over compact subintervals of $I$.

Let $\rho\in\mathcal R$ satisfy $\supp\rho\Subset I$, and define
$$
F_n^\rho(\tau,z)
:=\int_I
h_n(\tau,z,\lambda)
\rho(\lambda)
\Phi'(\lambda)
\,d\lambda.
$$
For a Borel set $E$ in the $(\tau,z)$-variables, write
$$
\left\langle
\nu^{\varepsilon_n,y},
\rho'
\right\rangle_\lambda(E)
:=
\int_{E\times I}
\rho'(\lambda)
\,d\nu^{\varepsilon_n,y}(\tau,z,\lambda).
$$

Pulling back \eqref{eq:pathwise-averaged-identity} by
$(t,x)=(\varepsilon_n\tau,y+\varepsilon_nz)$
(i.e., testing it with
$
\eta(t/\varepsilon_n,(x-y)/\varepsilon_n)
$,
changing variables to $(\tau,z)$, and 
dividing by $\varepsilon_n^d$) gives
\begin{align}
	\diver_{(\tau,z)}F_n^\rho
	&=
	\partial_\tau
	\cM_{\varepsilon_n,y}^\rho
	+
	\frac{\varepsilon_n}{2}
	\rho'
	\bigl(
	u^{\varepsilon_n,y}
	\bigr)
	G^2
	\bigl(
	\varepsilon_n\tau,
	y+\varepsilon_nz,
	u^{\varepsilon_n,y}
	\bigr)
	-
	\left\langle
	\nu^{\varepsilon_n,y},
	\rho'
	\right\rangle_\lambda
	\label{eq:rescaled-averaged-identity}
\end{align}
in the sense of distributions.

Fix a bounded cylinder $Q\Subset(0,\infty)\times\R^d$.
Define on $Q$
\begin{align*}
	A_n
	&:=
	\partial_\tau
	\cM_{\varepsilon_n,y}^\rho
	+
	\frac{\varepsilon_n}{2}
	\rho'\bigl(u^{\varepsilon_n,y}\bigr)
	G^2\bigl(
	\varepsilon_n\tau,
	y+\varepsilon_nz,
	u^{\varepsilon_n,y}
	\bigr),
	\\
	B_n
	&:=
	-
	\left\langle
	\nu^{\varepsilon_n,y},
	\rho'
	\right\rangle_\lambda,
	\\
	D_n
	&:=
	\diver_{(\tau,z)}F_n^\rho
	=
	A_n+B_n.
\end{align*}
By \eqref{eq:pathwise-martingale-vanish},
$$
\partial_\tau
\cM_{\varepsilon_n,y}^\rho
\longrightarrow0
\quad\text{in }H^{-1}(Q).
$$
The second summand in $A_n$ tends to zero in $L^2(Q)$ by the factor
$\varepsilon_n$ and \eqref{eq:noise-local-bound}.  Hence
$
A_n\to0
$
in $H^{-1}(Q)$.  Moreover,
$$
\norm{B_n}_{\cM(Q)}
\leq
\norm{\rho'}_\infty
\nu^{\varepsilon_n,y}\bigl(Q\times(-L,L)\bigr)
\longrightarrow0
$$
by \eqref{eq:pathwise-defect-vanish}.

The preceding two convergences verify
\eqref{eq:murat-zero-assumptions} with $F_n=D_n$: indeed,
$A_n\to0$ in $H^{-1}(Q)$ and $B_n\to0$ even in the measure norm.
Since $F_n^\rho$ is uniformly bounded, the sequence
$D_n=\diver_{(\tau,z)}F_n^\rho$ is bounded in $W^{-1,\infty}(Q)$.
Lemma~\ref{lem:murat-interpolation} therefore gives
$$
\diver_{(\tau,z)}F_n^\rho
\longrightarrow0
\quad\text{in }H^{-1}(Q).
$$

Now let
$
[c,{e}]\Subset I
$
and choose the approximations
$\rho_j^{c,{e},I}$ from
\eqref{eq:cutoff-approximation}.  Set
$$
F_n^{c,{e}}
:=
\int_c^{{e}}
h_n(\lambda)\Phi'(\lambda)
\,d\lambda.
$$
Uniformly in $n$,
\begin{align*}
	\norm{
	F_n^{\rho_j^{c,{e},I}}
	-
	F_n^{c,{e}}
	}_{L^\infty(Q)}
	&\leq
	\norm{\Phi'}_{L^\infty(I_0)}
	\norm{
	\rho_j^{c,{e},I}
	-
	\1_{(c,{e})}
	}_{L^1(\R)}
	\longrightarrow0.
\end{align*}
Consequently,
\begin{align*}
	\norm{
	\diver
	\left(
	F_n^{\rho_j^{c,{e},I}}
	-
	F_n^{c,{e}}
	\right)
	}_{H^{-1}(Q)}
	&\leq
	\norm{
	F_n^{\rho_j^{c,{e},I}}
	-
	F_n^{c,{e}}
	}_{L^2(Q)}
	\longrightarrow0
\end{align*}
as $j\to\infty$, uniformly in $n$.  Since, for every fixed $j$,
$\diver F_n^{\rho_j^{c,{e},I}}\to0$ in $H^{-1}(Q)$, we conclude that
$\diver F_n^{c,{e}}\to0$ in $H^{-1}(Q)$.

For $k\in[a,b)$, the convergence
$
\diver F_n^{k,b}\to0
$
in $H^{-1}(Q)$ and \eqref{eq:pathwise-truncated-flux-identity} give
\eqref{eq:pathwise-panov-production}.  The case $k=b$ is trivial.
For $k<a$, use \eqref{eq:pathwise-Phi-vn}; for $k>b$, the expression in
\eqref{eq:pathwise-panov-production} vanishes.  Since $Q$ was arbitrary,
the convergence asserted in
Proposition~\ref{prop:pathwise-condition-C} follows.
\end{proof}

\subsection{The weak limit of a blow-up}

The following proposition identifies every weak-$*$ blow-up limit directly
from the kinetic identity with the weak trace
\eqref{eq:kinetic-with-trace}.

\begin{proposition}[Pathwise weak limit of a blow-up]
\label{prop:pathwise-weak-blowup-limit}
Fix a realization \(\omega\in\Omega_*\), where \(\Omega_*\) is defined in
\eqref{eq:common-pathwise-event}, and let
\[
I\subseteq(-L,L)
\]
be an open interval. Assume that \eqref{eq:kinetic-with-trace} holds on
\(I\) for the prescribed determining family.

Let \((\varepsilon_n)\) and \(Y_\omega\) be as in
Lemma~\ref{lem:pathwise-blowup-subsequence}, fix \(y\in Y_\omega\), and
suppose that a subsequence of \(h^{\varepsilon_n,y}\) converges weak-\(*\)
in
\[
L^\infty_{\loc}
\bigl(
(0,\infty)\times\R^d\times I
\bigr)
\]
to a function \(\overline h\). Then
\begin{equation}\label{eq:pathwise-weak-blowup-constant}
	\overline h(\tau,z,\lambda)
	=
	h_0^\omega(y,\lambda)
\end{equation}
for almost every
\[
(\tau,z,\lambda)
\in
(0,\infty)\times\R^d\times I.
\]
In particular, taking \(I=(-L,L)\) gives the conclusion on the entire
kinetic interval.
\end{proposition}

\begin{proof}

Let \(I\) be the interval from the statement. First take a tensor-product
test function
$$
\psi(\tau,z,\lambda)
=
\eta(\tau,z)\rho(\lambda),
\qquad
\rho\in\mathcal R,
\qquad
\supp\rho\Subset I,
$$
where $\eta\in C_c^1([0,\infty)\times\R^d)$.  Pull back
\eqref{eq:pathwise-averaged-identity-with-trace}, or equivalently apply
\eqref{eq:kinetic-with-trace} with the test function
$$
\varphi_n(t,x,\lambda)
:=
\psi\left(
\frac{t}{\varepsilon_n},
\frac{x-y}{\varepsilon_n},
\lambda
\right).
$$
For all sufficiently large $n$, $\varphi_n$ is admissible in
\eqref{eq:kinetic-with-trace}.  Divide that equation by
$\varepsilon_n^d$ and perform the change of variables
$$
t=\varepsilon_n\tau,
\qquad
x=y+\varepsilon_nz.
$$

The boundary term becomes
$$
\int_{\R^d}\int_I
h_0^\omega(y+\varepsilon_nz,\lambda)
\psi(0,z,\lambda)
\,d\lambda\,dz
$$
and converges, by
\eqref{eq:pathwise-trace-lebesgue}, to
$$
\int_{\R^d}\int_I
h_0^\omega(y,\lambda)
\psi(0,z,\lambda)
\,d\lambda\,dz.
$$
The It\^o correction contains the factor $\varepsilon_n$ and therefore
tends to zero.  The kinetic-measure term tends to zero by
\eqref{eq:pathwise-defect-vanish}.

For the stochastic term, on a bounded spatial
set containing the support of $\eta$, the rescaled stochastic term can be
written as
$$
\int_0^\infty
\left\langle
\eta(\tau,\cdot),
d_\tau
\cM_{\varepsilon_n,y}^\rho(\tau,\cdot)
\right\rangle_{L^2_z}.
$$
Since $\eta$ is compactly supported in $\tau$ and
$\cM_{\varepsilon_n,y}^\rho(0,\cdot)=0$,
Hilbert-space integration by parts gives
$$
\int_0^\infty
\left\langle
\eta(\tau),
d_\tau
\cM_{\varepsilon_n,y}^\rho(\tau)
\right\rangle
=
-
\int_0^\infty
\left\langle
\partial_\tau\eta(\tau),
\cM_{\varepsilon_n,y}^\rho(\tau)
\right\rangle
\,d\tau.
$$
The right-hand side tends to zero by
\eqref{eq:pathwise-martingale-vanish}.

Passing to the weak-$*$ limit gives
\begin{align}
	&\int_0^\infty
	\int_{\R^d}
	\int_I
	\overline h(\tau,z,\lambda)
	\bigl(
	\partial_\tau\psi(\tau,z,\lambda)
	+
	a(\lambda)\cdot\nabla_z\psi(\tau,z,\lambda)
	\bigr)
	\,d\lambda\,dz\,d\tau
	\\
	&\qquad+
	\int_{\R^d}
	\int_I
	h_0^\omega(y,\lambda)
	\psi(0,z,\lambda)
	\,d\lambda\,dz
	=
	0
\label{eq:homogeneous-blowup-limit}
\end{align}
for the tensor-product determining family.  Density extends
\eqref{eq:homogeneous-blowup-limit} to
finite linear combinations and then to every test function in
$C_c^1\bigl(
[0,\infty)\times\R^d\times I
\bigr)$. 
Equation~\eqref{eq:homogeneous-blowup-limit} 
is the weak formulation, for 
almost every $\lambda\in I$, of
$$
\partial_\tau\overline h
+
a(\lambda)\cdot\nabla_z\overline h
=
0
$$
on $(0,\infty)\times\R^d$, with initial value
$\overline h(0,z,\lambda)=h_0^\omega(y,\lambda)$.
Set
$$
q(\tau,\zeta,\lambda)
:=
\overline h
\bigl(
\tau,
\zeta+a(\lambda)\tau,
\lambda
\bigr).
$$
Then $\partial_\tau q=0$ in the sense of distributions, and its boundary
value is the constant $h_0^\omega(y,\lambda)$. Hence
$q(\tau,\zeta,\lambda)=h_0^\omega(y,\lambda)$ almost everywhere, which proves
\eqref{eq:pathwise-weak-blowup-constant}.
\end{proof}

We shall also need the following deterministic consequence of a known strong
trace when the same scale sequence is viewed from almost every spatial
center.

\begin{lemma}[A pathwise strong trace and its blow-ups]
\label{lem:pathwise-trace-to-blowup}
Let $v\in L^\infty((0,T)\times\R^d)$ and
$v_0\in L^1_{\loc}(\R^d)$, and suppose that, for every compact
$K\Subset\R^d$,
$$
\esslim_{t\downarrow0}
\int_K
\abs{v(t,x)-v_0(x)}
\,dx
=
0.
$$
Given any sequence $\varepsilon_n\downarrow0$, one can extract a
subsequence such that, for almost every
$y\in\R^d$,
\begin{equation}\label{eq:pathwise-strong-trace-blowup}
	v(\varepsilon_n\tau,y+\varepsilon_nz)
	\longrightarrow
	v_0(y)
	\quad\text{in }
	L^1_{\loc}
	\bigl(
	(0,\infty)\times\R^d
	\bigr).
\end{equation}
The extraction can be made simultaneously for a countable family of pairs
$(v,v_0)$ satisfying these hypotheses.  This lemma is entirely deterministic.
\end{lemma}

\begin{proof}
Fix $K\Subset K'\Subset\R^d$ and $R>0$. Splitting the difference through
$v_0(y+\varepsilon_nz)$ and changing variables gives
\begin{align}
	&\int_K\int_0^R\int_{B_R}
	\abs{v(\varepsilon_n\tau,y+\varepsilon_nz)
	-v_0(y)}
	\,dz\,d\tau\,dy
	\notag\\
	&\quad\leq
	\abs{B_R}
	\int_0^R\int_{K'}
	\abs{v(\varepsilon_n\tau,x)-v_0(x)}
	\,dx\,d\tau
	\notag\\
	&\qquad+
	R
	\int_{B_R}\int_K
	\abs{
	v_0(y+\varepsilon_nz)
	-v_0(y)}
	\,dy\,dz.
	\label{eq:trace-blowup-averaged-estimate}
\end{align}
The first term is bounded by
$$
R\abs{B_R}
\operatorname*{ess\,sup}_{0<t<\varepsilon_nR}
\int_{K'}
\abs{v(t,x)-v_0(x)}
\,dx
$$
and therefore tends to zero.  The second term tends to zero by translation
continuity in $L^1_{\loc}(\R^d)$.

Choose a subsequence for which the left-hand side of
\eqref{eq:trace-blowup-averaged-estimate} is summable.  Fubini's theorem then
gives \eqref{eq:pathwise-strong-trace-blowup} for almost every $y\in K$.  A
diagonal extraction over a countable exhaustion by compact cylinders gives
\eqref{eq:pathwise-strong-trace-blowup} on all such cylinders; diagonalizing
once more over the prescribed countable family proves the final assertion of
Lemma~\ref{lem:pathwise-trace-to-blowup}.  See
\cite[Theorem~2]{Panov:2005lr} for the same extraction argument.
\end{proof}

%%%%%%%%%%%%%%%%%%%%%%%%%
%%%%%%%%%%%%%%%%%%%%%%%%%
\section{Pathwise dimension reduction on a degenerate interval}
\label{sec:pathwise-dimension-reduction}

Fix $\omega\in\Omega_*$ throughout this section, where $\Omega_*$ is the
full-probability event defined in \eqref{eq:common-pathwise-event}; no further
probability-null set is removed.  We use $\omega'\in\Omega$ only as the free
probability variable in the parameterized stochastic constructions that
must be completed before evaluation at the fixed $\omega$.

For an open interval \(I\subset(-L,L)\), we use the convention introduced
in Subsection~\ref{subsec:countable-reduction}: we say that
\eqref{eq:kinetic-with-trace} holds on \(I\) if it is valid for the
prescribed determining family localized to \(I\). The corresponding weak
trace is the restriction
\(
h_0^\omega\big|_{\R^d\times I}.
\)

\subsection{Linear reduction}

The next lemma converts a constant direction of the space--time flux into a
spatial direction with zero transformed transport velocity. To proceed, recall that
\[
a(\lambda):=f'(\lambda).
\]

\begin{lemma}[Linear reduction of a constant flux direction]
\label{lem:pathwise-linear-change}
Let $I\subset(-L,L)$ be an open interval. 
Suppose that
\(
(\xi_0,\xi)\in\R^{d+1}\setminus\{0\}
\)
and
\begin{equation}\label{eq:pathwise-directional-degeneracy}
	\xi_0+\xi\cdot a(\lambda)=0
	\qquad
	\text{for every }\lambda\in I.
\end{equation}
Then $\xi\neq0$. Moreover, there exist an invertible matrix
$A\in\R^{d\times d}$ and a vector $c\in\R^d$ such that, under the change
of variables
\begin{equation}\label{eq:pathwise-linear-change}
	w=Ax+ct,
\end{equation}
the transformed transport velocity
\[
\widetilde a(\lambda):=Aa(\lambda)+c
\]
satisfies
\[
\widetilde a_d(\lambda)=0
\qquad
\text{for every }\lambda\in I.
\] Moreover, the change of variables preserves Assumption~\ref{ass:coeff},
the pathwise kinetic identity on $I$, and the weak-trace property, with
the transformed coefficients, kinetic measure, kinetic function, and weak
trace defined in the proof.
\end{lemma}

\begin{proof}
The deterministic version of the reduction below is used in
\cite[Section~7, p.~904]{Panov:2005lr}; here we additionally verify its
compatibility with the stochastic kinetic formulation and the weak trace.

If $\xi=0$, then \eqref{eq:pathwise-directional-degeneracy} gives
$\xi_0=0$, contradicting $(\xi_0,\xi)\neq0$. Hence $\xi\neq0$.
Complete $\xi^\top$ to a basis of row vectors of $\R^d$, let
$A\in\R^{d\times d}$ be the corresponding invertible matrix with last
row $\xi^\top$, and set
\[
c:=\xi_0e_d.
\]
Since $a=f'$, the last component of
\[
\widetilde a(\lambda):=Aa(\lambda)+c
\]
is
\[
\widetilde a_d(\lambda)
=
\xi\cdot a(\lambda)+\xi_0
=
0
\qquad
\text{for every }\lambda\in I.
\]

Define
\begin{equation}\label{eq:transformed-solution}
\widetilde u(t,w)
:=
u\bigl(t,A^{-1}(w-ct)\bigr).
\end{equation}
The transformed flux may be taken as
\[
\widetilde f(r):=Af(r)+cr,
\]
so that
\[
\widetilde f'(r)=Af'(r)+c=Aa(r)+c=\widetilde a(r).
\]
For every $\ell\geq1$, define
\[
\widetilde b_\ell(t,w,\lambda)
:=
b_\ell\bigl(t,A^{-1}(w-ct),\lambda\bigr).
\]

We verify that the transformed coefficients satisfy
\eqref{eq:noise-local-bound}. Let $\widetilde K\Subset\R^d$ and set
\[
K
:=
\left\{
A^{-1}(w-ct):
w\in\widetilde K,\ t\in[0,T]
\right\}.
\]
Then $K\Subset\R^d$. Moreover,
\[
\partial_\lambda\widetilde b_\ell(t,w,\lambda)
=
\partial_\lambda b_\ell
\bigl(t,A^{-1}(w-ct),\lambda\bigr),
\]
and the chain rule gives
\[
\nabla_w\widetilde b_\ell(t,w,\lambda)
=
A^{-\top}
\nabla_x b_\ell
\bigl(t,A^{-1}(w-ct),\lambda\bigr).
\]
Consequently, with
\[
C_A:=\max\left\{1,\norm{A^{-\top}}^2\right\},
\]
we have, for every $R>0$,

\begin{align*}
&\sup_{(t,w,\lambda)\in[0,T]\times\widetilde K\times[-R,R]}
\sum_{\ell\geq1}
\left(
\abs{\widetilde b_\ell(t,w,\lambda)}^2
+
\abs{\partial_\lambda\widetilde b_\ell(t,w,\lambda)}^2
+
\abs{\nabla_w\widetilde b_\ell(t,w,\lambda)}^2
\right)
\\
&\qquad\leq
C_A
\sup_{(t,x,\lambda)\in[0,T]\times K\times[-R,R]}
\sum_{\ell\geq1}
\left(
\abs{b_\ell(t,x,\lambda)}^2
+
\abs{\partial_\lambda b_\ell(t,x,\lambda)}^2
+
\abs{\nabla_x b_\ell(t,x,\lambda)}^2
\right)
<\infty.
\end{align*}
Thus the transformed coefficients satisfy
\eqref{eq:noise-local-bound}.

Let
\[
\Theta(t,x,\lambda):=(t,Ax+ct,\lambda).
\]
Define the transformed kinetic measure by
\[
\widetilde m^\omega
:=
\abs{\det A}\,\Theta_\#m^\omega;
\]
equivalently,
\[
\int
\psi(t,w,\lambda)
\,d\widetilde m^\omega(t,w,\lambda)
=
\abs{\det A}
\int
\psi(t,Ax+ct,\lambda)
\,dm(\omega;t,x,\lambda)
\]
for every
\[
\psi
\in
C_c\bigl((0,T)\times\R^d\times\R\bigr).
\]
Since $\Theta$ is a homeomorphism preserving the time and kinetic
variables, $\widetilde m^\omega$ is again a nonnegative Radon measure
with the corresponding local finiteness and kinetic-support properties.

Define
\[
\widetilde h(t,w,\lambda)
:=
\1_{\{\lambda<\widetilde u(t,w)\}}
=
h\bigl(t,A^{-1}(w-ct),\lambda\bigr)
\]
and
\[
\widetilde G^2(t,w,\lambda)
:=
\sum_{\ell\geq1}
\abs{\widetilde b_\ell(t,w,\lambda)}^2.
\]
We also set
\begin{equation}\label{eq:transformed-weak-trace}
\widetilde h_0^\omega(w,\lambda)
:=
h_0^\omega(A^{-1}w,\lambda).
\end{equation}

To verify the transformed kinetic identity on \(I\), take a test function
\[
\psi\in C_c^1([0,T)\times\R^d\times I)
\]
from the prescribed determining family for the transformed data. By the
countable preparation in Subsection~\ref{subsec:countable-reduction}, its
fixed linear pullback
\[
\varphi(t,x,\lambda)
:=
\psi(t,Ax+ct,\lambda)
\]
belongs to the prescribed determining family for the original data.
Therefore, \eqref{eq:kinetic-with-trace} may be applied to \(\varphi\) on
the same fixed event.
 The chain rule gives
\begin{align*}
&\partial_t\varphi(t,x,\lambda)
+
a(\lambda)\cdot\nabla_x\varphi(t,x,\lambda)
\\
&\qquad=
\left(
\partial_t\psi
+
\bigl(Aa(\lambda)+c\bigr)\cdot\nabla_w\psi
\right)
(t,Ax+ct,\lambda)
\\
&\qquad=
\left(
\partial_t\psi
+
\widetilde a(\lambda)\cdot\nabla_w\psi
\right)
(t,Ax+ct,\lambda),
\end{align*}
while
\[
\partial_\lambda\varphi(t,x,\lambda)
=
\partial_\lambda\psi(t,Ax+ct,\lambda).
\]
Multiplying \eqref{eq:kinetic-with-trace} by
$\abs{\det A}$ and using
\[
dw=\abs{\det A}\,dx
\]
transforms all Lebesgue, stochastic, and boundary terms into the
corresponding terms involving
\[
(w,\widetilde u,\widetilde h,\widetilde a,
\widetilde b_\ell,\widetilde G,\widetilde h_0^\omega).
\]
The definition of $\widetilde m^\omega$ gives the corresponding
transformation of the kinetic-measure term. Therefore,
\eqref{eq:kinetic-with-trace} holds on $I$ for the transformed data.

It remains to verify that \eqref{eq:transformed-weak-trace} is indeed the
weak trace of $\widetilde h$. Let
\[
\psi\in C_c^1(\R^d\times\R).
\]
Changing variables and adding and subtracting the same term gives
\begin{align*}
&\int_{\R^d}\int_{\R}
\bigl(
\widetilde h(t,w,\lambda)
-
\widetilde h_0^\omega(w,\lambda)
\bigr)
\psi(w,\lambda)
\,d\lambda\,dw
\\
&\quad=
\abs{\det A}
\int_{\R^d}\int_{\R}
\bigl(
h(t,x,\lambda)-h_0^\omega(x,\lambda)
\bigr)
\psi(Ax,\lambda)
\,d\lambda\,dx
\\
&\qquad+
\abs{\det A}
\int_{\R^d}\int_{\R}
h(t,x,\lambda)
\bigl(
\psi(Ax+ct,\lambda)-\psi(Ax,\lambda)
\bigr)
\,d\lambda\,dx.
\end{align*}
The essential limit of the first term is zero by
\eqref{eq:pathwise-weak-trace}, applied to
$(x,\lambda)\mapsto\psi(Ax,\lambda)$. Since $0\leq h\leq1$, the
absolute value of the second term is bounded by
\[
\int_{\R^d}\int_{\R}
\abs{
\psi(w+ct,\lambda)-\psi(w,\lambda)
}
\,d\lambda\,dw,
\]
which tends to zero as $t\downarrow0$ by translation continuity in
$L^1(\R^d\times\R)$. This proves
\eqref{eq:transformed-weak-trace}.

Joint Borel measurability of the transformed coefficients and their
derivatives, as well as their $C^1$-regularity in $(w,\lambda)$ for each
fixed $t$, follow immediately from Assumption~\ref{ass:coeff} and the
affine change of variables.
\end{proof}
\subsection{Deterministic slicing of the kinetic measure}

We use the following deterministic form of Panov's measure-slicing lemma
\cite[Lemma~4.2]{Panov:2005lr}.  It concerns one fixed Radon measure and
requires no measurable dependence on an external parameter.

\begin{lemma}[Slicing a fixed Radon measure]
\label{lem:pathwise-measure-slicing}

Let $Q\subset\R^N$ be open, and let $\gamma$ be a fixed
nonnegative Radon measure on
$
Q\times\R_\zeta.
$
Here $\R_\zeta$ labels the transverse spatial coordinate {(denoted $w_q$ in the
application in the following subsection)}; 
$r$ is the integration variable and $\zeta$ its fixed value.
Let $\theta_j(r):=j\theta(jr)$, where
$$
\theta\in C_c^\infty((-1,1)),
\qquad
\theta\geq0,
\qquad
\theta(r)=\theta(-r),
\qquad
\int_{\R}\theta(r)\,dr=1.
$$
Then there is a set
$$
S_\gamma\subset\R,
\qquad
\mathcal L^1(\R\setminus S_\gamma)=0,
$$
such that, for every $\zeta\in S_\gamma$, there is a locally finite nonnegative
Radon measure $\gamma_\zeta$ on $Q$ satisfying
\begin{equation*}
	\lim_{j\to\infty}
	\int_{Q\times\R}
	\psi(q)\theta_j(r-\zeta)
	\,d\gamma(q,r)
	=
	\int_Q
	\psi(q)
	\,d\gamma_\zeta(q)
\end{equation*}
for every $\psi\in C_c(Q)$.

More generally, let $(E_n)$ be a countable family of open subsets of $Q$
such that $\gamma(E_n\times J)<\infty$ for every $n$ and every compact
interval $J\Subset\R$. The set $S_\gamma$ may be chosen so that
$\gamma_\zeta(E_n)<\infty$ for every $n$ and every $\zeta\in S_\gamma$. In
particular, finite mass on
local cylinders reaching a boundary such as $t=0$ is inherited by almost
every transverse section.
\end{lemma}

\begin{proof}
The proof of \cite[Lemma~4.2]{Panov:2005lr}, applied on a countable
relatively compact exhaustion of $Q$, gives the first assertion:
{there exists a set $\widetilde S_\gamma \subset \R$ of full 
measure and compatible local measures $\gamma_\zeta$ defined 
for every $\zeta \in \widetilde S_\gamma$.}

We verify the finite-mass addition.  Fix $E_n$ and a compact interval
$J_0\Subset\R$, and choose a compact interval $J_1$ such that
$J_0\Subset\operatorname{int}J_1$.  Since $E_n$ is open, there are functions
$\psi_k\in C_c(Q)$ such that $0\leq\psi_k\leq1$ and
$\psi_k\uparrow\1_{E_n}$.  For each $k$, the function
\begin{equation*}
	\zeta
	\longmapsto
	\1_{{\widetilde S_\gamma}}(\zeta)
	\int_Q\psi_k(q)\,d\gamma_\zeta(q)
\end{equation*} {is Lebesgue measurable: indeed, on $\widetilde S_\gamma$ 
it is the pointwise limit of continuous functions of $\zeta$ 
(the mollified integrals in the statement), and it trivially 
vanishes on the null set $\R\setminus \widetilde S_\gamma$.}  
Fatou's lemma and Tonelli's theorem give
\begin{align*}
	\int_{J_0\cap {\widetilde S_\gamma}}
	\int_Q\psi_k(q)\,d\gamma_\zeta(q)\,d\zeta
	&\leq
	\liminf_{j\to\infty}
	\int_{Q\times\R}
	\psi_k(q)
	\left(
	\int_{J_0}\theta_j(r-\zeta)\,d\zeta
	\right)
	\,d\gamma(q,r)
	\\
	&\leq
	\gamma(E_n\times J_1).
\end{align*}
For the last inequality, use $\supp\theta_j\subset(-1/j,1/j)$: for all
sufficiently large $j$, the inner integral is bounded by $1$ and vanishes
unless $r\in J_1$.  Letting $k\to\infty$ yields
\begin{equation*}
	\int_{J_0\cap {\widetilde S_\gamma}}
	\gamma_\zeta(E_n)\,d\zeta
	\leq
	\gamma(E_n\times J_1)
	<
	\infty.
\end{equation*}
This implies that
\[
\gamma_\zeta(E_n)<\infty
\]
for almost every
\(
\zeta\in J_0\cap\widetilde S_\gamma.
\)

Let \((J_m)_{m\geq1}\) be a compact exhaustion of \(\R\), and, for
\(n,m\geq1\), define
\[
N_{n,m}
:=
\left\{
\zeta\in J_m\cap\widetilde S_\gamma:
\gamma_\zeta(E_n)=\infty
\right\}.
\]
Applying the preceding conclusion with \(J_0=J_m\), we obtain
\[
\mathcal L^1(N_{n,m})=0
\qquad
\text{for every }n,m\geq1.
\]
Define
\[
S_\gamma
:=
\widetilde S_\gamma
\setminus
\bigcup_{n=1}^{\infty}
\bigcup_{m=1}^{\infty}
N_{n,m}.
\]
Then
\[
\mathcal L^1(\R\setminus S_\gamma)=0.
\]
Moreover, if \(\zeta\in S_\gamma\) and \(n\geq1\), choose \(m\) such
that \(\zeta\in J_m\). Since
\(\zeta\in\widetilde S_\gamma\setminus N_{n,m}\), it follows that
\[
\gamma_\zeta(E_n)<\infty.
\]
Thus \(\gamma_\zeta(E_n)<\infty\) for every \(n\geq1\) and every
\(\zeta\in S_\gamma\). This construction uses only the measurability of
the scalar evaluations
\(\zeta\mapsto\gamma_\zeta(E_n)\), and does not require measurability of
the measure-valued map \(\zeta\mapsto\gamma_\zeta\).

\end{proof}

\subsection{The kinetic equation at a fixed transverse coordinate}

For a reduction datum $\alpha$ of spatial dimension 
$q=d_\alpha$, recall that $x\in\R^q$ is its active spatial variable.  
Under the next change of variables
$w=Ax+ct$, write $w=(w',w_q)\in\R^{q-1}\times\R$.  Fixing $w_q=\zeta$
makes $w'$ the active variable for $\beta$ and appends $\zeta$ to the
accumulated parameter, which becomes $(\sigma,\zeta)$; see
\eqref{eq:reduction-coordinate-recursion}.  Thus $w'$ occupies the spatial
argument generically denoted by $x$ for a reduction datum; we retain $w'$
here to record its origin from $w=(w',w_q)$.  With the already fixed
$(\omega,\sigma)$ suppressed, we call
$(t,w')\mapsto\widetilde u(t,w',\zeta)$ the section 
at $\zeta$.  

The next lemma states the exact lower-dimensional
identity satisfied by almost every such section.
It assumes only the nonrecursive input properties
\textup{(i)}--\textup{(v)} of
Definition~\ref{def:pathwise-kinetic-realization} and proves the one-step
reduction property subsequently recorded in item~\textup{(vi)}. Since the
reduction data and all associated stochastic constructions were fixed in
advance and every reduction lowers the spatial dimension, this is a
well-founded finite induction.

\begin{lemma}[Pathwise kinetic identity at almost every transverse coordinate]
\label{lem:pathwise-slice-equation}
Let $\alpha\in\mathfrak R$ be a reduction datum of spatial dimension
$q=d_\alpha\geq1$ (see Subsection~\ref{subsec:countable-reduction}), let
$(A,c,I)$ be one of its fixed reductions, and let $\beta$ be the associated
reduction datum with $d_\beta=q-1$.  Write
$\mathscr D_I:=\mathscr D_\beta$
and
$\mathscr T_I:=\mathscr T_\beta$. 
Fix $\omega\in\Omega_*$
(defined in \eqref{eq:common-pathwise-event})
and a recursively admissible accumulated parameter
$
\sigma\in\R^{r_\alpha}
$
(that is, $\sigma\in\mathcal A_\alpha(\omega)$ as defined in
\eqref{eq:recursive-admissible-parameters})
and assume that the reduced solution, kinetic measure, weak trace, and
martingale paths at $(\omega,\sigma)$ satisfy items \textup{(i)}--\textup{(v)}
of Definition~\ref{def:pathwise-kinetic-realization}.  In the joint
martingale construction, $\omega'\in\Omega$ remains a free probability
argument and $\zeta\in\R$ remains a free transverse parameter.  Define
\begin{align}
	\widetilde u
	\bigl(
	\omega',t,w',\sigma,\zeta
	\bigr)
	&:=
	u_\alpha
	\left(
	\omega',t,
	A^{-1}\bigl((w',\zeta)-ct\bigr),
	\sigma
	\right)
	\notag\\
	&=
	u_\beta
	\bigl(
	\omega',t,w',(\sigma,\zeta)
	\bigr).
	\label{eq:random-reduction-transform}
\end{align}
The last equality follows from
\eqref{eq:reduction-coordinate-recursion} and
\eqref{eq:reduction-random-sections}.
The coordinate change is deterministic and leaves time unchanged, so it uses
the ordinary chain rule rather than It\^o's formula and does not transform
$W_\ell$; hence the coefficients below have no $\omega'$-dependence:
\begin{align}
	\widetilde b_{\ell,\alpha}(t,w,\sigma,\lambda)
	&:=
	b_{\ell,\alpha}
	\bigl(
	t,A^{-1}(w-ct),\sigma,\lambda
	\bigr),
	\notag\\
	\widetilde a(\lambda)
	&:=
	A f_\alpha'(\lambda)+c.
	\label{eq:transformed-reduction-coefficients}
\end{align}
By \eqref{eq:reduction-coordinate-recursion},
\begin{equation}\label{eq:reduced-noise-identification}
	\widetilde b_{\ell,\alpha}
	\bigl(
	t,(w',\zeta),\sigma,\lambda
	\bigr)
	=
	b_{\ell,\beta}
	\bigl(
	t,w',(\sigma,\zeta),\lambda
	\bigr).
\end{equation}
Let $\widetilde m^{\omega,\sigma}$ be the transformed current measure defined
by
\begin{align}
	\int
	\psi(t,w,\lambda)
	\,d\widetilde m^{\omega,\sigma}(t,w,\lambda)
	:=
	\abs{\det A}
	\int
	\psi(t,Ax+ct,\lambda)
	\,dm_\alpha^{\omega,\sigma}(t,x,\lambda)
	\label{eq:transformed-reduction-measure}
\end{align}
for every
$
\psi\in C_c((0,T)\times\R^q\times I).
$
These are the reduced versions of
\eqref{eq:transformed-solution} and the transformed data in
Lemma~\ref{lem:pathwise-linear-change}.  In deterministic quantities below,
the fixed pair $(\omega,\sigma)$ is suppressed; we write
$
\widetilde b_\ell(t,w,\lambda)
:=
\widetilde b_{\ell,\alpha}(t,w,\sigma,\lambda)
$
and
$
\widetilde m^\omega
:=
\widetilde m^{\omega,\sigma}.
$
Thus
$\widetilde h(t,w',\zeta,\lambda)
:=\1_{\{\lambda<\widetilde u(\omega,t,w',\sigma,\zeta)\}}$.
Lemma~\ref{lem:pathwise-linear-change}, applied with $d=q$, shows that these
transformed data satisfy \eqref{eq:kinetic-with-trace} on the open interval
$I\subset(-L,L)$, with the substitutions specified in that lemma.  By the
choice of the reduction, $\widetilde a_q(\lambda)=0$ for every $\lambda\in I$.
Let
$
\widetilde a_\parallel
:=
(\widetilde a_1,\ldots,\widetilde a_{q-1}).
$

Here
$
\mathscr D_I\subset C_c^1(\R^{q-1}\times I)
$
is $L^1_{\loc}$-determining
(i.e., for every relatively compact open cylinder
$
Q\Subset\R^{q-1}\times I,
$
the rational span of
$
\mathscr D_I\cap C_c^1(Q)
$
is dense in $L^1(Q)$) and
$
\mathscr T_I
$
contains the rational linear span of
$
\{\theta(t)\psi(w',\lambda):
\theta\in\Theta,\ \psi\in\mathscr D_I\}.
$
Then there is a set
$$
S_{\omega,\sigma}\subset\R,
\qquad
\mathcal L^1(\R\setminus S_{\omega,\sigma})=0,
$$
such that the following assertions hold for every
$\zeta\in S_{\omega,\sigma}$: Define
\begin{equation*}
	\widetilde u_\zeta(t,w')
	:=
	\widetilde u(\omega,t,w',\sigma,\zeta),
	\qquad
	\widetilde h_\zeta(t,w',\lambda)
	:=
	\1_{\{\lambda<\widetilde u_\zeta(t,w')\}},
\end{equation*}
\begin{equation*}
	\widetilde b_{\ell,\zeta}(t,w',\lambda)
	:=
	\widetilde b_\ell(t,w',\zeta,\lambda),
	\qquad
	\widetilde G_\zeta^2(t,w',\lambda)
	:=
	\sum_{\ell\geq1}
	\abs{\widetilde b_{\ell,\zeta}(t,w',\lambda)}^2.
\end{equation*}
Let $\theta_j(r)=j\theta(jr)$ be the nonnegative, even approximate identity
from Lemma~\ref{lem:pathwise-measure-slicing}; in particular,
$\int_{\R}\theta_j(r)\,dr=1$ and
$\supp\theta_j\subset(-1/j,1/j)$.
There is a locally finite nonnegative Radon measure
$\widetilde m_\zeta^{\omega,\sigma}$ on
$(0,T)\times\R^{q-1}\times I$. Since $(\omega,\sigma)$ is fixed, we
abbreviate this measure by $\widetilde m_\zeta^\omega$. Thus, for every
$\psi\in C_c((0,T)\times\R^{q-1}\times I)$,
\begin{align}
	&\int
	\psi(t,w',\lambda)
	\,d\widetilde m_\zeta^\omega(t,w',\lambda)
	\notag\\
	&\qquad=
	\lim_{j\to\infty}
		\int
		\psi(t,w',\lambda)
		\theta_j(w_q-\zeta)
		\,d\widetilde m^\omega(t,w',w_q,\lambda).
	\label{eq:def-sliced-kinetic-measure}
\end{align}
Thus the transverse mollifications of $\widetilde m^\omega$ converge vaguely
to $\widetilde m_\zeta^\omega$ at $w_q=\zeta$.  This is a
Lebesgue-density slice in the $w_q$-direction, not the set-theoretic
restriction of $\widetilde m^\omega$ to $\{w_q=\zeta\}$; no measurable
dependence of $\widetilde m_\zeta^\omega$ on $\zeta$ is asserted.
For every $R>0$ and $I'\Subset I$,
\begin{equation}\label{eq:sliced-measure-local-mass}
	\widetilde m_\zeta^\omega
	\bigl(
	(0,T)\times B_R^{q-1}\times I'
	\bigr)
	<\infty.
\end{equation}
The measure vanishes outside
$
(0,T)\times\R^{q-1}\times[-M,M].
$

For $\phi\in\mathscr T_I$, let
$
Z_\phi(\omega',t,\sigma,\zeta)
$
denote the jointly measurable representative of the parameterized It\^o
integral below.  This representative is among those selected in the
construction of $\Omega_{\mathrm{mar}}$ in
\eqref{eq:martingale-common-event}.  Here $\omega'$ is its free probability
argument; the fixed realization $\omega\in\Omega_*$ is inserted only after
the representative has been selected:
\begin{align}
	&Z_\phi(\omega',t,\sigma,\zeta)
	\notag\\
	&\quad
	:=\sum_{\ell\geq1}\int_0^t
	\Biggl(
	\int_{\R^{q-1}}
	\widetilde b_{\ell,\alpha}
	\bigl(
	r,(w',\zeta),\sigma,
	\widetilde u(\omega',r,w',\sigma,\zeta)
	\bigr)
	\phi
	\bigl(
	r,w',
	\widetilde u(\omega',r,w',\sigma,\zeta)
	\bigr)
	\,dw'
	\Biggr)
	dW_\ell(r)
	\label{eq:sliced-scalar-martingale}
\end{align}
By \eqref{eq:random-reduction-transform} and
\eqref{eq:reduced-noise-identification}, the spatial integrand in
\eqref{eq:sliced-scalar-martingale} is
$
\mathcal I_{\ell,\beta,\phi}
(\omega',r,(\sigma,\zeta))
$
from \eqref{eq:reduction-scalar-integrand}.  The representative is jointly
measurable in $(\omega',t,\sigma,\zeta)$, and the equality in
\eqref{eq:sliced-scalar-martingale} holds for
$(\Prob\otimes\mathcal L^{r_\alpha+1})$-almost every
$(\omega',\sigma,\zeta)$.  For
$\zeta\in S_{\omega,\sigma}$, set
$Z_{\phi,\zeta}^\omega(t):=Z_\phi(\omega,t,\sigma,\zeta)$;
the fixed $\sigma$ is suppressed on the left-hand side.
The path $Z_{\phi,\zeta}^\omega$ is continuous and starts from zero, and
\begin{align}
	&\int_0^T
	\int_{\R^{q-1}}
	\int_I
	\widetilde h_\zeta
	\bigl(
	\partial_t\phi
	+
	\widetilde a_\parallel(\lambda)\cdot\nabla_{w'}\phi
	\bigr)
	\,d\lambda\,dw'\,dt
	\notag\\
	&\quad=
	-Z_{\phi,\zeta}^\omega(T)
	-
	\frac12
	\int_0^T
	\int_{\R^{q-1}}
	\widetilde G_\zeta^2
	\bigl(
	t,w',\widetilde u_\zeta(t,w')
	\bigr)
	\partial_\lambda\phi
	\bigl(
	t,w',\widetilde u_\zeta(t,w')
	\bigr)
	\,dw'\,dt
	\notag\\
	&\qquad+
	\int
	\partial_\lambda\phi(t,w',\lambda)
	\,d\widetilde m_\zeta^\omega(t,w',\lambda).
	\label{eq:pathwise-sliced-kinetic-identity}
\end{align}

For every $\psi\in\mathscr D_I$, let
$
M_\psi(\omega',t,\sigma,\zeta)
$
be the jointly measurable version of the scalar It\^o integral in
\eqref{eq:reduction-scalar-martingales} for the reduction datum $\beta$, the
accumulated parameter $(\sigma,\zeta)$, and the constant-in-time extension
$\widehat\psi(t,w',\lambda):=\psi(w',\lambda)$ introduced before
\eqref{eq:reduction-scalar-integrand}.  Explicitly,
\begin{equation}
	M_\psi(\omega',t,\sigma,\zeta)
	:=
	\sum_{\ell\geq1}
	\int_0^t
	\mathcal I_{\ell,\beta,\widehat\psi}
	\bigl(
	\omega',r,(\sigma,\zeta)
	\bigr)
	dW_\ell(r)
	\label{eq:sliced-time-martingale}
\end{equation}
For $\zeta\in S_{\omega,\sigma}$, the path 
$M_{\psi,\zeta}^\omega(t):=M_\psi(\omega,t,\sigma,\zeta)$
is continuous and starts from zero; again, $\sigma$ is fixed 
and suppressed on the left-hand side.  Finally, for every
$\rho\in\mathcal R$ with $\supp\rho\Subset I$ 
and every compact $K'\Subset\R^{q-1}$, the 
$\zeta$-section of \eqref{eq:sliced-spatial-martingale} belongs to 
$C([0,T];L^2(K'))$ and starts from zero.  
The scalar and $L^2(K')$-valued versions obtained
from different localization intervals agree on their overlaps.  Moreover,
every $\zeta\in S_{\omega,\sigma}$ satisfies
$(\sigma,\zeta)\in\mathcal A_\beta(\omega)$; in particular, the section at
$(\sigma,\zeta)$ retains the jointly constructed 
processes for every subsequent reduction from $\beta$ in
\eqref{eq:reduction-spatial-martingale}--\eqref{eq:reduction-scalar-martingales}.
\end{lemma}

\begin{proof}
For $\phi\in\mathscr T_I$, use
$\phi(t,w',\lambda)\theta_j(w_q-\zeta)$ as a test function in
\eqref{eq:kinetic-with-trace} for the transformed data specified in
Lemma~\ref{lem:pathwise-linear-change}. Since
$\widetilde a_q=0$ on $I$, the transport term contains no derivative of
$\theta_j(w_q-\zeta)$.  The Lebesgue differentiation theorem gives, for almost
every $\zeta\in\R$, the limits of the time-derivative, tangential transport, and
It\^o-correction terms appearing in
\eqref{eq:pathwise-sliced-kinetic-identity}.

Apply Lemma~\ref{lem:pathwise-measure-slicing} to the fixed measure
$\widetilde m^\omega$, with
$Q=(0,T)\times\R^{q-1}\times I$ and $\zeta=w_q$. Using rational radii and
relatively compact open intervals
$I_k\Subset I$, choose a countable exhaustion by open cylinders of the form
$(0,T)\times B_R\times I_k$.
Lemma~\ref{lem:pathwise-measure-slicing} gives
\eqref{eq:def-sliced-kinetic-measure} and
\eqref{eq:sliced-measure-local-mass} outside one Lebesgue-null set of
transverse coordinates; arbitrary $R$ and $I'\Subset I$ follow by inclusion.

We next show that, on a localization
$J_\alpha\times J\Subset\R^{r_\beta}$, the process
$Z_\phi(\omega',t,\sigma,\zeta)$ in
\eqref{eq:sliced-scalar-martingale} is, for almost every
$(\omega',\sigma,\zeta)$, the $(\sigma,\zeta)$-section of
$Z_{\beta,\phi,J_\alpha\times J}(t)$ from
\eqref{eq:reduction-scalar-martingales}.  Here $\beta$ is the one-step
reduction in the statement, whose accumulated parameter is
$(\sigma,\zeta)$.  At this stage, $\omega'$ and $\sigma$ remain free
variables; the fixed realization $\omega$ is inserted only after the jointly
measurable representative has been selected.  Fix
$\phi\in\mathscr T_I$, a compact box
$J_\alpha\Subset\R^{r_\alpha}$, and a compact interval $J\Subset\R$, and set
\begin{equation}\label{eq:sliced-integrand}
	F_{\ell,\phi}(\omega',t,\sigma,\zeta)
	:=
	\int_{\R^{q-1}}
	\widetilde b_{\ell,\alpha}
	\bigl(
	t,(w',\zeta),\sigma,
	\widetilde u(\omega',t,w',\sigma,\zeta)
	\bigr)
	\phi
	\bigl(
	t,w',
	\widetilde u(\omega',t,w',\sigma,\zeta)
	\bigr)
	\,dw'.
\end{equation}
The predictability statement following
\eqref{eq:reduction-random-sections} shows that
$F_{\ell,\phi}$ is
$\mathcal P\otimes\cB(\R^{r_\alpha}\times\R)$-measurable.  Thus the
parameter variables may be retained throughout the It\^o construction.
The bound $\abs{\widetilde u}\leq M$, the zeroth-order coefficient estimate
in \eqref{eq:noise-local-bound}, transferred to the reduced coefficients by
\eqref{eq:reduced-noise-identification}, and Cauchy--Schwarz in $w'$ give
\begin{equation}\label{eq:sliced-integrand-bound}
	\sum_{\ell\geq1}
	\E
	\int_0^T
	\int_{J_\alpha}
	\int_J
	\abs{F_{\ell,\phi}(t,\sigma,\zeta)}^2
	\,d\zeta\,d\sigma
	\,dt
	<\infty.
\end{equation}
The standard parameterized It\^o-integration theorem
\cite[Theorem~IV.65]{Protter:2005aa}, applied here under the stronger
$L^2$ bound \eqref{eq:sliced-integrand-bound}, can be proved directly by
approximation with jointly predictable simple integrands.  It gives a
jointly measurable function
$z_{\phi,J_\alpha,J}$ such that, for
$(\Prob\otimes\mathcal L^{r_\alpha+1})$-almost 
every $(\omega',\sigma,\zeta)$,
\begin{equation}\label{eq:scalar-stochastic-fubini}
	z_{\phi,J_\alpha,J}(\omega',t,\sigma,\zeta)
	=
	\sum_{\ell\geq1}
	\int_0^t
	F_{\ell,\phi}(\omega',r,\sigma,\zeta)
	\,dW_\ell(r)
\end{equation}
for every $t\in[0,T]$, with a continuous scalar path.  Scalar
Burkholder--Davis--Gundy and Tonelli give
\begin{align}
	\E
	\int_{J_\alpha}
	\int_J
	\sup_{0\leq t\leq T}
	\abs{
	z_{\phi,J_\alpha,J}(t,\sigma,\zeta)
	}^2
	\,d\zeta\,d\sigma
	&\leq
	4
	\sum_{\ell\geq1}
	\E
	\int_0^T
	\int_{J_\alpha}
	\int_J
	\abs{F_{\ell,\phi}(t,\sigma,\zeta)}^2
	\,d\zeta\,d\sigma\,dt.
	\label{eq:stochastic-fubini-BDG}
\end{align}

The process in \eqref{eq:scalar-stochastic-fubini} is a jointly measurable
family of real-valued martingales, with one continuous path for almost every
fixed $(\sigma,\zeta)$.  To integrate these sections in the parameter
variables, we also regard the same family
$(F_{\ell,\phi})_{\ell\geq1}$ as a single
$L^2(J_\alpha\times J)$-valued integrand.  Accordingly,
\eqref{eq:sliced-integrand-bound} defines the
$L^2(J_\alpha\times J)$-valued It\^o integral
\begin{equation}\label{eq:sliced-L2-martingale}
\mathcal Z_{\phi,J_\alpha,J}(t)
:= \sum_{\ell\geq1}
\int_0^t
F_{\ell,\phi}(r,\cdot)
\,dW_\ell(r)
\end{equation}
as a continuous $L^2(J_\alpha\times J)$-valued martingale.  For a countable
dense set of
$
g\in L^2(J_\alpha\times J),
$
stochastic Fubini gives, at every rational time,
\begin{align}
	\int_{J_\alpha}
	\int_J
	g(\sigma,\zeta)
	z_{\phi,J_\alpha,J}(t,\sigma,\zeta)
	\,d\zeta\,d\sigma
	&=
	\sum_{\ell\geq1}
	\int_0^t
	\left(
	\int_{J_\alpha}
	\int_J
	g(\sigma,\zeta)
	F_{\ell,\phi}(r,\sigma,\zeta)
	\,d\zeta\,d\sigma
	\right)
	dW_\ell(r)
	\notag\\
	&=
	\left\langle
	\mathcal Z_{\phi,J_\alpha,J}(t),g
	\right\rangle_{L^2(J_\alpha\times J)}.
	\label{eq:stochastic-fubini-identification}
\end{align}
The parameterized construction makes
$z_{\phi,J_\alpha,J}$ jointly measurable in
$(\omega',t,\sigma,\zeta)$.  By
\eqref{eq:stochastic-fubini-BDG}, outside one $\Prob$-null set,
$$
\int_{J_\alpha}\int_J
\sup_{0\leq t\leq T}
\abs{z_{\phi,J_\alpha,J}(t,\sigma,\zeta)}^2
\,d\zeta\,d\sigma
<\infty.
$$
Together with the scalar path continuity in
\eqref{eq:scalar-stochastic-fubini}, dominated convergence shows that
$t\mapsto z_{\phi,J_\alpha,J}(t,\cdot,\cdot)$ belongs to
$C([0,T];L^2(J_\alpha\times J))$.  Consequently,
\eqref{eq:stochastic-fubini-identification}, first at rational times, and
density in the dual variable identify
$z_{\phi,J_\alpha,J}(t,\cdot,\cdot)$ with the continuous martingale
$\mathcal Z_{\phi,J_\alpha,J}(t)$ from
\eqref{eq:sliced-L2-martingale}.  Continuity extends the identity to every
$t\in[0,T]$, so, outside one $\Prob$-null set,
\begin{equation*}
	z_{\phi,J_\alpha,J}(t,\cdot,\cdot)
	=
	\mathcal Z_{\phi,J_\alpha,J}(t)
	\quad\text{in }L^2(J_\alpha\times J)
\end{equation*}
for every $t\in[0,T]$.
Tonelli applied to \eqref{eq:stochastic-fubini-BDG} gives, for
$(\Prob\otimes\mathcal L^{r_\alpha})$-almost every $(\omega',\sigma)$,
\begin{equation}\label{eq:fixed-parameter-stochastic-fubini-bound}
	\int_J
	\sup_{0\leq t\leq T}
	\abs{z_{\phi,J_\alpha,J}(\omega',t,\sigma,\zeta)}^2
	\,d\zeta
	<\infty,
	\qquad
	z_{\phi,J_\alpha,J}(\omega',\cdot,\sigma,\cdot)
	\in C([0,T];L^2(J)).
\end{equation}
Indeed, the first assertion follows from Tonelli; scalar path continuity in
\eqref{eq:scalar-stochastic-fubini} and dominated convergence then give the
second.  Repeating the argument with $\psi\in\mathscr D_I$ gives the same
fixed-$\sigma$ conclusion for the processes in
\eqref{eq:sliced-time-martingale} defined using the constant-in-time
extensions $\widehat\psi$.
Consequently, \eqref{eq:scalar-stochastic-fubini} supplies the scalar It\^o
integral for almost every $(\sigma,\zeta)$, while
\eqref{eq:stochastic-fubini-identification} states that pairing these scalar
integrals with $g\in L^2(J_\alpha\times J)$ is the same as first pairing
$F_{\ell,\phi}$ with $g$ in $(\sigma,\zeta)$ and then taking the It\^o
integral.  Applying this identity to the bounded convolution operator
$v(\sigma,\cdot)\mapsto\theta_j*v(\sigma,\cdot)$ from
$L^2(J_\alpha\times J)$ to $L^2(J_\alpha\times J_0)$, where $J_0\Subset J$,
moves transverse mollification through the It\^o integral and gives
\eqref{eq:mollified-stochastic-fubini}.  This identifies the martingale term
produced by the test function
$\phi(t,w',\lambda)\theta_j(w_q-\zeta)$ with the corresponding mollification
of the scalar sections in \eqref{eq:scalar-stochastic-fubini}.

For rational compact boxes
$
J_{\alpha,1},J_{\alpha,2}
$
and intervals $J_1,J_2$, uniqueness of the scalar It\^o integrals in
\eqref{eq:scalar-stochastic-fubini} shows that the representatives agree
almost everywhere on
$
(J_{\alpha,1}\cap J_{\alpha,2})\times(J_1\cap J_2),
$
up to indistinguishability in time.  A countable intersection fixes all
restriction identities simultaneously.  Ordinary Fubini gives the
corresponding fixed-$\sigma$ restriction identities for
$(\Prob\otimes\mathcal L^{r_\alpha})$-almost every $(\omega',\sigma)$;
these identities are included in clause~\textup{(b)} defining
$\mathcal C_\alpha(\omega)$.  The representatives therefore define the
single function
$
Z_\phi(\omega',t,\sigma,\zeta)
$
used in \eqref{eq:sliced-scalar-martingale}.

If $J_0\Subset J$, then, for large $j$, stochastic Fubini also yields in
$
L^2(\Omega;L^2(J_\alpha\times J_0))
$
\begin{align}
	&\sum_{\ell\geq1}
	\int_0^T
	\left(
	\int_J
	\theta_j(\widehat\zeta-\zeta)
	F_{\ell,\phi}(t,\sigma,\widehat\zeta)
	\,d\widehat\zeta
	\right)
	dW_\ell(t)
	\notag\\
	&\qquad=
	\int_J
	\theta_j(\widehat\zeta-\zeta)
	z_{\phi,J_\alpha,J}(T,\sigma,\widehat\zeta)
	\,d\widehat\zeta.
	\label{eq:mollified-stochastic-fubini}
\end{align}
The parameterized construction supplies jointly measurable representatives
of the two sides of \eqref{eq:mollified-stochastic-fubini} on
$\Omega\times J_\alpha\times J_0$.  Ordinary Fubini therefore shows that, for
$(\Prob\otimes\mathcal L^{r_\alpha})$-almost every $(\omega',\sigma)$, its
fixed-$\sigma$ form holds in $L^2(J_0)$.  Taking the countable intersection
over the prescribed test functions, boxes, intervals, and mollifier indices
gives precisely clause~\textup{(b)} in the definition of
$\mathcal C_\alpha(\omega)$.
Repeating \eqref{eq:scalar-stochastic-fubini} and
\eqref{eq:stochastic-fubini-identification} with
$\psi(w',\lambda)\in\mathscr D_I$ in place of
$\phi(t,w',\lambda)$ constructs the jointly measurable processes whose
fixed-realization sections are
\eqref{eq:sliced-time-martingale}.

The reduction data, test functions, rational boxes, rational intervals, and mollifier
indices form countable families,
so
\eqref{eq:scalar-stochastic-fubini} and
\eqref{eq:mollified-stochastic-fubini} are included in
$\Omega_{\mathrm{mar}}$ from
\eqref{eq:martingale-common-event}.  If
$\omega\in\Omega_*$ and $\sigma\in\mathcal A_\alpha(\omega)$, then
$\sigma\in\mathcal C_\alpha(\omega)$ by
\eqref{eq:recursive-admissible-parameters}.  Clause~\textup{(b)} and
\eqref{eq:fixed-parameter-stochastic-fubini-bound} therefore show that
$z_{\phi,J_\alpha,J}(\omega,T,\sigma,\cdot)$ 
belongs to $L^2(J)$ and that the fixed-$\sigma$ identity
\eqref{eq:mollified-stochastic-fubini} holds in $L^2(J_0)$.  Outside the union of
the corresponding null sets in $J_0$, one for each prescribed mollifier
index, the equality holds pointwise for every $j$.  If such a point is also
a Lebesgue point of
$z_{\phi,J_\alpha,J}(\omega,T,\sigma,\cdot)$ and satisfies the fixed-$\sigma$
restriction identities, the right-hand side converges along the full
sequence $j\to\infty$ to
$z_{\phi,J_\alpha,J}(\omega,T,\sigma,\zeta)
=Z_\phi(\omega,T,\sigma,\zeta)$; the left-hand side has the same limit.

For the $L^2_{\loc}(\R^{q-1})$-valued martingale processes used by the
lower-dimensional blow-up argument, define, for
$\rho\in\mathcal R$ with $\supp\rho\Subset I$, $J\Subset\R$, and
$K'\Subset\R^{q-1}$,
\begin{equation*}
	H_{\ell,\rho}(\omega',t,w',\sigma,\zeta)
	:=
	\rho
	\bigl(
	\widetilde u(\omega',t,w',\sigma,\zeta)
	\bigr)
	\widetilde b_{\ell,\alpha}
	\bigl(
	t,(w',\zeta),\sigma,
	\widetilde u(\omega',t,w',\sigma,\zeta)
	\bigr)
\end{equation*}
and
\begin{align}
	\mathscr Z_{\rho,J_\alpha,J,K'}(t)(\sigma,\zeta,w')
	:=
	\sum_{\ell\geq1}
	\int_0^t
	H_{\ell,\rho}(r,w',\sigma,\zeta)
	\,dW_\ell(r).
	\label{eq:sliced-spatial-martingale}
\end{align}
The coefficient bound \eqref{eq:noise-local-bound} first realizes
\eqref{eq:sliced-spatial-martingale} in 
$C([0,T];L^2(J_\alpha\times J\times K'))$. 
The parameterized construction used in
\eqref{eq:scalar-stochastic-fubini}, now with values in $L^2(K')$, gives a
jointly measurable representative such that, for
$(\Prob\otimes\mathcal L^{r_\alpha+1})$-almost 
every $(\omega',\sigma,\zeta)$, its $(\sigma,\zeta)$-section 
belongs to $C([0,T];L^2(K'))$ 
and starts from zero.  Uniqueness of Hilbert-space-valued 
It\^o integrals gives the restriction identities 
when $J_\alpha$, $J$, or $K'$ is enlarged.
The Hilbert-space BDG inequality and Tonelli also give
\begin{equation}\label{eq:fixed-parameter-spatial-martingale-bound}
	\E
	\int_{J_\alpha}\int_J
	\sup_{0\leq t\leq T}
	\norm{
	\mathscr Z_{\rho,J_\alpha,J,K'}(t)(\sigma,\zeta,\cdot)
	}_{L^2(K')}^2
	\,d\zeta\,d\sigma
	<\infty.
\end{equation}
Hence, for
$(\Prob\otimes\mathcal L^{r_\alpha})$-almost every $(\omega',\sigma)$,
the fixed-$\sigma$ section belongs to
$C([0,T];L^2(J\times K'))$: its fixed-$(\sigma,\zeta)$ paths are continuous
in $L^2(K')$, while
\eqref{eq:fixed-parameter-spatial-martingale-bound} supplies the dominating
function.  The countable intersection of these assertions and the
restriction identities gives clause~\textup{(c)} in the definition of
$\mathcal C_\alpha(\omega)$.
These identities are fixed simultaneously over the prescribed countable
exhaustions and included in $\Omega_{\mathrm{mar}}$.
By \eqref{eq:reduction-coordinate-recursion},
\eqref{eq:sliced-spatial-martingale} is exactly the
$(\sigma,\zeta)$-section of
\eqref{eq:reduction-spatial-martingale} for the reduction $\beta$.  The same
identification holds for the scalar processes
\eqref{eq:reduction-scalar-martingales}.  By
\eqref{eq:recursive-admissible-parameters}, almost every $\zeta\in\R$
retains all processes required at every later reduction.

Because $\sigma\in\mathcal A_\alpha(\omega)\subset
\mathcal C_\alpha(\omega)$, clause~\textup{(b)} and
\eqref{eq:fixed-parameter-stochastic-fubini-bound} give, at the fixed pair
$(\omega,\sigma)$, the scalar identity
\eqref{eq:scalar-stochastic-fubini} for almost every $\zeta$,
$z_{\phi,J_\alpha,J}(\omega,\cdot,\sigma,\cdot)
\in C([0,T];L^2(J))$, the analogous conclusion for
$M_\psi(\omega,\cdot,\sigma,\cdot)$ in
\eqref{eq:sliced-time-martingale}, and the fixed-$\sigma$ form of
\eqref{eq:mollified-stochastic-fubini} in $L^2(J_0)$ whenever
$J_0\Subset J$.  Clause~\textup{(c)} gives
$\mathscr Z_{\rho,J_\alpha,J,K'}
(\omega,\cdot)(\sigma,\cdot,\cdot)
\in C([0,T];L^2(J\times K'))$
and, for almost every $\zeta\in J$, a fixed-$(\sigma,\zeta)$ path in
$C([0,T];L^2(K'))$.  These conclusions hold simultaneously for the
prescribed countable test functions and localizations and are consistent
under restriction.  Hence, after $(\omega,\sigma)$ has been fixed, only
Lebesgue-null sets in $\zeta$ remain in the following intersection.
Intersect the full-measure sets of transverse coordinates supplied by the
Lebesgue differentiation theorem,
Lemma~\ref{lem:pathwise-measure-slicing}, and
\eqref{eq:scalar-stochastic-fubini} and its fixed-$\sigma$ restriction
identities, together with the fixed-$\sigma$
representatives of \eqref{eq:mollified-stochastic-fubini} for every
prescribed mollifier index and the section sets for
\eqref{eq:sliced-time-martingale} and
\eqref{eq:sliced-spatial-martingale} for every prescribed test function and
localization, and with the recursively admissible reduced-parameter sets
from \eqref{eq:recursive-admissible-parameters}.  Call the resulting
full-measure intersection
$S_{\omega,\sigma}$.  Passing to the limit in
\eqref{eq:kinetic-with-trace} for the transformed data specified in
Lemma~\ref{lem:pathwise-linear-change}, tested with
$\phi(t,w',\lambda)\theta_j(w_q-\zeta)$, gives
\eqref{eq:pathwise-sliced-kinetic-identity} for every
$\phi\in\mathscr T_I$ and $\zeta\in S_{\omega,\sigma}$.

For $\psi\in\mathscr D_I$, use the identities corresponding to
$\theta(t)\psi(w',\lambda)$.  Recall from
\eqref{eq:sliced-time-martingale} that
$M_{\psi,\zeta}^\omega(t):=M_\psi(\omega,t,\sigma,\zeta)$ is the continuous
sample path of the scalar It\^o integral constructed using the
constant-in-time extension
$\widehat\psi(t,w',\lambda)=\psi(w',\lambda)$; the fixed $\sigma$ is
suppressed in its notation.  For every $\theta\in\Theta$, the parameterized
It\^o construction gives
\begin{equation*}
	Z_{\theta\psi,\zeta}^\omega(T)
	=
	\int_0^T
	\theta(t)
	\,dM_{\psi,\zeta}^\omega(t)
	=
	-
	\int_0^T
	\theta'(t)M_{\psi,\zeta}^\omega(t)
	\,dt.
\end{equation*}
%The one-dimensional distributional argument used in
%Lemma~\ref{lem:time-integrated}, together with the local mass bound
%\eqref{eq:sliced-measure-local-mass}, yields a unique lower-dimensional
%weak trace $h_{0,\zeta}^\omega$ satisfying the analogue of
%\eqref{eq:pathwise-weak-trace} on $I$.

The $C^1$-density built into $\mathscr T_I$ in
Subsection~\ref{subsec:countable-reduction}, together with the continuity
of the spatial martingale pairing, extends
\eqref{eq:pathwise-sliced-kinetic-identity} to the prescribed extended
test functions
\[
\eta(t,w')\rho(\lambda)
\]
supported away from $t=0$. The cutoff argument used in
Lemma~\ref{lem:kinetic-with-trace}, together with
\eqref{eq:sliced-measure-local-mass}, then yields the lower-dimensional
analogue of \eqref{eq:kinetic-with-trace} for the prescribed determining
family localized to $I$. Combining this identity with
\eqref{eq:sliced-spatial-martingale} gives the lower-dimensional analogues
of \eqref{eq:pathwise-averaged-identity} and
\eqref{eq:pathwise-averaged-identity-with-trace}.

For the fixed $\zeta\in S_{\omega,\sigma}$, write
\[
\widetilde m_\zeta^\omega
:=
\widetilde m_\zeta^{\omega,\sigma}
\]
for the Radon measure constructed in
\eqref{eq:def-sliced-kinetic-measure}. The local mass bound
\eqref{eq:sliced-measure-local-mass} holds for this measure.

The one-dimensional distributional argument used in
Lemma~\ref{lem:time-integrated} then yields a unique lower-dimensional
weak trace $h_{0,\zeta}^\omega$ satisfying the analogue of
\eqref{eq:pathwise-weak-trace} on $I$. The $C^1$-density built into
$\mathscr T_I$ in Subsection~\ref{subsec:countable-reduction}, together
with the continuity of the spatial martingale pairing, extends
\eqref{eq:pathwise-sliced-kinetic-identity} to the prescribed extended
test functions
\[
\eta(t,w')\rho(\lambda)
\]
supported away from $t=0$. The cutoff argument used in
Lemma~\ref{lem:kinetic-with-trace}, together with
\eqref{eq:sliced-measure-local-mass}, therefore yields the
lower-dimensional analogue of \eqref{eq:kinetic-with-trace} for the
prescribed determining family localized to $I$, with kinetic-measure term
given by $\widetilde m_\zeta^\omega$. Combining this identity with
\eqref{eq:sliced-spatial-martingale} gives the lower-dimensional analogues
of \eqref{eq:pathwise-averaged-identity} and
\eqref{eq:pathwise-averaged-identity-with-trace}.

The measure $\widetilde m_\zeta^\omega$ is constructed separately for each
fixed parameter $(\sigma,\zeta)$. We do not select a Borel map
\[
\zeta
\longmapsto
\widetilde m_\zeta^\omega
\]
with values in
\[
\cM_{\loc}\bigl((0,T)\times\R^{q-1}\times I\bigr),
\]
nor a jointly measurable family
\[
(\omega,\sigma,\zeta)
\longmapsto
\widetilde m_\zeta^{\omega,\sigma}.
\]
Consequently, the section measures are not asserted to form a family of
random kinetic measures, and no adaptedness or predictability is claimed
for their cumulative processes. This causes no difficulty because no step
of the proof integrates $\widetilde m_\zeta^\omega$ with respect to
$\zeta$. The stochastic-Fubini arguments
\eqref{eq:scalar-stochastic-fubini},
\eqref{eq:mollified-stochastic-fubini}, and
\eqref{eq:sliced-spatial-martingale} concern only the jointly measurable
integrands \eqref{eq:sliced-integrand} and $H_{\ell,\rho}$ and are carried
out before evaluation at the fixed realization.
\end{proof}

The induction uses the following class, which records the identities and
martingale processes preserved by one-step reduction.

\begin{definition}[Admissible reduced kinetic realization]
\label{def:pathwise-kinetic-realization}
Let $\alpha\in\mathfrak R$ be a reduction datum of spatial dimension
$q=d_\alpha$, with kinetic interval $I_\alpha$, reduced flux $f_\alpha$,
noise coefficients
$
b_{\ell,\alpha},
$
and determining families
$
\mathscr D_\alpha
$
and
$
\mathscr T_\alpha.
$
For a fixed realization $\omega\in\Omega_*$, where $\Omega_*$ is defined in
\eqref{eq:common-pathwise-event}, and a recursively admissible accumulated
transverse parameter
$
\sigma\in\mathcal A_\alpha(\omega)\subset\R^{r_\alpha},
$
with $\mathcal A_\alpha(\omega)$ defined in
\eqref{eq:recursive-admissible-parameters},
an admissible reduced kinetic realization associated with $\alpha$ consists
of the following objects:
\begin{enumerate}[label=\textup{(\roman*)}]
	\item the measurable function
	$
	u_\alpha^{\omega,\sigma}:
	(0,T)\times\R^q\to[-M,M]
	$
	obtained from \eqref{eq:reduction-random-sections}, and
	$
	h_\alpha^{\omega,\sigma}
	=
	\1_{\{\lambda<u_\alpha^{\omega,\sigma}\}};
	$

	\item a locally finite nonnegative Radon measure
	$
	m_\alpha^{\omega,\sigma}
	$
	on
	$
	(0,T)\times\R^q\times I_\alpha
	$
	with kinetic support in $[-M,M]\cap I_\alpha$ and finite mass on
	$
	(0,T)\times K\times I'
	$
	for every $K\Subset\R^q$ and $I'\Subset I_\alpha$;

	\item the continuous scalar paths
	$
	M_{\alpha,\psi}^{\omega,\sigma}
	$
	and
	$
	Z_{\alpha,\phi}^{\omega,\sigma},
	$
	starting from zero, for every
	$
	\psi\in\mathscr D_\alpha
	$
	and
	$
	\phi\in\mathscr T_\alpha,
	$
	such that the localized time-integrated identity
	\eqref{eq:time-integrated} and the reduced kinetic identity---namely,
	\eqref{eq:kinetic-weak} for $\alpha_0$ and
	\eqref{eq:pathwise-sliced-kinetic-identity} after a reduction---hold
	with the coefficients associated with $\alpha$;

	\item the unique weak trace
	$
	h_{0,\alpha}^{\omega,\sigma}
	$
	characterized by
	\eqref{eq:pathwise-weak-trace}, together with the kinetic identity with
	the weak trace
	\eqref{eq:kinetic-with-trace}, both localized to $I_\alpha$.  The latter
	identity is required for the predetermined tensor family generated by
	$\Theta$ and $\mathscr D_\alpha$, together with the translations,
	dilations, mollifications, and fixed linear pullbacks assigned to
	$\mathfrak R$.  For test functions in $\mathscr T_\alpha$, the stochastic term
	is the corresponding scalar path from \textup{(iii)}.  For an extended
	tensor
	$
	\eta(t,x)\rho(\lambda)
	$
	and its prescribed pullbacks, it is the pairing of $\eta$ with the
	process in \textup{(v)}, as in
	\eqref{eq:pathwise-martingale-pairing};

	\item for every $\rho\in\mathcal R$ with
	$\supp\rho\Subset I_\alpha$,
	the spatial martingale path
	$
	\mathscr M_\alpha^{\rho,\omega,\sigma}
	$
	obtained from the fixed-$(\omega,\sigma)$ section of
	\eqref{eq:reduction-spatial-martingale}, and taking values
	in
	$C([0,T];L^2_{\loc}(\R^q))$
	starting from zero.  With $\mathscr M^\rho$ replaced by
	$\mathscr M_\alpha^{\rho,\omega,\sigma}$ and with the variables and
	coefficients associated with $\alpha$, it satisfies
	\eqref{eq:pathwise-martingale-pairing},
	\eqref{eq:pathwise-averaged-identity}, and
	\eqref{eq:pathwise-averaged-identity-with-trace};

	\item if $q\geq1$, then for every one-step reduction
	$\beta$
	of $\alpha$ there is a set
	$S_{\alpha\beta}^{\omega,\sigma}\subset\R$
	of full Lebesgue measure, contained in
	$\{\zeta:(\sigma,\zeta)\in\mathcal A_\beta(\omega)\}$, such that, for every
	$\zeta\in S_{\alpha\beta}^{\omega,\sigma}$, 
	Lemma~\ref{lem:pathwise-slice-equation} supplies
	$
	m_\beta^{\omega,(\sigma,\zeta)}
	:=
	\widetilde m_\zeta^{\omega,\sigma}
	$
	through \eqref{eq:def-sliced-kinetic-measure},
	$
	Z_{\beta,\phi}^{\omega,(\sigma,\zeta)}
	:=
	Z_{\phi,\zeta}^\omega
	$
	satisfying \eqref{eq:pathwise-sliced-kinetic-identity}, and
	$
	M_{\beta,\psi}^{\omega,(\sigma,\zeta)}
	:=
	M_{\psi,\zeta}^\omega
	$
	through \eqref{eq:sliced-time-martingale}.  The restriction identities
	for \eqref{eq:sliced-spatial-martingale} define
	$
	\mathscr M_\beta^{\rho,\omega,(\sigma,\zeta)}
	$
	and give the reduced forms of
	\eqref{eq:pathwise-averaged-identity} and
	\eqref{eq:pathwise-averaged-identity-with-trace}.  The proof of
	Lemma~\ref{lem:pathwise-slice-equation} also gives
	$
	h_{0,\beta}^{\omega,(\sigma,\zeta)}
	:=
	h_{0,\zeta}^\omega
	$
	and the reduced form of \eqref{eq:kinetic-with-trace}.  Finally,
	\eqref{eq:reduction-spatial-martingale}--\eqref{eq:reduction-scalar-martingales}
	supply the martingale processes for all subsequent reductions at
	$(\sigma,\zeta)$.  The resulting section is an admissible reduced
	realization associated with $\beta$.
\end{enumerate}
Item \textup{(vi)} is vacuous in dimension zero.  Since every reduction
lowers the dimension by one, the recursion terminates after at most $q$
steps.  Starting with the empty parameter at $\alpha_0$,
which belongs to $\mathcal A_{\alpha_0}(\omega)$ by
\eqref{eq:recursive-admissibility-event} for $\omega\in\Omega_*$,
Lemma~\ref{lem:pathwise-slice-equation} constructs, for each already fixed
realization, the full-measure sets of next transverse coordinates in
\textup{(vi)}.  After one reduction, items
\textup{(iii)}--\textup{(v)} are supplied by
\eqref{eq:sliced-time-martingale},
\eqref{eq:pathwise-sliced-kinetic-identity}, and
\eqref{eq:sliced-spatial-martingale}, together with the weak-trace and cutoff
arguments in Lemma~\ref{lem:pathwise-slice-equation}.  No measurable family
of the section measures
$
m_\alpha^{\omega,\sigma}
$
is selected or integrated over $\sigma$.
\end{definition}

The localized blow-up results are stable within this class.  More precisely,
the proofs of Lemma~\ref{lem:pathwise-blowup-subsequence},
Proposition~\ref{prop:pathwise-condition-C}, and
Proposition~\ref{prop:pathwise-weak-blowup-limit} use only items
\textup{(ii)}, \textup{(iv)}, and \textup{(v)} of
Definition~\ref{def:pathwise-kinetic-realization}.  They therefore apply
verbatim for every reduction datum $\alpha\in\mathfrak R$, with the reduced
dimension $q=d_\alpha$ in place of $d$, the
coefficients associated with $\alpha$ in place of the original
coefficients, and every kinetic cutoff in those proofs supported compactly
in $I_\alpha$.

We can now apply the lower-dimensional trace statement at almost every
transverse coordinate.

\begin{proposition}[Pathwise strong trace on a degenerate interval]
\label{prop:pathwise-degenerate-truncation}
Fix a realization $\omega\in\Omega_*$, where $\Omega_*$ is defined in
\eqref{eq:common-pathwise-event}.  Let $\alpha\in\mathfrak R$ be a reduction
datum of spatial dimension $q=d_\alpha\geq1$, let
$
u=u_\alpha^{\omega,\sigma}
$
be an admissible reduced realization associated with $\alpha$ in the sense of
Definition~\ref{def:pathwise-kinetic-realization}
(so $\sigma\in\mathcal A_\alpha(\omega)$; the probability realization
$\omega$ and accumulated parameter $\sigma$ are fixed throughout), and let
$\beta$ be one of its one-step
reductions, with $d_\beta=q-1$.  Write
$
f:=f_\alpha.
$
Assume
Proposition~\ref{prop:pathwise-induction} at $\beta$.  Let
$(I,\xi_0,\xi,A,c)$ be the fixed data for the one-step reduction
$\alpha\to\beta$, so that
$(\xi_0,\xi)\in\R\times\R^q\setminus\{0\}$ and
\begin{equation*}
	\xi_0+\xi\cdot f'(\lambda)=0
	\qquad
	\text{for every }\lambda\in I
\end{equation*}
on an open interval $I\subset(-L,L)$.
Let $[a,b]\Subset I$ be the truncation interval attached to this reduction datum in
$\mathfrak R$. Then the truncation $T_{a,b}(u)$ has a strong initial
trace in $L^1_{\loc}(\R^q)$.
\end{proposition}

\begin{proof}
Write $(\omega,\sigma)$ for the fixed realization and accumulated parameter
underlying $u$.  The corresponding transform
\eqref{eq:random-reduction-transform} is
$$
\widetilde u(t,w)
:=
u_\alpha^{\omega,\sigma}
\bigl(
t,A^{-1}(w-ct)
\bigr).
$$
By the corresponding version of
\eqref{eq:transformed-weak-trace}, its weak trace is
$$
\widetilde h_{0,\alpha}^{\omega,\sigma}(w,\lambda)
=
h_{0,\alpha}^{\omega,\sigma}(A^{-1}w,\lambda).
$$
For the rest of the proof, suppress $\sigma$ and the reduction index, and write
$
\widetilde h_0^\omega
:=
\widetilde h_{0,\alpha}^{\omega,\sigma}.
$

By item \textup{(vi)} of
Definition~\ref{def:pathwise-kinetic-realization} and
Lemma~\ref{lem:pathwise-slice-equation}, for every $\zeta$ in a set of full
Lebesgue measure, the section
$\widetilde u_\zeta(t,w')=\widetilde u(t,w',\zeta)$
is an admissible reduced realization associated with $\beta$, with identity
\eqref{eq:pathwise-sliced-kinetic-identity}, martingale processes
\eqref{eq:sliced-scalar-martingale}--\eqref{eq:sliced-spatial-martingale},
and all processes required by subsequent reductions in
\eqref{eq:reduction-spatial-martingale}--\eqref{eq:reduction-scalar-martingales}.
The induction hypothesis therefore gives a strong initial trace for
$T_{a,b}(\widetilde u_\zeta)$.

We now identify the weak trace of the lower-dimensional equation
\eqref{eq:pathwise-sliced-kinetic-identity} at coordinate $\zeta$ with the
corresponding section of $\widetilde h_0^\omega$, without selecting these
lower-dimensional weak traces as a measurable family.

Let $\cD\subset C_c^1(\R^{q-1}\times I)$ be countable and dense in
$L^1$ on every compact cylinder. For
$\phi\in\cD$ and all sufficiently large $n$, define
$$
A_{n,\phi}(\zeta)
:=
n\int_0^{1/n}
\int_{\R^{q-1}}\int_I
\widetilde h(t,w',\zeta,\lambda)
\phi(w',\lambda)
\,d\lambda\,dw'\,dt.
$$
For the fixed realization, $A_{n,\phi}$ is a measurable scalar function of
$\zeta$.  Define
$$
L_\phi(\zeta)
:=
\lim_{n\to\infty}A_{n,\phi}(\zeta)
$$
where the limit exists, and set $L_\phi(\zeta)=0$ otherwise.

For every $\zeta\in S_{\omega,\sigma}$ outside one common null set, the weak trace
constructed from \eqref{eq:pathwise-sliced-kinetic-identity} gives
\begin{equation}\label{eq:sliced-pairing-limit}
	L_\phi(\zeta)
	=
	\int_{\R^{q-1}}\int_I
	h_{0,\zeta}^\omega(w',\lambda)
	\phi(w',\lambda)
	\,d\lambda\,dw',
\end{equation}
where $h_{0,\zeta}^\omega$ denotes the uniquely determined weak trace of the
lower-dimensional equation at coordinate $\zeta$.  No measurable choice of the
family
$$
\zeta\longmapsto h_{0,\zeta}^\omega
$$
is made.

Let $\chi\in C_c^1(\R)$.
Since
$$
\abs{A_{n,\phi}(\zeta)}
\leq
\norm{\phi}_{L^1(\R^{q-1}\times I)},
$$
dominated convergence gives
$$
\int_{\R}
\chi(\zeta)L_\phi(\zeta)
\,d\zeta
=
\lim_{n\to\infty}
\int_{\R}
\chi(\zeta)A_{n,\phi}(\zeta)
\,d\zeta.
$$
The right-hand side equals
\begin{align*}
	\lim_{n\to\infty}
	n\int_0^{1/n}
	\int_{\R^q}\int_I
	\widetilde h(t,w,\lambda)
	\phi(w',\lambda)\chi(w_q)
	\,d\lambda\,dw\,dt.
\end{align*}
Equations~\eqref{eq:transformed-weak-trace} and
\eqref{eq:pathwise-weak-trace} identify this limit as
$$
\int_{\R}
\chi(\zeta)
\int_{\R^{q-1}}\int_I
\widetilde h_0^\omega(w',\zeta,\lambda)
\phi(w',\lambda)
\,d\lambda\,dw'\,d\zeta.
$$
Since $\chi$ is arbitrary,
\begin{equation}\label{eq:transformed-slice-pairing}
L_\phi(\zeta)
=
\int_{\R^{q-1}}\int_I
\widetilde h_0^\omega(w',\zeta,\lambda)
\phi(w',\lambda)
\,d\lambda\,dw'
\end{equation}
for almost every $\zeta\in\R$.

Taking a countable intersection over $\phi\in\cD$ and comparing
\eqref{eq:transformed-slice-pairing} with
\eqref{eq:sliced-pairing-limit}, density gives
\begin{equation}\label{eq:pathwise-slice-trace-identification}
	h_{0,\zeta}^\omega(w',\lambda)
	=
	\widetilde h_0^\omega(w',\zeta,\lambda)
	\quad\text{for almost every }(w',\lambda)
\end{equation}
for almost every admissible $\zeta\in\R$.

Define
\begin{equation*}
	\widetilde g_0^{a,b}(w,\lambda)
	:=
	\1_{\{\lambda<a\}}
	+
	\1_{\{a<\lambda<b\}}
	\widetilde h_0^\omega(w,\lambda).
\end{equation*}
For almost every $(t,w,\lambda)$,
$$
\1_{\{\lambda<T_{a,b}(\widetilde u(t,w))\}}
=
\1_{\{\lambda<a\}}
+
\1_{\{a<\lambda<b\}}
\widetilde h(t,w,\lambda).
$$
By \eqref{eq:transformed-weak-trace} and the density extension of
\eqref{eq:pathwise-weak-trace} from smooth test functions to all $L^1$ test functions on
compact cylinders,
$\widetilde g_0^{a,b}$ is the pathwise weak trace of
$\1_{\{\lambda<T_{a,b}(\widetilde u)\}}$.

By \eqref{eq:pathwise-slice-trace-identification} and the induction
hypothesis, for almost every $\zeta\in\R$ the section
$\widetilde g_0^{a,b}(\cdot,\zeta,\cdot)$ is a subgraph function. Define
$$
\widetilde v_0(w)
:=
\int_{-L}^L
\widetilde g_0^{a,b}(w,\lambda)
\,d\lambda
-
L.
$$
Then $\widetilde v_0$ is measurable, takes values in $[a,b]$, and
$\widetilde g_0^{a,b}(w,\lambda)
=\1_{\{\lambda<\widetilde v_0(w)\}}$ for almost every $(w,\lambda)$.
Corollary~\ref{cor:essential-chi-rigidity} therefore gives, for every compact
$\widetilde K\Subset\R^q$,
\begin{equation}\label{eq:transformed-truncated-strong-trace}
\esslim_{t\downarrow0}
\int_{\widetilde K}
\abs{
T_{a,b}(\widetilde u(t,w))
-
\widetilde v_0(w)
}
\,dw
=
0.
\end{equation}

Finally, local $L^1$ convergence is preserved by the invertible affine
change of variables. Define $v_0(x):=\widetilde v_0(Ax)$. Since
$T_{a,b}(u(t,x))=T_{a,b}(\widetilde u(t,Ax+ct))$,
\eqref{eq:transformed-truncated-strong-trace} and translation continuity of
$\widetilde v_0$ in $L^1_{\loc}$ give
$$
\esslim_{t\downarrow0}
\int_K
\abs{
T_{a,b}(u(t,x))
-
v_0(x)
}
\,dx
=
0
$$
for every compact $K\Subset\R^q$.

\end{proof}

%%%%%%%%%%%%%%%%%%%%%%%%%
%%%%%%%%%%%%%%%%%%%%%%%%%
\section{Identification of the pathwise weak trace}
\label{sec:pathwise-identification}

Fix a realization $\omega\in\Omega_*$, where $\Omega_*$ is defined in
\eqref{eq:common-pathwise-event}.  The induction below uses
\eqref{eq:kinetic-with-trace}, the weak convergence
\eqref{eq:pathwise-weak-trace}, and the local finiteness following from
\eqref{eq:local-mass-event}, namely
$
m(\omega;(0,T)\times K\times[-R,R])<\infty
$
for every compact $K\Subset\R^d$ and every $R>0$.  It also uses the
martingale processes
\eqref{eq:spatial-martingale-process} and
\eqref{eq:reduction-spatial-martingale}.  For each reduction datum, item
\textup{(vi)} of
Definition~\ref{def:pathwise-kinetic-realization} supplies a full-measure
set of next transverse coordinates and all processes needed in subsequent
reductions.
The induction is applied separately at each such coordinate through
\eqref{eq:pathwise-sliced-kinetic-identity}; no section measure or weak
trace is integrated with respect to a transverse coordinate.  The only
transverse-coordinate functions integrated below are the explicit scalar
averages $A_{n,\phi}$.

Let $I\subset(-L,L)$ be open and let $[a,b]\Subset I$.  The kinetic
function of the truncation $T_{a,b}(u)$ is, up to the irrelevant endpoint
levels,
\begin{equation*}
	h^{a,b}(t,x,\lambda)
	:=
	\1_{\{\lambda<T_{a,b}(u(t,x))\}}
	=
	\1_{\{\lambda<a\}}
	+
	\1_{\{a<\lambda<b\}}h(t,x,\lambda).
\end{equation*}
Correspondingly, define
\begin{equation}\label{eq:truncated-pathwise-weak-trace}
	h_0^{\omega,a,b}(x,\lambda)
	:=
	\1_{\{\lambda<a\}}
	+
	\1_{\{a<\lambda<b\}}h_0^\omega(x,\lambda).
\end{equation}
Approximating $\1_{(a,b)}$ in $L^1(\R_\lambda)$ by smooth kinetic
cutoffs and using the uniform bound $0\leq h\leq1$, we see that
$h_0^{\omega,a,b}$ is the pathwise weak trace of $h^{a,b}$.

The next proposition states the induction assertion for these truncated
kinetic functions.

\begin{proposition}[Pathwise induction]
\label{prop:pathwise-induction}
Let $\alpha\in\mathfrak R$ be a reduction datum of spatial dimension
$q=d_\alpha\geq0$, with kinetic interval $I=I_\alpha$ and attached
truncation $[a,b]=[a_\alpha,b_\alpha]\Subset I$.  For the fixed
$\omega\in\Omega_*$, with $\Omega_*$ defined in
\eqref{eq:common-pathwise-event}, and a fixed recursively admissible
accumulated transverse parameter
$\sigma\in\mathcal A_\alpha(\omega)$,
let $u=u_\alpha^{\omega,\sigma}$ be an admissible reduced realization
associated with $\alpha$ in the sense of
Definition~\ref{def:pathwise-kinetic-realization}.  In this statement and
its proof, write
$
f:=f_\alpha
$
and suppress the fixed accumulated parameter $\sigma$ from
$u$, $h_0^\omega$, and the reduced coefficients.  Then there exists a
measurable function $u_0^{\omega,a,b}:\R^q\to[a,b]$ such that
$$
h_0^{\omega,a,b}(x,\lambda)
=
\1_{\{\lambda<u_0^{\omega,a,b}(x)\}}
$$
for almost every $(x,\lambda)$.  Consequently, for every compact
$K\Subset\R^q$,
$$
\esslim_{t\downarrow0}
\int_K
\abs{
T_{a,b}(u(t,x))
-
u_0^{\omega,a,b}(x)
}
\,dx
=
0.
$$
\end{proposition}

\begin{proof}
We argue by induction on the spatial dimension.

\emph{Base case $q=0$.}
Regard $\R^0=\{\varnothing\}$ as the one-point measure space.
Let $E_\omega\subset(0,T)$ be a full-measure set on which the pairings in
\eqref{eq:pathwise-weak-trace} are represented, and take any
$t_n\in E_\omega$ with $t_n\downarrow0$. Since
$T_{a,b}(u(t_n))\in[a,b]$,
there is a subsequence, not relabeled, and a number $r\in[a,b]$ such that
$T_{a,b}(u(t_n))\to r$.
The corresponding kinetic functions satisfy
$$
\1_{\{\lambda<T_{a,b}(u(t_n))\}}
\longrightarrow
\1_{\{\lambda<r\}}
\quad\text{strongly in }L^1(-L,L).
$$
On the other hand, every such subsequence has the same weak limit, namely the
unique pathwise weak trace $h_0^{\omega,a,b}$.  Hence
$h_0^{\omega,a,b}(\lambda)=\1_{\{\lambda<r\}}$ almost everywhere.
Corollary~\ref{cor:essential-chi-rigidity} then gives the
strong trace.

\emph{Inductive step.}
Assume Proposition~\ref{prop:pathwise-induction} in spatial dimension
$q-1$.  Fix a reduction datum $\alpha\in\mathfrak R$ with
$d_\alpha=q$, $I_\alpha=I$, and
$[a_\alpha,b_\alpha]=[a,b]\Subset I$, and fix an admissible reduced
realization associated with $\alpha$.

Let $F_I\subset I$ be the countable dense set fixed for $\alpha$. Let
$\mathscr I_{a,b}$ be the countable family
of all intervals
$$
[c,e]\Subset(a,b),
\qquad
c,e\in F_I,
\qquad
c<e,
$$
for which one can find an open interval
$$
D\Subset I,
\qquad
[c,e]\Subset D,
$$
and a vector $\eta\in\R^{q+1}\setminus\{0\}$ such that
\begin{equation*}
	\eta\cdot\Phi(r)
	=
	\mathrm{const}
	\qquad
	\text{for every }r\in D,
\end{equation*}
where $\Phi(r)=(r,f(r))$.

For every $[c,e]\in\mathscr I_{a,b}$, the construction of
$\mathfrak R$ fixes an interval $D$ and a nonzero vector $\eta$ such that
$\eta\cdot\Phi$ is constant on $D$, together with the corresponding reduction
data.  The induction hypothesis and
Proposition~\ref{prop:pathwise-degenerate-truncation} therefore show that
$T_{c,e}(u)$ has a pathwise strong initial trace.

Fix a sequence
$
\varepsilon_n\downarrow0.
$
Apply Lemma~\ref{lem:pathwise-blowup-subsequence} and then
Lemma~\ref{lem:pathwise-trace-to-blowup}, diagonally over the countable family
$$
\left\{
T_{c,e}(u):
[c,e]\in\mathscr I_{a,b}
\right\}.
$$
After passing to a common subsequence, there exists a set
$$
Y_\omega\subset\R^q,
\qquad
\mathcal L^q(\R^q\setminus Y_\omega)=0,
$$
such that \eqref{eq:pathwise-defect-vanish},
\eqref{eq:pathwise-martingale-vanish}, and
\eqref{eq:pathwise-trace-lebesgue} hold, and such that
\eqref{eq:pathwise-strong-trace-blowup} holds for every
$T_{c,e}(u)$ with $[c,e]\in\mathscr I_{a,b}$.

Fix $y\in Y_\omega$, and define
$v_n(\tau,z):=T_{a,b}(u^{\varepsilon_n,y}(\tau,z))$.
Proposition~\ref{prop:pathwise-condition-C} shows that $(v_n)$ satisfies
Panov's compactness condition \eqref{eq:panov-C}.

If $[c,e]\in\mathscr I_{a,b}$, then
$T_{c,e}(v_n)=T_{c,e}(u^{\varepsilon_n,y})$,
because $[c,e]\Subset(a,b)$.
Lemma~\ref{lem:pathwise-trace-to-blowup} therefore gives strong local
convergence, after the common extraction, for every one of these truncated
sequences. Lemma~\ref{lem:panov-closure}, applied with $K_0=[a,b]$, now
implies that $(v_n)$ is strongly precompact in
$L^1_{\loc}
\bigl(
(0,\infty)\times\R^q
\bigr)$.

Passing to a further subsequence, we may assume that
$$
v_n\longrightarrow\overline v
\quad\text{strongly in }
L^1_{\loc}
\bigl(
(0,\infty)\times\R^q
\bigr).
$$
Since $v_n\in[a,b]$, we also have
$\overline v\in[a,b]$ almost everywhere.

After another extraction, the kinetic functions $h^{\varepsilon_n,y}$
converge weak-$*$ locally.
Proposition~\ref{prop:pathwise-weak-blowup-limit} identifies their weak limit
on $I$ as $h_0^\omega(y,\lambda)$, which is independent of $(\tau,z)$.

For every $\lambda\in(a,b)$,
$\1_{\{\lambda<v_n(\tau,z)\}}
=h^{\varepsilon_n,y}(\tau,z,\lambda)$.
Moreover, on every bounded cylinder $Q\Subset(0,\infty)\times\R^q$,
$$
\int_Q\int_a^b
\abs{
\1_{\{\lambda<v_n(\tau,z)\}}
-
\1_{\{\lambda<\overline v(\tau,z)\}}
}
\,d\lambda\,dz\,d\tau
=
\int_Q
\abs{v_n-\overline v}
\,dz\,d\tau.
$$
Thus the strong convergence of $v_n$ gives
$$
\1_{\{\lambda<v_n\}}
\longrightarrow
\1_{\{\lambda<\overline v\}}
\quad\text{strongly in }
L^1_{\loc}
\bigl(
(0,\infty)\times\R^q\times(a,b)
\bigr).
$$
Comparing this strong limit with the weak-$*$ limit of
$h^{\varepsilon_n,y}$ identified in
Proposition~\ref{prop:pathwise-weak-blowup-limit} gives
\begin{equation}\label{eq:pointwise-truncated-equilibrium}
	h_0^\omega(y,\lambda)
	=
	\1_{\{\lambda<\overline v(\tau,z)\}}
\end{equation}
for almost every
$(\tau,z,\lambda)\in(0,\infty)\times\R^q\times(a,b)$.

Since the left-hand side of
\eqref{eq:pointwise-truncated-equilibrium} is independent of $(\tau,z)$ and
the map
$$
r
\longmapsto
\bigl(
\lambda\longmapsto\1_{\{\lambda<r\}}
\bigr)
$$
is injective from $[a,b]$ into $L^1(a,b)$,
\eqref{eq:pointwise-truncated-equilibrium} implies that
$\overline v$ is constant in $(\tau,z)$.  Hence there exists a number
$r_{a,b}^\omega(y)\in[a,b]$ such that
\begin{equation}\label{eq:pointwise-trace-equilibrium}
h_0^\omega(y,\lambda)
=
\1_{\{\lambda<r_{a,b}^\omega(y)\}}
\quad\text{for almost every }\lambda\in(a,b).
\end{equation}
Thus, for almost every $y\in\R^q$, there is
$r_{a,b}^\omega(y)\in[a,b]$ satisfying
\eqref{eq:pointwise-trace-equilibrium}.

Define
$$
u_0^{\omega,a,b}(x)
:=
\int_{-L}^L
h_0^{\omega,a,b}(x,\lambda)
\,d\lambda
-
L.
$$
The function $u_0^{\omega,a,b}$ is measurable and takes values in
$[a,b]$.  Equation~\eqref{eq:pointwise-trace-equilibrium} and the
definition \eqref{eq:truncated-pathwise-weak-trace} give
$$
h_0^{\omega,a,b}(x,\lambda)
=
\1_{\{\lambda<u_0^{\omega,a,b}(x)\}}
$$
for almost every $(x,\lambda)$.

Corollary~\ref{cor:essential-chi-rigidity}, applied to every compact spatial
set, now yields
$$
\esslim_{t\downarrow0}
\int_K
\abs{
T_{a,b}(u(t,x))
-
u_0^{\omega,a,b}(x)
}
\,dx
=
0.
$$
This completes the induction.
\end{proof}

For the original datum $\alpha_0\in\mathfrak R$, take
$I=(-L,L)$, $a=-M$, and $b=M$.
For every $\omega\in\Omega_*$, with $\Omega_*$ defined in
\eqref{eq:common-pathwise-event}, one has
$\varnothing\in\mathcal A_{\alpha_0}(\omega)$ by
\eqref{eq:recursive-admissibility-event}.  Equations
\eqref{eq:kinetic-with-trace}, \eqref{eq:pathwise-weak-trace}, and
\eqref{eq:reduction-spatial-martingale}--\eqref{eq:reduction-scalar-martingales},
together with the recursive conclusion of
Lemma~\ref{lem:pathwise-slice-equation}, show that the original solution is an
admissible reduced realization associated with $\alpha_0$.

Since $\abs{u}\leq M$, we have $T_{-M,M}(u)=u$.
For $\phi\in C_c^1(\R^d)$ and
$\rho\in C_c^1((-L,-M))$, the convergence
\eqref{eq:pathwise-weak-trace} and the identity
$h(t,x,\lambda)=1$ for $\lambda\in(-L,-M)$ give
\begin{equation*}
	\int_{\R^d}\int_{\R}
	h_0^\omega(x,\lambda)\phi(x)\rho(\lambda)
	\,d\lambda\,dx
	=
	\left(
	\int_{\R^d}\phi(x)\,dx
	\right)
	\left(
	\int_{\R}\rho(\lambda)\,d\lambda
	\right).
\end{equation*}
Hence $h_0^\omega=1$ almost everywhere on
$\R^d\times(-L,-M)$. Testing instead with
$\rho\in C_c^1((M,L))$ gives $h_0^\omega=0$ almost everywhere on
$\R^d\times(M,L)$. Together with
\eqref{eq:truncated-pathwise-weak-trace}, these identities yield
$h_0^{\omega,-M,M}=h_0^\omega$ almost everywhere.

Proposition~\ref{prop:pathwise-induction} therefore identifies the full
pathwise weak trace as
\begin{equation*}
	h_0^\omega(x,\lambda)
	=
	\1_{\{\lambda<u_0^\omega(x)\}}
\end{equation*}
for a uniquely determined measurable function
$u_0^\omega:\R^d\to[-M,M]$.
It is given by
\begin{equation*}
	u_0^\omega(x)
	=
	\int_{-L}^L
	h_0^\omega(x,\lambda)
	\,d\lambda
	-
	L.
\end{equation*}
Corollary~\ref{cor:essential-chi-rigidity} gives, for every compact
$K\Subset\R^d$,
\begin{equation}\label{eq:pathwise-final-strong-trace}
	\esslim_{t\downarrow0}
	\int_K
	\abs{u(\omega,t,x)-u_0^\omega(x)}
	\,dx
	=
	0.
\end{equation}
This completes the pathwise argument, without requiring measurability of
$\omega\mapsto u_0^\omega$.

%%%%%%%%%%%%%%%%%%%%%%%%%
%%%%%%%%%%%%%%%%%%%%%%%%%
\section{Measurability and the strong trace in expectation}
\label{sec:final-measurability}

We recover joint and $\cF_0$-measurability from time averages and then pass
from pathwise convergence to convergence in expectation.

\begin{lemma}[Measurable representative obtained from time averages]
\label{lem:final-measurability}
The pathwise trace $u_0^\omega$ admits an
$\cF_0\otimes\cB(\R^d)$-measurable representative
$u_0:\Omega\times\R^d\to[-M,M]$.
For almost every $\omega$,
$u_0(\omega,\cdot)=u_0^\omega$ in $L^1_{\loc}(\R^d)$.
Choose an integer $n_0$ such that $1/n_0<T$.  For $n\geq n_0$, set
$$
u_n(\omega,x)
:=
n\int_0^{1/n}
u(\omega,t,x)
\,dt.
$$
Then, for every compact $K\Subset\R^d$,
$$
\E\int_K
\abs{u_n(\omega,x)-u_0(\omega,x)}
\,dx
\longrightarrow0.
$$
\end{lemma}

\begin{proof}
By the representative fixed after Definition~\ref{def:kinetic-solution},
$u$ is predictable and satisfies $\abs{u}\leq M$ everywhere.  Hence
$\abs{u_n}\leq M$.  Predictability is used below to prove the joint
$\cF_{1/n}\otimes\cB(\R^d)$-measurability of $u_n$, while the uniform bound
gives local integrability and the domination needed for the final
convergence in expectation.

Fix the integer $n_0$ chosen in the statement.
For $n\geq n_0$, consider the time average $u_n$ defined in the statement.
The predictable sigma-field restricted to
$\Omega\times(0,1/n)$ is contained in
$\cF_{1/n}\otimes\cB((0,1/n))$.  Parameterized integration (integration with
respect to $t$, with $(\omega,x)$ retained as the parameter) therefore shows
that
\begin{equation}\label{eq:time-average-measurability}
	(\omega,x)
	\longmapsto
	u_n(\omega,x)
\end{equation}
is $\cF_{1/n}\otimes\cB(\R^d)$-measurable.  Moreover,
$\abs{u_n}\leq M$.

Fix a compact set $K\Subset\R^d$.  For every realization for which
\eqref{eq:pathwise-final-strong-trace} holds,
\begin{align*}
	\int_K
	\abs{u_n(\omega,x)-u_0^\omega(x)}
	\,dx
	&\leq
	n\int_0^{1/n}
	\int_K
	\abs{u(\omega,t,x)-u_0^\omega(x)}
	\,dx\,dt
	\\
	&\leq
	\operatorname*{ess\,sup}_{0<t<1/n}
	\int_K
	\abs{u(\omega,t,x)-u_0^\omega(x)}
	\,dx.
\end{align*}
The right-hand side tends to zero by
\eqref{eq:pathwise-final-strong-trace}.  Hence
$u_n(\omega,\cdot)\to u_0^\omega$ in $L^1_{\loc}(\R^d)$
for almost every $\omega$.

Equip $L^1_{\loc}(\R^d)$ with a standard complete separable metric, for
example
$$
d(v,w)
:=
\sum_{j=1}^\infty
2^{-j}
\min
\left\{
1,
\norm{v-w}_{L^1(B_j)}
\right\}.
$$
We verify the corresponding function-space measurability.  For every
$j\geq1$ and $v\in L^1(B_j)$, Tonelli's theorem and
\eqref{eq:time-average-measurability} imply measurability of
\begin{equation*}
	\omega
	\longmapsto
	\norm{
	u_n(\omega,\cdot)-v
	}_{L^1(B_j)}.
\end{equation*}
Open balls centered at a countable dense subset generate the Borel
sigma-field of $L^1(B_j)$.  Hence $\omega
\longmapsto u_n(\omega,\cdot)\big|_{B_j}$ 
is strongly measurable.  Since the Borel sigma-field associated with the
metric $d$ is generated by these restriction maps,
\begin{equation}\label{eq:L1loc-time-average-measurability}
	\omega
	\longmapsto
	u_n(\omega,\cdot)
\end{equation}
is an $\cF_{1/n}$-measurable
$L^1_{\loc}(\R^d)$-valued map.

Let $C$ be the measurable set on which
$(u_n(\omega,\cdot))$ is Cauchy in the metric $d$, and define
\begin{equation}\label{eq:measurable-trace-definition}
	u_0(\omega,\cdot)
	:=
	\begin{cases}
		\displaystyle
		\lim_{n\to\infty}u_n(\omega,\cdot),
		&\omega\in C,
		\\
		0,
		&\omega\notin C.
	\end{cases}
\end{equation}
Completeness of $L^1_{\loc}(\R^d)$ and
\eqref{eq:L1loc-time-average-measurability} show that this is a
measurable map
$\omega\mapsto u_0(\omega,\cdot)\in L^1_{\loc}(\R^d)$
which agrees with $u_0^\omega$ for almost every $\omega$.

The subset
$$
\left\{
v\in L^1_{\loc}(\R^d):
\abs{v}\leq M
\text{ almost everywhere}
\right\}
$$
is closed in $L^1_{\loc}(\R^d)$.  Since every $u_n$ belongs to this set, the
limit in \eqref{eq:measurable-trace-definition} satisfies
$\abs{u_0}\leq M$ almost everywhere.

It remains to identify the initial sigma-field.  Fix $m\geq n_0$.  For
every $n\geq m$, \eqref{eq:L1loc-time-average-measurability} and
$\cF_{1/n}\subset\cF_{1/m}$ for $n\geq m$,
show that the tail $(u_n)_{n\geq m}$, its Cauchy set $C$, and the
deterministically defined limit \eqref{eq:measurable-trace-definition} are
$\cF_{1/m}$-measurable. Therefore $u_0$ is measurable with respect to
$\bigcap_{m\geq n_0}\cF_{1/m}$, which equals $\cF_0$ by right-continuity
of the filtration.

We next turn the $L^1_{\loc}(\R^d)$-valued random variable $u_0$ into one
jointly measurable function of $(\omega,x)$.  Recall that
$B_j:=\{x\in\R^d:\abs{x}<j\}$ is the open ball of radius $j$ centered at
the origin; the nested balls $(B_j)_{j\geq1}$ have finite measure and exhaust
$\R^d$.  We work locally because $u_0$ need not belong to $L^1(\R^d)$,
whereas its restriction to every $B_j$ is integrable.
For $j\geq1$, set
$U_j(\omega):=u_0(\omega,\cdot)|_{B_j}\in L^1(B_j)$.
The map $U_j$ is strongly $\cF_0$-measurable, and the bound above gives
\begin{equation*}
	\int_\Omega
	\norm{U_j(\omega)}_{L^1(B_j)}
	\,d\Prob(\omega)
	\leq
	M\abs{B_j}.
\end{equation*}
Thus $U_j\in L^1(\Omega,\cF_0,\Prob;L^1(B_j))$.
By the Bochner--Fubini identification and its jointly measurable
representative formulation
\cite[Propositions~1.2.24 and~1.2.25]{Hytonen:2016aa}, there exists an
$\cF_0\otimes\cB(B_j)$-measurable function $\widetilde u_j$ such that
\begin{equation}\label{eq:jointly-measurable-local-representative}
	\widetilde u_j(\omega,\cdot)
	=
	U_j(\omega)
	\quad\text{in }L^1(B_j)
\end{equation}
for almost every $\omega$.  These local representatives describe
restrictions of the same $L^1_{\loc}$-valued map, but they need not agree
pointwise on their overlapping domains.  To obtain one function without
making additional pointwise choices, use the disjoint annuli
$B_j\setminus B_{j-1}$, which cover $\R^d$.  With
$B_0:=\varnothing$, define
\begin{equation*}
	\widetilde u_0(\omega,x)
	:=
	\widetilde u_j(\omega,x)
	\quad\text{for }
	x\in B_j\setminus B_{j-1}.
\end{equation*}
The disjointness makes this definition unambiguous.  For each $k$,
\eqref{eq:jointly-measurable-local-representative} and the identities
$U_j=U_k|_{B_j}$ for $j\leq k$ show that the pasted function represents
$U_k$ on $B_k$ for almost every $\omega$.  The countability of the pasting
therefore shows that
$\widetilde u_0$ is $\cF_0\otimes\cB(\R^d)$-measurable
and represents the $L^1_{\loc}(\R^d)$-valued map $u_0$ for almost every
$\omega$.  Rename $\widetilde u_0$ as $u_0$ and replace it by
$
T_{-M,M}(u_0).
$
The replacement preserves its $L^1_{\loc}$ equivalence class and
$\cF_0\otimes\cB(\R^d)$-measurability, and it now takes values in
$[-M,M]$ at every point.

Finally, for every compact $K\Subset\R^d$,
$$
\int_K
\abs{u_n(\omega,x)-u_0(\omega,x)}
\,dx
\longrightarrow0
$$
for almost every $\omega$, and
$$
0
\leq
\int_K
\abs{u_n-u_0}
\,dx
\leq
2M\abs{K}.
$$
Dominated convergence gives
$$
\E\int_K
\abs{u_n-u_0}
\,dx
\longrightarrow0.
$$
\end{proof}

Dominated convergence now upgrades the pathwise result
\eqref{eq:pathwise-final-strong-trace} to the product-space conclusion.

\begin{corollary}[Strong trace in expectation]
\label{cor:strong-trace-expectation}
For every compact set $K\Subset\R^d$,
$$
\esslim_{t\downarrow0}
\E\int_K
\abs{u(\omega,t,x)-u_0(\omega,x)}
\,dx=0.
$$
Equivalently,
$u(t)\to u_0$
strongly in $L^1(\Omega\times K)$, in the essential-time sense.
\end{corollary}

\begin{proof}
For a compact set $K\Subset\R^d$, define
$F_K(t,\omega):=\int_K\abs{u(\omega,t,x)-u_0(\omega,x)}\,dx$.
The joint measurability of $u$ and $u_0$ implies that
$(t,\omega)\mapsto F_K(t,\omega)$ is measurable. Moreover,
$0\leq F_K(t,\omega)\leq2M\abs{K}$.

For $\delta>0$, set
$$
G_{K,\delta}(\omega)
:=
\operatorname*{ess\,sup}_{0<t<\delta}
F_K(t,\omega).
$$
The essential supremum is measurable because its superlevel sets can be
expressed in terms of the Lebesgue measure of time-superlevel sets.  Indeed,
for every $a\in\R$,
$$
\{G_{K,\delta}>a\}
=
\left\{
\omega:
\int_0^\delta
\1_{\{F_K(t,\omega)>a\}}
\,dt
>0
\right\}.
$$
The indicator in the last integral is jointly measurable in $(t,\omega)$
because $F_K$ is jointly measurable.  Tonelli's theorem therefore makes the
integral a measurable function of $\omega$, so the displayed set belongs to
$\cF$.  Hence all superlevel sets of $G_{K,\delta}$ are measurable, and
$G_{K,\delta}$ is a random variable.  This permits the expectation and the
dominated-convergence argument below.
Since
$u_0(\omega,\cdot)=u_0^\omega$ in $L^1_{\loc}(\R^d)$
for almost every $\omega$, the pathwise trace
\eqref{eq:pathwise-final-strong-trace} gives
$G_{K,\delta}(\omega)\to0$ as $\delta\downarrow0$
for almost every $\omega$.  Also,
$0\leq G_{K,\delta}\leq2M\abs{K}$. Dominated convergence therefore
yields $\E G_{K,\delta}\to0$.

For almost every $t\in(0,\delta)$,
$F_K(t,\omega)\leq G_{K,\delta}(\omega)$
for almost every $\omega$.  Consequently,
$$
\operatorname*{ess\,sup}_{0<t<\delta}
\E F_K(t)
\leq
\E G_{K,\delta}.
$$
Letting $\delta\downarrow0$, we obtain
$$
\esslim_{t\downarrow0}
\E F_K(t)
=
0,
$$
which is the asserted strong trace in expectation.
\end{proof}

\begin{proof}[Completion of the proof of Theorem~\ref{thm:main}]
Lemma~\ref{lem:final-measurability} gives an
$\cF_0\otimes\cB(\R^d)$-measurable representative $u_0$ and a
full-probability set $\Omega_0\subset\Omega_*$ on which
$u_0(\omega,\cdot)=u_0^\omega$ in $L^1_{\loc}(\R^d)$.  Equation
\eqref{eq:pathwise-final-strong-trace} gives the pathwise strong trace for
every $\omega\in\Omega_0$, and
Corollary~\ref{cor:strong-trace-expectation} gives convergence in expectation.

It remains to prove uniqueness. Let $v_0$ be another
$\cF_0\otimes\cB(\R^d)$-measurable function satisfying the strong-trace
conclusion of Corollary~\ref{cor:strong-trace-expectation}, with $v_0$ in
place of $u_0$. For every compact $K\Subset\R^d$ and almost every $t>0$,
\begin{align*}
	\E\int_K
	\abs{u_0(\omega,x)-v_0(\omega,x)}
	\,dx
	&\leq
	\E\int_K
	\abs{u(\omega,t,x)-u_0(\omega,x)}
	\,dx
	\\
	&\quad+
	\E\int_K
	\abs{u(\omega,t,x)-v_0(\omega,x)}
	\,dx.
\end{align*}
Taking the essential limit as $t\downarrow0$ gives
$\E\int_K\abs{u_0-v_0}\,dx=0$. Hence $u_0=v_0$ almost everywhere on
$\Omega\times K$. An exhaustion of $\R^d$ proves
uniqueness on $\Omega\times\R^d$.
\end{proof}

The following remark relates the recovered trace to a prescribed initial
datum when the solution is known independently to solve a Cauchy problem.

\begin{remark}
No value at $t=0$ is included in the definition of the solution.  The trace
$u_0$ is recovered from the solution itself.  Suppose, in addition, that a
prescribed datum $\overline u_0$ satisfies
$$
\esslim_{t\downarrow0}
\E\int_K
\abs{u(t,x)-\overline u_0(x)}
\,dx=0
$$
for every compact $K\Subset\R^d$, as required by the chosen Cauchy-solution
concept for \eqref{eq:spde}.  The uniqueness argument in the proof of
Theorem~\ref{thm:main} then gives
$u_0=\overline u_0$ almost everywhere on $\Omega\times\R^d$.
\end{remark}

%%%%%%%%%%%%%%%%%%%%%%%%%
%%%%%%%%%%%%%%%%%%%%%%%%%
\section*{Acknowledgment}

The research of ME is supported by the Croatian Science Foundation
(UIP-2025-02-1337) and by the European Union - NextGenerationEU 
through the National
Recovery and Resilience Plan 2021-2026 (IK IA 1.1.3. Impact4Math, 
institutional grant of
University of Zagreb, Faculty of Science). 
The research of KHK is supported by the Research Council
of Norway (351123/NASTRAN), and that of DM is partly supported by the
Austrian Science Fund (P~35508).

%%%%%%%%%%%%%%%%%%%%%%%%%
%%%%%%%%%%%%%%%%%%%%%%%%%

\end{document}